\documentclass{amsart}
\usepackage[foot]{amsaddr}
\usepackage{float, graphicx}
\usepackage[]{epsfig}
\usepackage{amsmath, amsthm, amssymb,}
\usepackage{epsfig}
\usepackage{verbatim}
\usepackage{multicol}
\usepackage{url}
\usepackage{latexsym}
\usepackage{mathrsfs}
\usepackage[colorlinks, bookmarks=true]{hyperref}
\usepackage{graphicx}
\usepackage{enumerate}
\usepackage[normalem]{ulem}
\usepackage{bm}
\usepackage{dsfont}
\usepackage{stmaryrd}
\usepackage{enumitem}
\usepackage{mathtools}

\usepackage{subcaption}
\usepackage{tikz}
\usetikzlibrary{arrows}
\usetikzlibrary{arrows.meta}
\usetikzlibrary{calc}
\definecolor{wwhhii}{rgb}{1.,1.,1.}
\definecolor{rreedd}{rgb}{1.,0.,0.}
\definecolor{uuuuuu}{rgb}{0.26666666666666666,0.26666666666666666,0.26666666666666666}
\definecolor{darkgreen}{HTML}{0d8513}

\usepackage{color}
\usepackage{xcolor}



\numberwithin{equation}{section}

\newcommand{\deltaw}{\mathfrak f_0}
\newcommand{\wtt}{\varpi_t}
\newcommand{\wwtt}{\widetilde\varpi_t}

\newcommand{\cL}{{\mathcal L}}

\newcommand{\cP}{{\mathcal P}}

\newcommand{\cT}{{\mathcal T}}

\newcommand{ \sfx }{{\mathsf x}}
\newcommand{ \sfy }{{\mathsf y}}

\newcommand{\sfY}{{\mathsf Y}}
\newcommand{ \sft }{{\mathsf t}}
\newcommand{\sfa}{{\mathsf a}}
\newcommand{ \sfb }{{\mathsf b}}
\newcommand{ \sfg }{{\mathsf g}}

\newcommand{\sfd}{{\mathsf d}}
\newcommand{\sfz}{{\mathsf z}}
\newcommand{\sfw}{{\mathsf w}}

\newcommand{\sfq}{{\mathsf q}}
\newcommand{\sfQ}{{\mathsf Q}}

\newcommand{\sfM}{{\mathsf M}}

\newcommand{ \sfh}{{\mathsf h}}
\newcommand{ \sff}{{\mathsf f}}
\newcommand{\sfA}{{\mathsf A}}
\newcommand{\sfB}{{\mathsf B}}
\newcommand{\sfD}{{\mathsf D}}

\newcommand{\sfP}{{\mathsf P}}
\newcommand{\sfH}{{\mathsf H}}

\newcommand{ \sfs }{{\mathsf s}}

\newcommand{\sfR}{{\mathsf R}}
\newcommand{ \sfr}{{\mathsf r}}

\newcommand{\fA}{{\mathfrak A}}

\newcommand{\fc}{{\mathfrak c}}
\newcommand{\fC}{{\mathfrak C}}
\newcommand{\fL}{{\mathfrak L}}

\newcommand{\fd}{{\mathfrak d}}

\newcommand{\fr}{{\mathfrak r}}

\newcommand{\ft}{{\mathfrak t}}

\newcommand{\fs}{{\mathfrak s}}

\newcommand{\fq}{{\mathfrak q}}

\newcommand{\fD}{{\mathfrak D}}
\newcommand{\fP}{{\mathfrak P}}
\newcommand{\fR}{{\mathfrak R}}
\newcommand{\fB}{{\mathfrak B}}
\newcommand{\fit}{\bar t_0}
\newcommand{\fft}{\bar t_1}

\newcommand{\rd}{{\rm d}}

\newcommand{\ri}{\mathrm{i}}

\newcommand{\bC}{{\mathbb C}}

\newcommand{\bB}{\mathbb{B}}

\newcommand{\bH}{\mathbb{H}}

\newcommand{\bR}{{\mathbb R}}

\newcommand{\al}{\alpha}

\newcommand{\Z}{\mathbb{Z}}
\newcommand{\C}{\mathbb{C}}
\newcommand{\R}{\mathbb{R}}
\newcommand{\N}{\mathbb{N}}
\newcommand{\T}{\mathbb{T}}
\newcommand{\HH}{\mathbb{H}}
\renewcommand{\i}{\mathbf{i}}
\newcommand{\im}{\mathrm{Im}\,}

\newcommand{\e}{\varepsilon}
\DeclareMathOperator{\PP}{\mathbb{P}}
\newcommand{\cal}{\mathcal}

\DeclareMathOperator{\dist}{dist}

\DeclareMathOperator{\OO}{\mathcal{O}}
\DeclareMathOperator{\oo}{o}
\DeclareMathOperator{\argmax}{argmax}

\renewcommand{\Re}{\mathop{\mathrm{Re}}}
\renewcommand{\Im}{\mathop{\mathrm{Im}}}
\renewcommand{\sim}{\asymp}

\newcommand{\del}{\partial}

\newcommand{\wt}{\widetilde}
\newcommand{\cofa}{a_1}
\newcommand{\coffa}{a}

\newcommand{\qq}[1]{[\![{#1}]\!]}

\newcommand{\beq}{\begin{equation}}
\newcommand{\eeq}{\end{equation}}

\newcommand{\Adm}{\mathrm{Adm}}
\usepackage{tikz}
\usetikzlibrary{arrows}
\usepackage{cleveref}

\DeclareMathOperator{\north}{no}
\DeclareMathOperator{\so}{so}
\DeclareMathOperator{\ea}{ea}
\DeclareMathOperator{\we}{we}

\newtheorem{thm}{Theorem}[section]
\newtheorem{prop}[thm]{Proposition}
\newtheorem{lem}[thm]{Lemma}

\theoremstyle{remark}
\newtheorem{rem}[thm]{Remark}
\theoremstyle{definition}
\newtheorem{definition}[thm]{Definition}
\newtheorem{assumption}[thm]{Assumption}

\title{Pearcey universality at cusps of polygonal lozenge tiling}

\author{Jiaoyang Huang}
\address{Department of Statistics and Data Science, University of Pennsylvania, Pennsylvania, PA, USA}
\email{huangjy@wharton.upenn.edu}

\author{Fan Yang}
\address{Yau Mathematical Sciences Center, Tsinghua University, and Beijing Institute of Mathematical Sciences and Applications, Beijing, China.}
\email{fyangmath@mail.tsinghua.edu.cn}

\author{Lingfu Zhang}
\address{Department of Statistics, University of California, Berkeley, CA, USA.}
\email{lfzhang@berkeley.edu}

\newcommand{\g}{\mathcal E}
\newcommand{\oox}{\sfx}

\newcommand{\oot}{\sft}
\newcommand{\ooz}{\sfz}
\newcommand{\oow}{\sfw}

\newcommand{\den}{\rho_0}
\newcommand{\fff}{f_0}
\newcommand{\mmm}{m_0}

\newcommand{\lsca}{n}
\newcommand{\ssca}{\xi}

\newcommand{\Em}{E_-}
\newcommand{\Ep}{E_+}
\newcommand{\cusx}{x_c}
\newcommand{\cust}{t_c}

\newcommand{\qut}{\gamma}
\newcommand{\zzz}{z}

\newcommand{\zzc}{z_c}
\newcommand{\sct}{\overline{t}}

\newcommand{\pct}{\tau}
\newcommand{\pcx}{\gamma}
\newcommand{\pcz}{\mathbf{z}}
\newcommand{\pcw}{\mathbf{w}}

\newcommand{\Dfin}{\mathsf{D}}
\newcommand{\Sfin}{\mathsf{G}}
\newcommand{\oSfin}{\overline{\Sfin}}
\newcommand{\oell}{\ell}
\newcommand{\ooell}{\overline{\oell}}
\newcommand{\soell}{\oell^*}

\newcommand{\err}{\epsilon_4}
\newcommand{\ere}{\epsilon_2}
\newcommand{\errd}{\epsilon_3}
\newcommand{\errh}{\epsilon_1}

\newcommand{\z}{z}
\newcommand{\w}{w}
\newcommand{\semci}{{\rm sc}}

\newcommand{\don}{\mathds{1}}

\newcommand{\BE}{B}

\newcommand{\Lat}{\Lambda}

\begin{document}
\maketitle

\begin{abstract}
We study uniformly random lozenge tilings of general simply connected polygons. Under a technical assumption that is presumably generic with respect to polygon shapes, we show that the local statistics around a cusp point of the arctic curve converge to the Pearcey process. This verifies the widely predicted universality of edge statistics in the cusp case. Together with the smooth and tangent cases proved in \cite{aggarwal2021edge,ERT}, these are believed to be the three types of edge statistics that can arise in a generic polygon. Our proof is via a local coupling of the random tiling with non-intersecting Bernoulli random walks (NBRW). To leverage this coupling, we establish an optimal concentration estimate for the tiling height function around the cusp. As another step and also a result of potential independent interest, we show that the local statistics of NBRW around a cusp converge to the Pearcey process when the initial configuration consists of two parts with proper density growth, via careful asymptotic analysis of the determinantal formulas.
\end{abstract}

 \begin{figure}[!htb]
     \centering
     \begin{subfigure}{0.48\textwidth}
         \centering
         \includegraphics[width=4cm]{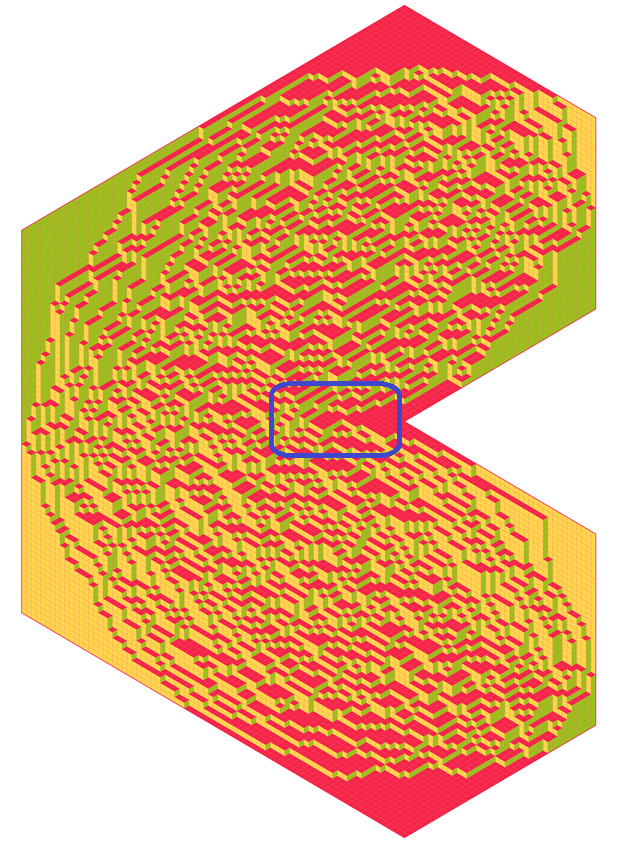}
     \end{subfigure}
     \hfill
     \begin{subfigure}{0.48\textwidth}
         \centering
\includegraphics[width=7cm]{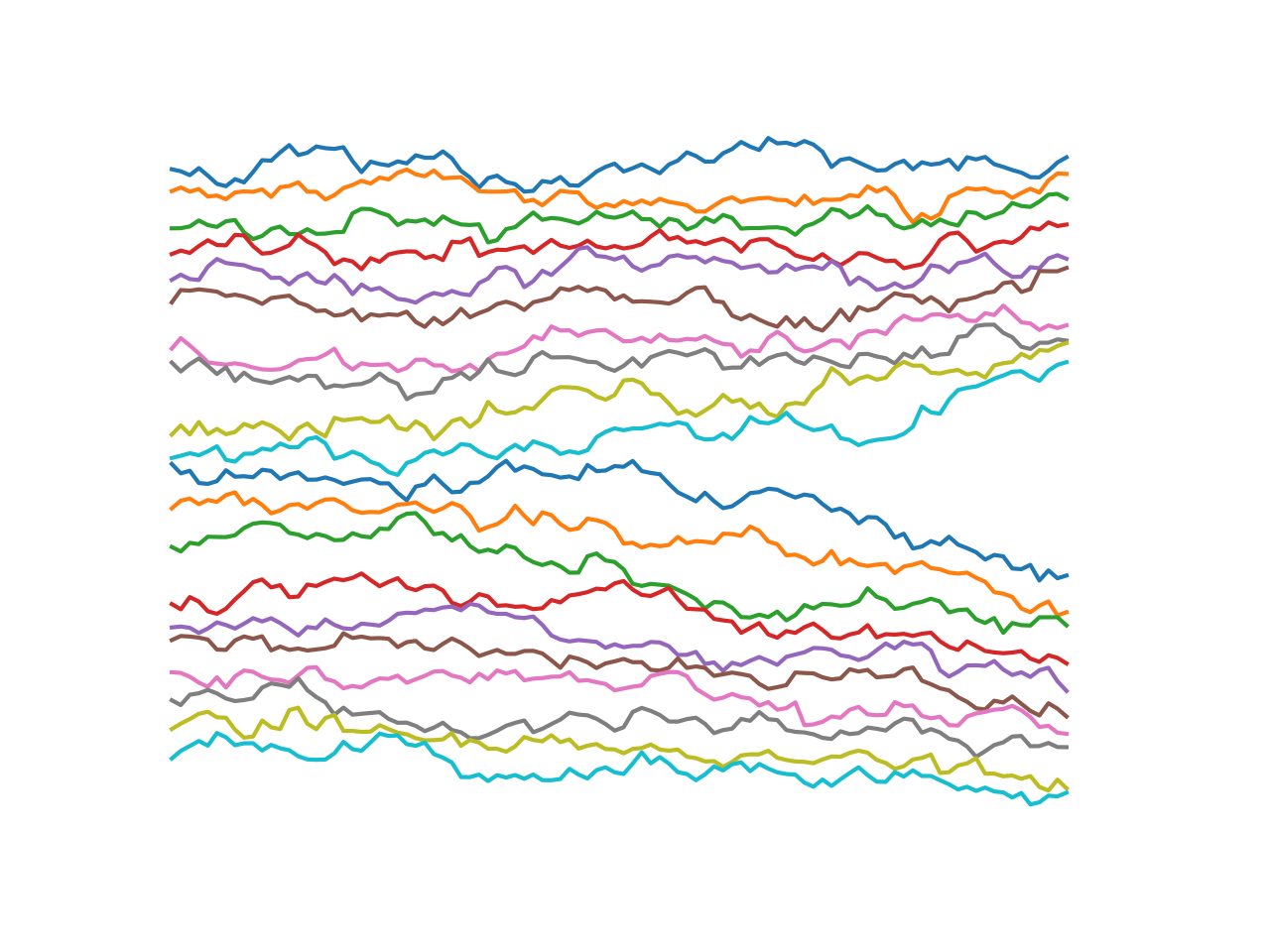}
     \end{subfigure}
        \caption{The left panel a uniformly sampled lozenge tiling, from the website of Leonid Petrov \url{https://lpetrov.cc/2016/08/Tilings-examples/} (using here under CC BY-SA 4.0). There is a cusp point in the blue box, where the paths formed by green and yellow tiles should converge to the Pearcey process, depicted in the right panel.}
        \label{fig:Pea}
\end{figure}

{
  \hypersetup{linkcolor=black}
  \tableofcontents
}

\section{Introduction}
 
The random lozenge tiling model is an exactly solvable two-dimensional statistical mechanical system that has attracted a significant amount of studies over the past few decades. For this model, many physical quantities of interest such as the partition function and correlation functions can be expressed in terms of determinants of an inverse Kasteleyn matrix. For random tilings of large domains, asymptotic analysis of these determinants leads to predictions of various universality phenomena in the large-scale limit; see, for instance, the book \cite{RT} for a comprehensive review. One such fundamental result is the \emph{limit shape} phenomenon claiming that the height function of a uniformly random tiling of a large domain would concentrate (after proper scaling) around a deterministic function. This behavior was first established for domino tilings of essentially arbitrary domains \cite{LSRT,VPT}, where the limit shape is expressed through a variational principle as the maximizer of a certain surface tension functional of the height function. This result was later extended to the case of random lozenge tilings in \cite{LSCE}, where the limit shape was written as the solution to a complex Burgers equation  which, in many cases, can be solved easily through the classical method of characteristics.

An interesting and important feature of the limit shape phenomenon is that the boundary condition induces a phase transition of the local statistics. Depending on the shape of the domain, it admits both frozen regions, where the associated height function is almost flat and deterministic, and liquid regions, where the height function is more rough and random. The curve separating these two regions is then called the arctic boundary. The reader can refer to \cite{RTT,TSP} for some early studies of this phenomenon in the context of random tilings, but we remark that a similar notion was discovered even earlier for Wulff Crystals in the Ising model; see e.g., \cite{ECM,Cerf,DKS_Book}. 

It is then natural to ask whether the local statistics are universal inside the liquid region and on the arctic boundaries, and how the universal limits behave if they exist. It is conjectured in \cite{VPT} that around a point inside the liquid region, the local statistics should be given by the ergodic Gibbs translation-invariant (EGTI) measure with slope matching the gradient of the limiting shape. It is known that the EGTI measure is unique and can be expressed as determinantal point processes with certain explicit \emph{extended discrete sine kernels} \cite{KOS, AST_Shef}. This conjecture was completely proved for random lozenge tilings of essentially arbitrary simply-connected domains in \cite{ULS}, based on and improving many previous proofs under stronger assumptions on the shapes of domains, such as  \cite{borodin2010q,Honeycomb,petrov2014asymptotics,gorin2017bulk,gorin2019universality,LLTS}, to name a few. 

\subsection*{Edge statistics and universal conjectures} 
Compared with the bulk statistics inside the liquid region, the \emph{edge statistics} near the arctic boundary exhibit much richer behaviors due to various possible singularities that the arctic boundary may develop. 
In studying edge statistics, the domain is usually taken to be polygonal (see \Cref{p} below). 
Besides being a reasonably general class of domains, such a restriction of being polygonal seems to be essential in establishing universality of edge statistics.
In fact, unlike the bulk statistics which are only determined by the macroscopic shape of the domain, the edge statistics are also sensitive to microscopic perturbations and can be altered by even a single defect at the boundary of the domain (as discussed in \cite{aggarwal2021edge} after the statement of the main result there). 

A detailed study of the arctic boundaries on general polygonal domains (which may not be simply connected) was conducted in \cite{LSCE,DMCS}, which showed that they are actually algebraic curves determined by the shapes of the polygons.
We note that \cite{LSCE} requires the sides to be cyclicly oriented and the domain to be simply connected, and these assumptions were removed in \cite{DMCS}.
In \cite{DMCS}, a complete classification of the regularity of these arctic curves is proved, with a total of six cases identified (this result holds for more general dimer models with periodic weight structure): (1) a smooth point of the arctic curve; (2) a point where the arctic curve is tangent to a side of the polygon; (3) a generic cusp point; (4) a cuspidal turning point; (5\&6) two types of tacnodes. The first three cases should appear for generic polygonal domains, whereas the last three cases are often referred to as non-generic singularities, in the sense that they are believed to be sensitive to perturbations of the side lengths (although this is not rigorously proved, and even its precise meaning is subtle; see \cite[Remark 2.8]{aggarwal2021edge} and the discussion below \Cref{pa}). 
The \emph{universal edge fluctuation conjecture} (see e.g., Section 9 of \cite{DMCS} or Lecture 19.2 of \cite{RT}) states that the local statistics of uniform lozenge tilings for polygonal domains have a universal scaling limit for each case. 
This conjecture has been verified for the first two cases. For case (1), the Airy line ensemble is conjectured to appear as the scaling limit. This was first proved for some special classes of domains, and recently solved in \cite{aggarwal2021edge} for general simply connected polygonal domains.
For case (2), it is conjectured that the GUE-corners process is the universal scaling limit at such tangency points. This was proved in \cite{ERT} for almost general domains, which improved previous results for some special classes of domains.

The goal of this paper is then to prove the universal edge fluctuation conjecture for case (3) (i.e., the \emph{cusp universality}). We will show that at any generic cusp on the arctic boundary, the local statistics of the uniformly random lozenge tiling converge to the Pearcey process.

The Pearcey process is a determinantal process described by the extended Pearcey kernel (given in \eqref{eq:pearcey}) and should be realized as a family of continuous random processes (see the right panel of \Cref{fig:Pea}).
The name `Pearcey' is from the connection between the kernel and Pearcey integrals.
Its first appearance traces back to \cite{brezin1998level,brezin1998universal} on certain matrices with Gaussian randomness. The limiting eigenvalue distribution around certain cusp points was shown to be a determinantal point process, whose kernel is then termed the `Pearcey kernel', corresponding to a single time slice of the Pearcey process.
Later, the extended Pearcey kernel was obtained at cusps of the arctic boundaries of random skew 3D partitions (which can be viewed as weighted random tilings of an infinite domain) \cite{RSPPP} and cusps of non-intersecting Brownian bridges starting from the origin and conditioned to end at two points \cite{Tracy2006}.
The Pearcey universality at cusps of more general non-intersecting Brownian bridges was established in \cite{adler2010universality,AFVM2010}. 
In random matrix theory, besides \cite{brezin1998level,brezin1998universal}, Pearcey limits have been proved for some other Gaussian matrix models, such as \cite{AM_CPAM,adler2011pearcey,Capitaine2016,HHN_Pearcey}, and for general Wigner-type matrices (with non-Gaussian entries) at cusps of the global density of states \cite{CEKS_Cusp,EKS_Cusp}.

As for random tilings, around any cusp, the local statistics can be encoded by a family of Bernoulli paths (as will be explained in \Cref{ssec:nbp} below).
Since \cite{RSPPP}, Pearcey limits of such paths have been established for special classes of domains; see e.g., \cite{borodin2008asymptotics, borodin2011limits, petrov2014asymptotics, borodin2015random, duits2020two}. 
It is natural to predict that the Pearcey universality at cusps holds for general polygonal domains, as stated in case (3) of the universal edge fluctuation conjecture.
Such a prediction actually traces back to \cite{RSPPP} and has been stated in many works such as \cite{DMCS,TCP,UEFDIPS,TNSDP,RT, EFLS, aggarwal2021edge,AvM_tac}. Our main result in this paper verifies this prediction for simply-connected polygonal domains, under certain technical conditions of the arctic curve.
\begin{thm}
Let $\fP$ be a simply-connected rational polygonal set forbidding certain presumably non-generic behaviors, as specified in \Cref{p} and Assumption \ref{pa} below.
For a uniformly random lozenge tiling of $n\fP$, around any cusp point of the arctic boundary, the corresponding paths (under appropriate scaling) converge to the Pearcey process as $n\to\infty$.
\end{thm}
A more formal and precise statement of this result is stated as Theorem \ref{thm:univ} below.
We remark that, as in all previous works showing Pearcey limits, our convergence to the Pearcey process is in the sense of convergence of point processes at finitely many times, due to the lack of a continuous theory of the Pearcey process. See \Cref{s:dccp} for more discussions.

\subsection*{Proof ideas} 
We now outline our proof of the cusp universality. First, we prove an optimal concentration (or rigidity) estimate for the corresponding Bernoulli paths near the cusp we are considering. For a simply connected polygon of diameter order $n$, near a smooth point of the arctic curve, the extreme path is concentrated within $n^{1/3+\delta}$ of the limit shape for any constant $\delta>0$ as shown in \cite{aggarwal2021edge}. 
We extend the argument there to the vicinity of a cusp and show that the extreme path is within $n^{1/4+\delta}$ of the cusp of the limit shape (the Pearcey fluctuation of the paths near a cusp is expected to be of order $n^{1/4}$). The concentration estimate is better as we get further away from the cusp and becomes $\oo(n^{1/4})$ if the distance from the cusp is at least $n^{1/2+\e}$ for a constant $\e>0$. 
For a smooth point of the arctic curve, such an optimal rigidity estimate almost suffices to deduce the Airy universality, because, as done in \cite{aggarwal2021edge}, one can sandwich the associated Bernoulli paths between two Airy line ensembles with different curvatures
to approximate the paths with $\oo(n^{1/3})$ error, which is negligible under the Airy scaling.
However, such a straightforward comparison cannot be carried out at a cusp, since it is mostly surrounded by the liquid region and is connected to the frozen region only in the tangent direction, and, furthermore, the Pearcey process is not versatile enough.

Instead, we will carve out a small domain around the cusp with height and width of order $\Omega(n^{1/2+\e})$. 
We then rewrite the tiling configuration in this domain using the well-known representation of a family of non-intersecting Bernoulli paths.
We consider the model of \emph{non-intersecting Bernoulli random walks (NBRW)} introduced in \cite{KCR}. It can be viewed as a family of independent simple random walks conditioned on never colliding, or a Markov chain in a discrete Weyl chamber. It also has the local Gibbs resampling property as tiling. 
We can construct an NBRW on the domain such that the limiting particle configuration matches the limit shape of the tiling function well. Then, the monotonicity property of the NBRW together with the concentration estimates of order $\oo(n^{1/4})$ on the boundaries of the domain shows that the NBRW is a good approximation of the tiling Bernoulli paths with a negligible $\oo(n^{1/4})$ error under the Pearcey scaling.

Now, the problem is reduced to showing the Pearcey universality of the corresponding NBRW, which is another challenging step of our proof and can be of independent interest (see \Cref{thm:univNBRWlim} below). 
It is known that the trajectories of NBRW is a determinantal point process, and a contour integral formula for the kernel is given in \cite{gorin2019universality}. 
We do an asymptotic analysis of the formula and show that when the initial configuration is appropriate (i.e., has two separate parts with proper density growth), the kernel near the cusp is close to the extended Pearcey kernel.
We use the \emph{steepest descent method}, which is well-known and can be traced back to Riemann in the 19th century. Its application in the study of determinantal point processes was pioneered by Okounkov (see e.g., \cite{okounkov2002symmetric}). Since then it has become standard and widely used in such tasks of asymptotic analysis in integrable probability (see e.g., \cite[Section 5]{borodin2016lectures}, or \cite[Lectures 15--18]{RT} in the context of tiling).
There are several technical challenges  in applying this method to our setting (see Section \ref{ssec:kapprox} for more details).
First, to show universality, we need to work with general initial conditions which require extra care. Second, the Pearcey process corresponds to that the saddle point (to be analyzed using the steepest descent method) is a `triple critical point' (as seen in \cite{RSPPP}). Besides, the fact that the distance between the cusp point and the boundary of the domain is of order much smaller than $\lsca$ makes it hard to tame the behavior of the analyzed function away from the saddle point. Much technical effort and some innovations (such as a multi-step approximation of the analyzed functions and a discretization of the contours) are presented to overcome these issues.

\medskip

Finally, we mention some possible future directions regarding tiling (or dimer) models that are closely related to this paper. First, the framework developed in this paper for the proof of cusp universality and Pearcey statistics can be applied to models beyond the realm of tilings. For example, it may be used to establish the Pearcey statistics for the Brownian motions on large unitary groups in certain regimes of interest \cite{AL2022}. Second, the scaling limits of random tilings around the three types of non-generic singularities have been proved for some special domains; see e.g., \cite{TCP} for the cusp-Airy process around cuspidal turning points, \cite{AJM_tac,AvM_tac} for the Tacnode process, and \cite{THCAF,TNSDP,AvM_tac} for the discrete Tacnode process. It would be interesting to prove the universality of these processes in tiling models. 
The third direction is to establish local statistics universality for other tiling models. 
For uniformly random domino tilings, we expect the existing methods can be adapted to show universalities at smooth and cusp points, analogous to \cite{aggarwal2021edge} and this paper.
It would also be interesting to consider random tilings with non-uniform measures, such as the various weighted ones \cite{BergDuits_2019,CDKL_CMP,Char_period,Mkrtchyan2021,BCJ_AAP,duits2020two}. 

\subsection*{Organization of the remaining text}
In Section \ref{Walks}, we formally define our model and present the main results regarding the Pearcey universality of uniform lozenge tilings (\Cref{thm:univ}) and NBRW (\Cref{thm:univNBRWlim}). In Section \ref{s:NBwe}, we introduce the monotonicity and Gibbs properties of uniform random tiling that will be used repeatedly. 
We will prove \Cref{thm:univ} in \Cref{sec:univproof} by combining three main ingredients: NBRW universality (from \Cref{thm:univNBRWlim}), optimal rigidity around cusps, and limiting height function estimates.
We will prove \Cref{thm:univNBRWlim} in \Cref{s:nibr}, and complete the remaining two steps in \Cref{s:optimalR} and \Cref{s:det} respectively.

\subsection*{Acknowledgement}
JH is supported by NSF grant DMS-2054835 and DMS-2331096.
FY is supported by Yau Mathematical Sciences Center, Tsinghua University, and Beijing Institute of Mathematical Sciences and Applications.
LZ is supported by the Miller Institute for Basic Research in Science, at the University of California, Berkeley, and NSF grant DMS-2246664. 
The authors would like to thank Amol Aggarwal, Erik Duse, Vadim Gorin, and Nicolai Reshetikhin for helpful comments on an earlier draft of this paper.
 
\section{Setup and main results} 
	
	\label{Walks}

To facilitate the presentation, we introduce some necessary notations that will be used throughout the paper. In this paper, we are interested in the asymptotic regime with $n\to \infty$. When we refer to a constant, it will not depend on the parameter $n$. Unless otherwise noted, we will use $C$ to denote a large positive constant, whose value may change from line to line. Similarly, we will use $\epsilon$, $\delta$, $c$, $\fc$, $\fd$ etc.~to denote small positive constants. 
For an event $\Xi_n$ whose definition depends on $n$, we say that it holds \emph{with overwhelming probability} (w.o.p.), if for any constant $D>0$ there is $\mathbb P(\Xi_n)\ge 1-n^{-D}$ for all large enough $n$.
For any two (possibly complex) sequences $a_n$ and $b_n$ depending on $n$, $a_n = \OO(b_n)$ means that $|a_n| \le C|b_n|$ for some constant $C>0$, whereas $a_n=\oo(b_n)$ or $|a_n|\ll |b_n|$ means that $|a_n| /|b_n| \to 0$ as $n\to \infty$. We say that $a_n \lesssim b_n$ (or $b_n\gtrsim a_n$) if $a_n = \OO(b_n)$, and $a_n \asymp b_n$ (or $a_n =\Omega(b_n)$) if $a_n = \OO(b_n)$ and $b_n = \OO(a_n)$. 

For any $x,y\in\R\cup\{-\infty, \infty\}$, $x\le y$, we denote $\llbracket x, y\rrbracket = [x,y]\cap \Z$, $x\vee y=\max\{x, y\}$, and $x\wedge y=\min\{x, y\}$.
For an event $A$, we let $\don_A$ or $\don[A]$ denote its indicator function.
For any set $S$ we use $|S|$ to denote its cardinality. For any $D\subset \R^2$ we use $\overline{D}$ to denote its closure.
We use $\HH=\{\z\in \C: \Im \z >0\}$ and $\HH^-=\{\z\in\C: \Im \z <0\}$ to denote the upper- and lower-half complex planes, respectively.
We also employ the Pochhammer symbols $(z)_{k}=z(z+1)\ldots(z+k-1)$ and the binomial coefficients $\binom{k}{a}=(-1)^a\frac{(-k)_a}{a!}$, for any $z\in\C$ and $k,a\in\Z_{\ge 0}$.

\subsection{Lozenge tiling}	\label{FunctionWalks}
	We denote by $\mathbb{T}$ the \emph{triangular lattice}, namely, the graph whose vertex set is $\mathbb{Z}^2$ and whose edge set consists of edges connecting $(\sfx, \sft), (\sfx', \sft') \in \mathbb{Z}^2$ if $(\sfx' - \sfx, \sft' - \sft) \in \{ (1, 0), (0, 1), (1, 1)\}$. The axes of $\mathbb{T}$ are the lines $\{ \sfx = 0 \}$, $\{ \sft = 0 \}$, and $\{ \sfx  = \sft \}$, and the faces of $\mathbb{T}$ are triangles with vertices of the form $\big\{ (\sfx, \sft), (\sfx + 1, \sft), (\sfx + 1, \sft + 1) \big\}$ or $\big\{ (\sfx, \sft), (\sfx, \sft + 1), (\sfx + 1, \sft + 1) \big\}$. A \emph{domain} $\mathsf{R} \subseteq \mathbb{R}^2$ is a finite union of triangular faces that is simply-connected. As a slight abuse of this notation, we also denote by $\mathsf{R}$ the set of all vertices incident to these triangular faces or the subgraph of $\mathbb{T}$ induced by these vertices.
 
	\begin{figure}[!ht]
		
		\begin{center}		
			
			\begin{tikzpicture}[
				>=stealth,
				auto,
				style={
					scale = .3
				}
				]		
				
				\draw[-, black] (15, 7) node[above, scale = .7]{$\mathsf{H}$}-- (15, 5) node[right, scale = .7]{$\mathsf{H}$} -- (13, 3) node[below, scale = .7]{$\mathsf{H}$} -- (13, 5) node[left, scale = .7]{$\mathsf{H}$} -- (15, 7); 
				\draw[-, dashed, black] (13, 5) -- (15, 5);
				
				\draw[-, black] (19, 5) node[above, scale = .7]{$\mathsf{H}$} -- (21, 5) node[right, scale = .7]{$\mathsf{H} + 1$} -- (19, 3) node[below, scale = .7]{$\mathsf{H} + 1$} -- (17, 3) node[left, scale = .7]{$\mathsf{H}$}-- (19, 5);
				\draw[-, dashed, black] (19, 3) -- (19, 5);
				
				\draw[-, black] (25, 3) node[below, scale = .7]{$\mathsf{H}$} -- (25, 5) node[above, scale = .7]{$\mathsf{H}$} -- (27, 5) node[above, scale = .7]{$\mathsf{H} + 1$} -- (27, 3) node[below, scale = .7]{$\mathsf{H} + 1$} -- (25, 3);
				\draw[-, dashed, black] (25, 3) -- (27, 5);
				
				\draw[-, black] (-9-2, 1.5) node[below, scale = .7]{$\mathsf{H}$} -- (-7-2, 1.5) node[below, scale = .7]{$\mathsf{H} + 1$} -- (-5-2, 3.5) node[right, scale = .7]{$\mathsf{H} + 1$} -- (-5-2, 5.5) node[above, scale = .7]{$\mathsf{H} + 1$} -- (-7-2, 5.5) node[above, scale = .7]{$\mathsf{H}$} -- (-9-2, 3.5) node[left, scale = .7]{$\mathsf{H}$} -- (-9-2, 1.5);
				\draw[-, black] (-9-2, 3.5) -- (-7-2, 3.5)node[scale = .7]{$\mathsf{H}+1$} -- (-7-2, 1.5);
				\draw[-, black] (-7-2, 3.5) -- (-5-2, 5.5);
				
				\draw[-, black] (0, 0) -- (5, 0);
				\draw[-, black] (0, 0) -- (0, 4);
				\draw[-, black] (5, 0) -- (8, 3);
				\draw[-, black] (0, 4) -- (3, 7);
				\draw[-, black] (8, 3) -- (8, 7); 
				\draw[-, black] (8, 7) -- (3, 7);	
				
				\draw[-, black] (1, 0) -- (1, 3) -- (0, 3) -- (1, 4) -- (1, 5) -- (2, 5) -- (4, 7);
				\draw[-, black] (0, 2) -- (2, 2) -- (2, 0);
				\draw[-, black] (0, 1) -- (3, 1) -- (3, 0) -- (4, 1) -- (6, 1);
				\draw[-, black] (2, 6) -- (5, 6) -- (6, 7) -- (6, 6) -- (8, 6);
				\draw[-, black] (4, 6) -- (5, 7);
				\draw[-, black] (1, 4) -- (2, 4) -- (2, 5); 
				\draw[-, black] (2, 4) -- (3, 5) -- (3, 6);
				\draw[-, black] (1, 3) -- (2, 4);
				\draw[-, black] (2, 2) -- (3, 3) -- (3, 4) -- (2, 4) -- (2, 3) -- (1, 2);
				\draw[-, black] (2, 3) -- (3, 3) -- (3, 2) -- (4, 3) -- (4, 6) -- (5, 6) -- (5, 3) -- (4, 3);
				\draw[-, black] (3, 5) -- (5, 5) -- (6, 6);
				\draw[-, black] (7, 7) -- (7, 4) -- (8, 4) -- (7, 3) -- (7, 2) -- (6, 2) -- (4, 0);
				\draw[-, black] (8, 5) -- (7, 5) -- (6, 4) -- (6, 5) -- (7, 6);
				\draw[-, black] (3, 3) -- (4, 4) -- (6, 4);
				\draw[-, black] (3, 4) -- (4, 5);
				\draw[-, black] (6, 5) -- (5, 5);
				\draw[-, black] (7, 4) -- (6, 3) -- (7, 3);
				\draw[-, black] (2, 1) -- (3, 2) -- (5, 2) -- (6, 3);
				\draw[-, black] (3, 1) -- (5, 3) -- (6, 3) -- (6, 4); 
				\draw[-, black] (4, 1) -- (4, 2); 
				\draw[-, black] (5, 1) -- (5, 2); 
				\draw[-, black] (6, 2) -- (6, 3); 
				
				\draw[] (0, 0) circle [radius = 0] node[below, scale = .5]{$0$};
				\draw[] (1, 0) circle [radius = 0] node[below, scale = .5]{$1$};
				\draw[] (2, 0) circle [radius = 0] node[below, scale = .5]{$2$};
				\draw[] (3, 0) circle [radius = 0] node[below, scale = .5]{$3$};
				\draw[] (4, 0) circle [radius = 0] node[below, scale = .5]{$4$};
				\draw[] (5, 0) circle [radius = 0] node[below, scale = .5]{$5$};
				
				\draw[] (0, 1) circle [radius = 0] node[left, scale = .5]{$0$};
				\draw[] (1, 1) circle [radius = 0] node[above = 3, left = 0, scale = .5]{$1$};
				\draw[] (2, 1) circle [radius = 0] node[above = 3, left = 0, scale = .5]{$2$};
				\draw[] (3, 1) circle [radius = 0] node[left = 1, above = 0, scale = .5]{$3$};
				\draw[] (4, 1) circle [radius = 0] node[right = 1, below = 0, scale = .5]{$3$};
				\draw[] (5, 1) circle [radius = 0] node[right = 1, below = 0, scale = .5]{$4$};
				\draw[] (6, 1) circle [radius = 0] node[below, scale = .5]{$5$};
				
				\draw[] (0, 2) circle [radius = 0] node[left, scale = .5]{$0$};
				\draw[] (1, 2) circle [radius = 0] node[above = 3, left = 0, scale = .5]{$1$};
				\draw[] (2, 2) circle [radius = 0] node[left = 1, above = 0, scale = .5]{$2$};
				\draw[] (3, 2) circle [radius = 0] node[left = 3, above = 0, scale = .5]{$2$};
				\draw[] (4, 2) circle [radius = 0] node[below = 3, right = 0, scale = .5]{$3$};
				\draw[] (5, 2) circle [radius = 0] node[below = 1, right = 0, scale = .5]{$4$};
				\draw[] (6, 2) circle [radius = 0] node[right = 1, below = 0, scale = .5]{$4$};
				\draw[] (7, 2) circle [radius = 0] node[below, scale = .5]{$5$};
				
				\draw[] (0, 3) circle [radius = 0] node[left, scale = .5]{$0$};
				\draw[] (1, 3) circle [radius = 0] node[left = 1, above = 0, scale = .5]{$1$};
				\draw[] (2, 3) circle [radius = 0] node[above = 1, left = 0, scale = .5]{$1$};
				\draw[] (3, 3) circle [radius = 0] node[left = 3, above = 0, scale = .5]{$2$};
				\draw[] (4, 3) circle [radius = 0] node[right = 3, above = 0, scale = .5]{$2$};
				\draw[] (5, 3) circle [radius = 0] node[right = 3, above = 0, scale = .5]{$3$};
				\draw[] (6, 3) circle [radius = 0] node[below = 3, right = 0, scale = .5]{$4$};
				\draw[] (7, 3) circle [radius = 0] node[below = 1, right = 0, scale = .5]{$5$};
				\draw[] (8, 3) circle [radius = 0] node[right, scale = .5]{$5$};
				
				\draw[] (0, 4) circle [radius = 0] node[left = 1, scale = .5]{$0$};
				\draw[] (1, 4) circle [radius = 0] node[above = 1, left = 0, scale = .5]{$0$};
				\draw[] (2, 4) circle [radius = 0] node[left = 3, above = 0, scale = .5]{$1$};
				\draw[] (3, 4) circle [radius = 0] node[above, scale = .5]{$2$};
				\draw[] (4, 4) circle [radius = 0] node[above = 1, left = 0, scale = .5]{$2$};
				\draw[] (5, 4) circle [radius = 0] node[left = 3, above = 0, scale = .5]{$3$};
				\draw[] (6, 4) circle [radius = 0] node[left = 3, above = 0, scale = .5]{$4$};
				\draw[] (7, 4) circle [radius = 0] node[right = 3, above = 0, scale = .5]{$4$};
				\draw[] (8, 4) circle [radius = 0] node[right, scale = .5]{$5$};
				
				\draw[] (1, 5) circle [radius = 0] node[above, scale = .5]{$0$};
				\draw[] (2, 5) circle [radius = 0] node[left = 1, above = 0, scale = .5]{$1$};
				\draw[] (3, 5) circle [radius = 0] node[above = 2, left = 0, scale = .5]{$1$};
				\draw[] (4, 5) circle [radius = 0] node[left = 3, above = 0, scale = .5]{$2$};
				\draw[] (5, 5) circle [radius = 0] node[left = 3, above = 0, scale = .5]{$3$};
				\draw[] (6, 5) circle [radius = 0] node[left = 1, above = 0, scale = .5]{$4$};
				\draw[] (7, 5) circle [radius = 0] node[right = 3, above = 0, scale = .5]{$4$};
				\draw[] (8, 5) circle [radius = 0] node[right, scale = .5]{$5$};			
				
				\draw[] (2, 6) circle [radius = 0] node[above, scale = .5]{$0$};
				\draw[] (3, 6) circle [radius = 0] node[above, scale = .5]{$1$};
				\draw[] (4, 6) circle [radius = 0] node[left = 2, above = 0, scale = .5]{$2$};
				\draw[] (5, 6) circle [radius = 0] node[left = 2, above = 0, scale = .5]{$3$};
				\draw[] (6, 6) circle [radius = 0] node[right = 3, above = 0, scale = .5]{$3$};
				\draw[] (7, 6) circle [radius = 0] node[right = 3, above = 0, scale = .5]{$4$};
				\draw[] (8, 6) circle [radius = 0] node[right, scale = .5]{$5$};
				
				\draw[] (3, 7) circle [radius = 0] node[above, scale = .5]{$0$};
				\draw[] (4, 7) circle [radius = 0] node[above, scale = .5]{$1$};
				\draw[] (5, 7) circle [radius = 0] node[above, scale = .5]{$2$};
				\draw[] (6, 7) circle [radius = 0] node[above, scale = .5]{$3$};
				\draw[] (7, 7) circle [radius = 0] node[above, scale = .5]{$4$};
				\draw[] (8, 7) circle [radius = 0] node[above, scale = .5]{$5$};
				
			\end{tikzpicture}
			
		\end{center}
		
		\caption{\label{tilinghexagon} Depicted to the right are the three types of lozenges. Depicted in the middle is a lozenge tiling of a hexagon. One may view this tiling as a packing of boxes (of the type depicted on the left) into a large corner, which gives rise to a height function (shown in the middle). }
	\end{figure}
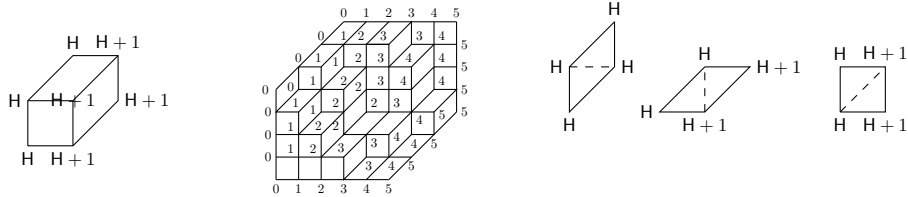 
 When viewing $\mathsf{R}$ as a vertex set, the \emph{boundary} $\partial \mathsf{R} \subseteq \mathsf{R}$ is the set of vertices $\mathsf{v} \in \mathsf{R}$ adjacent to a vertex in $\mathbb{T} \setminus \mathsf{R}$; when viewing $\mathsf{R}$ as a union of triangular faces, $\partial \mathsf{R}$ is the union of its boundary edges.

 A \emph{dimer covering} of a domain $\mathsf{R} \subseteq \mathbb{T}$ is defined to be a perfect matching on the dual graph of $\mathsf{R}$ (which has a vertex for each triangular face of $\mathsf{R}$, and an edge for each pair of adjacent triangular faces). A pair of adjacent triangular faces in any such matching forms a parallelogram, which we will also refer to as a \emph{lozenge} or \emph{tile}. Lozenges can be oriented in one of three ways; see the right side of \Cref{tilinghexagon} for all three orientations. 
 The vertices are in the form of
 \begin{itemize}
     \item $\big\{ (\sfx, \sft), (\sfx, \sft + 1), (\sfx + 1, \sft + 2), (\sfx + 1, \sft + 1) \big\}$, the left lozenge in the right side of \Cref{tilinghexagon}, or
     \item $\big\{ (\sfx, \sft), (\sfx + 1, \sft), (\sfx + 2, \sft + 1), (\sfx + 1, \sft + 1) \big\}$, the middle lozenge in the right side of \Cref{tilinghexagon}, or
     \item $\big\{ (\sfx, \sft), (\sfx + 1, \sft), (\sfx + 1, \sft + 1), (\sfx , \sft + 1) \big\}$, the right lozenge in the right side of \Cref{tilinghexagon}.
 \end{itemize}
 These lozenges are referred to as  \emph{type $1$},  \emph{type $2$}, and  \emph{type $3$} lozenges, respectively.
 A dimer covering of $\mathsf{R}$ can equivalently be interpreted as a tiling of $\mathsf{R}$ by lozenges of types $1$, $2$, and $3$. Therefore, we will also refer to a dimer covering of $\mathsf{R}$ as a \emph{(lozenge) tiling}. We call $\mathsf{R}$ \emph{tileable} if it admits a tiling.
 
 The main object we investigate in this paper is \emph{uniformly random tiling}, where we consider the probability measure on the (finite) space of all tilings of a tileable domain where each tiling has the same probability.

\subsubsection{Height function and its restriction at the boundary}  \label{ssec:hf}
For a chosen vertex $\mathsf{v}$ of $\mathsf{R}$ and an integer $h_0 \in \mathbb{Z}$, one can associate with any tiling of $\mathsf{R}$ a \emph{height function} $\mathsf{H}: \mathsf{R} \rightarrow \R$ as follows. First, set $\mathsf{H} (\mathsf{v}) = h_0$, and then define $\mathsf{H}$ at the remaining vertices of $\mathsf{R}$ in such a way that the height functions along the four vertices of any lozenge in the tiling are of the form depicted on the right side of \Cref{tilinghexagon}. In particular, we require that $\mathsf{H} (\sfx + 1, \sft) = \mathsf{H} (\sfx, \sft)$ if and only if $(\sfx, \sft)$ and $(\sfx + 1, \sft)$ are vertices of the same type $1$ lozenge, and that $\mathsf{H} (\sfx, \sft) - \mathsf{H} (\sfx, \sft + 1) = 1$ if and only if $(\sfx, \sft)$ and $(\sfx, \sft + 1)$ are vertices of the same type $2$ lozenge. Since $\mathsf{R}$ is simply connected, the height function $\mathsf{H}$ on the vertex set $\mathsf{R}$ is uniquely determined by these conditions (up to adding a global constant which is necessarily an integer).
This height function $\mathsf{H}$ can be extended by linearity to the faces of $\mathsf{R}$, so that it may also be viewed as a piecewise linear function on $\mathsf{R} \subseteq \mathbb{R}^2$.

For any height function $\mathsf{H}$, we refer to the restriction $\mathsf{h} = \mathsf{H}|_{\partial \mathsf{R}}$ as the \emph{boundary height function}, which is a piecewise linear function on the boundary edges.
We note that for any tileable domain $\sfR$, the boundary height function, up to a global shift, is independent of the choice of the tiling (thereby uniquely determined by $\sfR$).
Indeed, along any boundary edge with slope $1$ or $\infty$, any boundary height function $\sfh$ must be constant. Along any boundary edge with slope $0$, $\sfh$ must grow linearly with rate $1$, i.e., for any $(\sfx, \sft), (\sfx+1, \sft) \in \partial\sfR\cap\T$, there is $\sfH(\sfx, \sft+1)=\sfH(\sfx, \sft)+1$.
Since $\sfR$ is simply connected, $\partial \sfR$ is a closed curve, and the above rules determine $\sfh$ once its value at one point in $\partial \sfR$ is given.

We refer to the middle of \Cref{tilinghexagon} for an example; as depicted there, we can also view a (lozenge) tiling of $\mathsf{R}$ (which is a hexagon) as a packing of $\mathsf{R}$ by boxes of the type shown on the left side of \Cref{tilinghexagon}. In this case, the value $\mathsf{H} (\mathsf{u})$ of the height function associated with this tiling at some vertex $\mathsf{u} \in \mathsf{R}$ denotes the height of the stack of boxes at $\mathsf{u}$.

A tiling can also be interpreted as a family of non-intersecting Bernoulli paths.
  
 \subsubsection{Non-intersecting Bernoulli paths} \label{ssec:nbp}
 	A  \emph{Bernoulli path} is a function $\sfb:\llbracket \sfr, \sfs\rrbracket\to \Z$ for some $\sfr, \sfs\in \Z$,
 such that $\sfb (\oot + 1) - \sfb (\oot) \in \{ 0, 1 \}$ for each $\oot \in \qq{\sfr, \sfs - 1}$. It denotes the space-time trajectory of a walk, which takes either a `non-jump' ($\sfb (\oot + 1) = \sfb (\oot)$) or a `right-jump' ($\sfb (\oot + 1)= \sfb (\oot)+1$) at each step. We call the interval $\qq{\sfr, \sfs}$ the \emph{time span} of the Bernoulli path $\sfb$. 
As an extension of the notion of Bernoulli paths,
for any $I\subset \llbracket \sfr, \sfs\rrbracket$, we also call $\sfb$ restricted to $I$ a Bernoulli path (whose time span $I$ is possibly a union of several discrete intervals).

Take any $M, N\in \Z$, $M\le N$, and $I_i\subset \Z$ for each $i\in\qq{M,N}$.
A family of (consecutive) Bernoulli paths $\{\sfb_i\}_{i\in\qq{M, N}}$, with each $\sfb_i$ having time span $I_i$, is called \emph{non-intersecting}, if for any $M \le i < j \le N$ and any $\oot\in I_i\cap I_j$, there is always $\sfb_i (\oot) - i\le \sfb_j (\oot)-j$.
Another notation that we will also use to denote such non-intersecting Bernoulli paths is in the form of a function $\sfB$, from $\Z$ to the set
\[
\{\{\sfx_i\}_{i\in \Phi}\in\Z^\Phi: \subset \Z,\; \sfx_i-i\le \sfx_j-j, \forall i<j\in \Phi\},
\]
with $\sfB(\oot)=\{\sfb_i(\oot)\}_{i\in \qq{M,N}, I_i\ni \oot}$ for each $\oot\in\Z$, and $\Phi \subset \Z$ is any index set.

\begin{figure}	[!ht]
		\begin{center}		
			\begin{tikzpicture}[
				>=stealth,
				auto,
				style={
					scale = .4
				}
				]
				
				\draw[-, black]  (1.5, 2) -- (1.5, 3) -- (2.5, 4) ; 
				\draw[-, black] (2.5, 2) -- (3.5, 3) -- (3.5, 4)--(4.5,5); 
				\draw[-, black] (4.5, 2) -- (4.5, 3) -- (5.5, 4)--(5.5,5)--(5.5,6);
				\draw[-, black] (6.5, 2) -- (6.5, 3) -- (7.5, 4)--(7.5,5) ; 
				\draw[-, black]  (7.5, 2) -- (8.5, 3) -- (8.5, 4);
				\draw[-, black]  (9.5, 2) -- (9.5, 3) -- (10.5, 4) --(10.5,5);
				
				\filldraw[fill=black] (1.5, 2) circle [radius = .1] node[below , scale = .6]{$\mathsf{q}_{1} (0)$};	
				\filldraw[fill=black] (2.5, 2) circle [radius = .1] node[ below, scale = .6]{$\mathsf{q}_{2} (0)$};	
				\filldraw[fill=black] (4.5, 2) circle [radius = .1] node[below, scale = .6]{$\mathsf{q}_3 (0)$};	
				\filldraw[fill=black] (6.5, 2) circle [radius = .1] node[below, scale = .6]{$\mathsf{q}_4 (0)$};	
				\filldraw[fill=black] (7.5, 2) circle [radius = .1] node[below, scale = .6]{$\mathsf{q}_5 (0)$};	
				\filldraw[fill=black] (9.5, 2) circle [radius = .1] node[below, scale = .6]{$\mathsf{q}_6 (0)$};

				\filldraw[fill=black] (1.5, 2) circle [radius = .1];
				\filldraw[fill=black] (2.5, 2) circle [radius = .1];
				\filldraw[fill=black] (4.5, 2) circle [radius = .1];
				\filldraw[fill=black] (6.5, 2) circle [radius = .1];
				\filldraw[fill=black] (7.5, 2) circle [radius = .1];
				\filldraw[fill=black] (9.5, 2) circle [radius = .1];
				
				\filldraw[fill=black] (1.5, 3) circle [radius = .1];
				\filldraw[fill=black] (3.5, 3) circle [radius = .1];
				\filldraw[fill=black] (4.5, 3) circle [radius = .1];
				\filldraw[fill=black] (6.5, 3) circle [radius = .1];
				\filldraw[fill=black] (8.5, 3) circle [radius = .1];
				\filldraw[fill=black] (9.5, 3) circle [radius = .1];
				
				\filldraw[fill=black] (2.5, 4) circle [radius = .1]node[above , scale = .6]{$\mathsf{q}_{1} (2)$};
                    \filldraw[fill=black] (3.5, 4)circle [radius = .1];
				\filldraw[fill=black] (4.5, 5) circle [radius = .1]node[above , scale = .6]{$\mathsf{q}_{2} (3)$};
    \filldraw[fill=black] (5.5, 5) circle [radius = .1];
    \filldraw[fill=black] (5.5, 6) circle [radius = .1]node[above , scale = .6]{$\mathsf{q}_{3} (4)$};
				\filldraw[fill=black] (5.5, 4) circle [radius = .1];
				\filldraw[fill=black] (7.5, 4) circle [radius = .1];
                    \filldraw[fill=black] (7.5, 5) circle [radius = .1]node[above , scale = .6]{$\mathsf{q}_{4} (3)$};
				\filldraw[fill=black] (8.5, 4) circle [radius = .1]node[above , scale = .6]{$\mathsf{q}_{5} (2)$};
				\filldraw[fill=black] (10.5, 4) circle [radius = .1];
                    \filldraw[fill=black] (10.5, 5) circle [radius = .1]node[above , scale = .6]{$\mathsf{q}_{6} (3)$};

				\draw[-, black] (15, 2) -- (15, 3) -- (16, 4) -- (17, 4) -- (18, 4) -- (19, 5) -- (19, 4) -- (20, 4)  -- (21, 5) -- (21, 4) -- (23, 4)  -- (24, 5)--(24,4) -- (25, 4) -- (24, 3) -- (24, 2)  -- (23, 2) --(23,3)--(22,2)-- (20, 2) --(20,3)-- (19, 2) -- (18, 2) -- (18, 3) -- (17, 2) -- (15, 2);

				\draw[-, black, dashed] (15.5, 2)node[below, scale = .6]{$\mathsf{q}_{1}$}  -- (15.5, 3) -- (16.5, 4) ; 
				\draw[-, black, dashed] (16.5, 2) node[below , scale = .6]{$\mathsf{q}_{2}$} -- (17.5, 3) -- (17.5, 4)--(18.5,5); 
				\draw[-, black, dashed] (18.5, 2) node[ below , scale = .6]{$\mathsf{q}_3$} -- (18.5, 3) -- (19.5, 4) --(19.5,5)--(19.5,6);
				\draw[-, black, dashed] (20.5, 2) node[below, scale = .6]{$\mathsf{q}_4$}-- (20.5, 3) -- (21.5, 4) -- (21.5,5); 
				\draw[-, black, dashed]  (21.5, 2)node[below, scale = .6]{$\mathsf{q}_5$} -- (22.5, 3) -- (22.5, 4);
				\draw[-, black, dashed]  (23.5, 2)node[below, scale = .6]{$\mathsf{q}_6$} -- (23.5, 3) -- (24.5, 4)-- (24.5,5);
				\draw[-, black] (18, 3) -- (18, 4) ;
				\draw[-, black] (16, 2) -- (17, 3) -- (17, 4) --(18,5)--(19,5); 
				\draw[-, black] (18, 2) -- (18, 3) -- (19, 4);
				\draw[-, black]  (20, 2) -- (20, 3) -- (21, 4) ; 
                    \draw[-, black] (20,4)--(20,5)--(20,6)--(19,6)--(19,5)--(20,5);
				\draw[-, black] (21, 2) -- (22, 3) -- (22, 4) --(22,5)--(21,5) ;
				\draw[-, black ] (23, 2) -- (23, 3) -- (24, 4);

				\draw[-, black] (15, 3) -- (16, 3); 
				\draw[-, black] (16, 4) -- (17, 4); 
				\draw[-, black] (16, 3) -- (17, 4);
				
				\draw[-, black] (17, 2) -- (18, 3) -- (18, 2);
				
				\draw[-, black] (20, 4) -- (19, 4);
				\draw[-, black] (20, 4) -- (19, 3) -- (19, 2);
				\draw[-, black] (18, 2) -- (19, 2);
				\draw[-, black] (18, 3) -- (19, 3);
				\draw[-, black] (19, 2) -- (20, 3) -- (20, 4);
				\draw[-, black] (21, 2) -- (21, 3) -- (22, 4) ;
				\draw[-, black]  (22, 2) -- (23, 3) -- (23, 4) ;

				\draw[-, black] (20, 2) -- (22, 2);
				\draw[-, black] (20, 3) -- (21, 3);
				\draw[-, black] (21, 4) -- (23, 4);
				\draw[-, black] (22, 3) -- (23, 3);

				\draw[-, black] (17, 3) -- (18, 3);
				\draw[-, black] (24, 3) -- (25, 4)--(25,5)--(24,5);
				\draw[-, black] (23, 2) -- (24, 2); 
				\draw[-, black] (23, 3) -- (24, 3); 
				\draw[-, black] (24, 4) -- (25, 4);
				\draw[-, black] (16, 2) -- (16, 3);		
			\end{tikzpicture}	
		\end{center}
		\caption{\label{walksfigure} Depicted to the left is an ensemble consisting of six non-intersecting Bernoulli paths. Depicted to the right is an associated lozenge tiling. }		
	\end{figure}
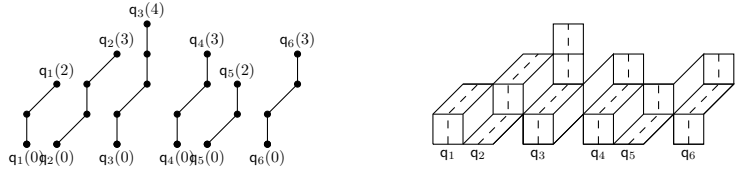 	
	For any domain $\mathsf{R}$ and any tiling $\mathscr{M}$ of $\mathsf{R}$, we may interpret $\mathscr{M}$ as a family of non-intersecting Bernoulli paths by (roughly speaking) first omitting all type $1$ lozenges from $\mathscr{M}$, and then viewing any type $2$ or type $3$ tile as a right-jump or non-jump of a Bernoulli path, respectively; see \Cref{walksfigure} for a depiction. 
 More formally, the non-intersecting Bernoulli paths are defined by taking any height function $\mathsf{H}: \mathsf{R} \rightarrow \mathbb{Z}$ associated with the tiling $\mathscr{M}$, and letting $\sfb_i(\oot)$ be the number satisfying
 \begin{equation}  \label{eq:defbH}
 \mathsf{H}(\sfb_i(\oot), \oot)=i,\quad  \mathsf{H}(\sfb_i(\oot)+1, \oot)=i+1,    
 \end{equation}
if such a number exists (note that the number is also unique since $\mathsf{H}(\cdot, \oot)$ is non-decreasing).
We remark that the non-intersecting Bernoulli paths are uniquely determined by the tiling $\mathscr{M}$, modulus a global shift of the indices of individual paths.

	\subsection{Limit shapes}

	\label{HeightLimit} 
	
	To analyze the limits of height functions of random tilings, it will be useful to introduce continuum analogs of several notions considered in \Cref{FunctionWalks}. 
 We set 
	\begin{flalign} \label{t}
		\mathcal{T} = \big\{ (s, t) \in (0,1) \times (-1,0): s+t>0 \big\} \subset \mathbb{R}^2, 
	\end{flalign} 

	\noindent and its closure $\overline{\mathcal{T}} = \big\{ (s, t) \in [0,1] \times [-1,0]: s+t\ge 0 \big\}$. We interpret $\overline{\mathcal{T}}$ as the set of possible gradients, also called \emph{slopes}, for a continuum height function; $\mathcal{T}$ is then the set of `non-frozen' or `liquid' slopes, whose associated tilings contain tiles of all types. For any simply-connected open set $\mathfrak{R} \subset \mathbb{R}^2$, we say that a function $H : \mathfrak{R} \rightarrow \mathbb{R}$ is \emph{admissible} if $H$ is $1$-Lipschitz and $\nabla H(v) \in \overline{\mathcal{T}}$ for almost all $v \in \mathfrak{R}$. 
 For any function $h: \partial \mathfrak{R} \rightarrow \mathbb{R}$, we define $\Adm (\mathfrak{R}; h)$ to be the set of admissible functions $H: \mathfrak{R} \rightarrow \mathbb{R}$ with $H |_{\partial \mathfrak{R}} = h$; and we say that $h: \partial \mathfrak{R}$ \emph{admits an admissible extension to $\mathfrak{R}$} if $\Adm (\mathfrak{R}; h)$ is not empty.
	
	We say a sequence of domains $\mathsf{R}_1, \mathsf{R}_2, \ldots \subset \mathbb{T}$ \emph{converges} to a simply-connected set $\mathfrak{R} \subset \mathbb{R}^2$ if $n^{-1} \mathsf{R}_n \subseteq \mathfrak{R}$ for each $n \ge 1$ and  $\lim_{n \rightarrow \infty} \dist (n^{-1} \del\mathsf{R}_n, \del\mathfrak{R}) = 0$. We further say a sequence $\mathsf{h}_1, \mathsf{h}_2, \ldots $ of boundary height functions on $\mathsf{R}_1, \mathsf{R}_2, \ldots $ \emph{converges} to a boundary height function $h : \partial \mathfrak{R} \rightarrow \mathbb{R}$ if $\lim_{n \rightarrow \infty} n^{-1} \mathsf{h}_n (n v_n) = h (v)$ for any sequence of points $v_n\to v$ with $v_n\in n^{-1} \del\mathsf{R}_n$ and $v\in \mathfrak \del R$.  
	
	To state results on the limiting height function of random tilings, for any $x \in \mathbb{R}_{\ge 0}$ and $(s, t) \in \overline{\mathcal{T}}$ we denote the \emph{Lobachevsky function} $L: \mathbb{R}_{\ge 0} \rightarrow \mathbb{R}$ and the \emph{surface tension} $\sigma : \overline{\mathcal{T}} \rightarrow \mathbb{R}$ by 
	\begin{flalign}
		\label{sigmal} 
		L(x) = - \displaystyle\int_0^x \log |2 \sin z| \mathrm{d} z; \qquad \sigma (s, t) = \displaystyle\frac{1}{\pi} \Big( L(\pi (1-s)) + L (- \pi t) + L \big( \pi ( s + t) \big) \Big).
	\end{flalign}
	\noindent For any admissible $H:\mathfrak{R}\to \R$, we further define the \emph{entropy functional}
	\begin{flalign}
		\label{efunctionh} 
		\mathcal{E} (H) = \displaystyle\int_{\mathfrak{R}} \sigma \big( \nabla H (v) \big) \mathrm{d}v.
	\end{flalign}
	
	The following variational principle of \cite{VPT}  states that the height function associated with a uniformly random tiling of a sequence of domains converging to $\mathfrak{R}$ converges to the maximizer of $\mathcal{E}$ with high probability.

	\begin{lem}[{\cite[Theorem 1.1]{VPT}}]
		
		\label{hzh} 
		
		Let $\mathsf{R}_1, \mathsf{R}_2, \ldots \subset \mathbb{T}$ denote a sequence of tileable domains, with associated boundary height functions $\mathsf{h}_1, \mathsf{h}_2, \ldots $, respectively. Assume that they converge to a simply-connected set $\mathfrak{R} \subset \mathbb{R}^2$ with piecewise smooth boundary, and a boundary height function $h : \partial \mathfrak{R} \rightarrow \mathbb{R}$, respectively. Denoting the height function associated with a uniformly random tiling of $\mathsf{R}_n$ with boundary height function $\mathsf h_n$ by $\mathsf{H}_n$, we have for any constant $\e>0$,
		\begin{flalign*}
			\displaystyle\lim_{n \rightarrow \infty} \mathbb{P} \bigg( \displaystyle\max_{\mathsf{v} \in \mathsf{R}_n} \big| n^{-1} \mathsf{H}_n (\mathsf{v}) - H^* (n^{-1} \mathsf{v}) \big| > \varepsilon \bigg) = 0,
		\end{flalign*}  
	
		\noindent where $H^*$ is the unique maximzer of $\mathcal{E}$ on $\mathfrak{R}$ with boundary data $h$,
		\begin{flalign}
			\label{hmaximum}
			H^* = \displaystyle\argmax_{H \in \Adm (\mathfrak{R}; h)} \mathcal{E} (H).
		\end{flalign}
	\end{lem} 

	\noindent The fact that there is a unique maximizer described as in \eqref{hmaximum} follows from Proposition 4.5 of \cite{MCFARS}. 
 The region where $\nabla H^*\in \cT$ is called the \emph{liquid region} $\mathfrak{L} = \mathfrak{L} (\mathfrak{\fR}) \subset \mathfrak{R}$,
 	\begin{flalign}
		\label{al} 
		\mathfrak{L} = \big\{ v \in \mathfrak{R}: \nabla H^* (v) \in \mathcal{T} \big\},
	\end{flalign}
 where we expect to see all three types of lozenges. 

\subsection{Complex slope}\label{s:cslope}
	
 An important quantity that characterizes the limiting height function $H^*$ as in \eqref{hmaximum} is the \emph{complex slope} $f^*: \mathfrak{L} \rightarrow \mathbb{H}^-$. For  any $(x,t) \in \mathfrak{L}$, $f^*(x,t) \in \mathbb{H}^-$ is the unique complex number satisfying 
\begin{flalign}
		\label{fh}
		\arg^* f^*(x,t) = - \pi \partial_x H^* (x,t), \qquad \arg^* \big( f^*(x,t) + 1 \big) = \pi \partial_t H^* (x,t);
	\end{flalign}
	see \Cref{slope1} for a depiction. Hereafter, for any $z \in \R\cup\mathbb{H}^-\setminus \{0\}$, we set $\arg^* z = \theta \in [-\pi, 0]$ to be the unique number in $[-\pi, 0]$ satisfying $e^{-\mathrm{i} \theta} z \in \mathbb{R}_{> 0}$. Note that we interpret $1 - \partial_x H^* (x,t)$ and $-\partial_t H^* (x,t)$ as the approximate proportions of types $1$ tiles and type $2$ tiles around $(nx,nt) \in \mathsf{R}_n$, respectively (which follows from the definition of the limiting height function in \Cref{ssec:hf}).   Below we also denote $f^*_t (x) = f^* (x, t)$ for any $(x, t) \in \mathfrak{L}$.
 
	\begin{figure}[!ht]
	\begin{center}
	 \includegraphics[scale=0.14,trim={0cm 8cm 0 10cm},clip]{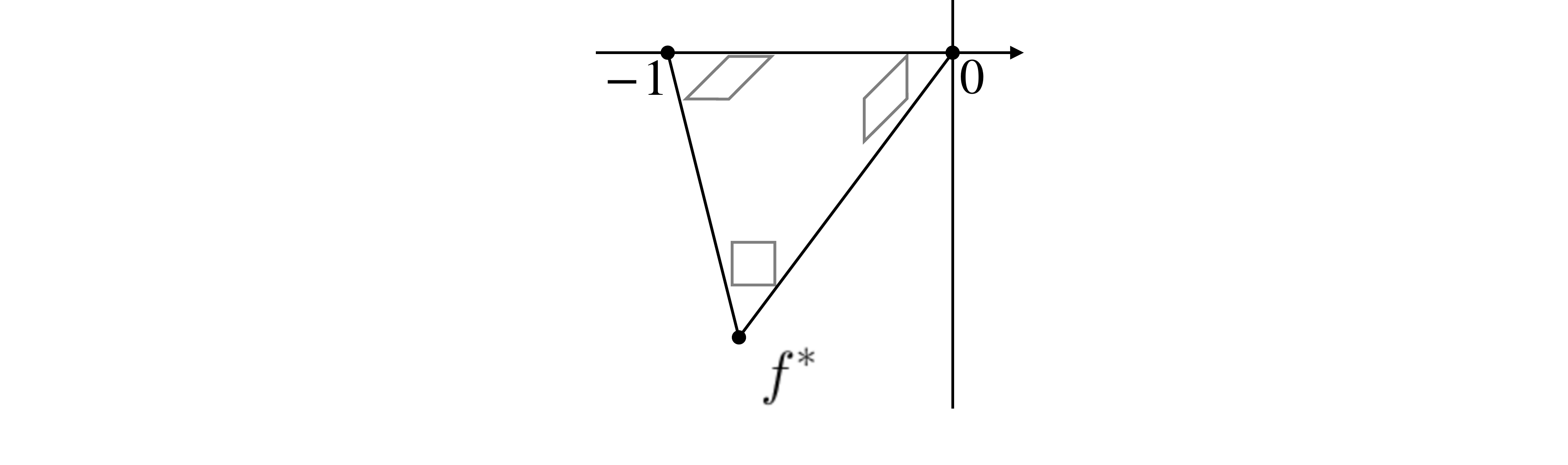}
	 \caption{Shown above the complex slope $f^* = f^* (x,t)$.}
	 \label{slope1}
	 \end{center}
	 \end{figure}	
		  The following result from \cite{LSCE} indicates that the complex slope $f^*$ satisfies the complex Burgers equation in the liquid region.
	  \begin{prop}[{\cite[Theorem 1]{LSCE}}]
	 	
	 	\label{fequation}
	 	
	 	For any $(x, t) \in \mathfrak{L}$, we have that
	 	\begin{flalign}
	 		\label{ftx}
	 		\partial_t f^*_t (x) + \partial_x f^*_t (x) \displaystyle\frac{f^*_t (x)}{f^*_t (x) + 1} = 0.
	 	\end{flalign} 
	 	 	
	 	\end{prop} 

    \subsection{Polygonal domains}
This paper concerns tilings of polygonal domains, which we describe now.
 \begin{definition} 
 	\label{p} 
 	An open set $\mathfrak{P} \subset \mathbb{R}^2$ is \emph{polygonal} if its boundary $\partial \mathfrak{P}$ consists of a finite union of line segments, each of which is parallel to an axis of $\mathbb{T}$. 
  For the rest of this paper, whenever we take a polygonal set, it is always assumed to be simply-connected.
  The set is \emph{rational polygonal} if, in addition, every endpoint of each segment in $\partial\fP$ is a rational point.
  We note that being rational is equivalent to that
  there exists some $n_0\in\N$ with $n_0\mathfrak{P}$ being a tileable domain. 
 \end{definition}
From this definition, for any $n\in n_0\N$,  $\mathsf{P}\equiv \mathsf{P}_n=n \mathfrak{P}$ is a tileable domain, and is therefore associated with a (unique up to a global shift) boundary height function $\mathsf{h} = \mathsf{h}_n$. 
  We set $h: \partial \mathfrak{P} \rightarrow \mathbb{R}$ by $h (v) = n^{-1} \mathsf{h} (nv)$ for each $v \in \partial \mathfrak{P}\cap n^{-1}\mathbb{T}$, and linearly interpolating between points on $n^{-1}\mathbb{T}$.
  It is straightforward to check that this function $h$ is determined by $\mathfrak{P}$ (i.e., independent of $n$), up to a global shift.

  Let $H^*$ be the limiting height function of uniformly random lozenge tiling of $\fP$, as defined in \eqref{hmaximum}. We recall $\cT$ from \eqref{t} and the liquid region $\mathfrak{L} = \mathfrak{L} (\mathfrak{\fP}) \subset \mathfrak{P}$ from \eqref{al}. We denote the \emph{arctic boundary} $\mathfrak{A} = \mathfrak{A} (\mathfrak{P}) \subset \overline{\mathfrak{P}}$ by
	\begin{flalign}
		\label{al2} 
		\mathfrak{L} = \big\{ v \in \mathfrak{P}: \nabla H^* (v) \in \mathcal{T} \big\}, \quad \text{and} \quad \mathfrak{A} = \del \fL.
	\end{flalign}
The liquid region and arctic boundary are determined by the set $\mathfrak{P}$, and have the following properties.

\begin{lem}[{\cite{LSCE,DMCS}}]
	
	\label{pla}
Assume that $\mathfrak{P}$ is a rational polygonal set, then the followings hold.
	\begin{enumerate}	
		\item 
  For the maximizer $H^* = \argmax_{H \in \Adm(\mathfrak{R}; h)} \mathcal{E} (H)$, which is determined by $\mathfrak{P}$ up to a global shift,
  $\nabla H^*$ is piecewise constant on $\mathfrak{P} \setminus \mathfrak{L} (\mathfrak{P})$, taking values in $\big\{ (0, 0), (1, 0), (1, -1) \big\}$. 
		\item The arctic boundary $\mathfrak{A} (\mathfrak{P})$ is an algebraic curve, and its singularities are all either ordinary cusps or tacnodes.
  
	\end{enumerate}
	\end{lem} 
These results are proved in \cite{LSCE,DMCS} and quoted in this form as \cite[Lemma 2.3]{aggarwal2021edge}. The first statement is by \cite[Theorem 1.9]{DMCS}, and the second  statement is by \cite[Theorem 1.2, Theorem 1.10]{DMCS} (see also \cite[Theorem 2, Proposition 5]{LSCE}).

For polygonal set, it was proved in \cite[Theorem 1.2, Theorem 1.5]{DMCS} that the complex slope $(x,t)\mapsto f_t^*(x)$ extends to the arctic boundary. More precisely, the complex slope extends to a continuous function from $\overline{\cL(\fP)}$ to the one point compactification $\bC\cup\{\infty\}$. For any $(x, t) \in \fA$, $f^*_t(x)\in \bR\cup \{\infty\}$ and the slope of the arctic boundary at $(x,t)$ is given by
\begin{equation}\label{e:slope}
\frac{f^*_t(x)+1}{f^*_t(x)}.
\end{equation} 

	 For a nonsingular point in $\mathfrak{A}$, we call it a \emph{tangency location} of $\mathfrak{A}$, if the tangent line to $\mathfrak{A}$ has slope in $\{ 0, 1, \infty \}$. 
We need to impose the following assumptions of a rational polygonal set $\mathfrak{P}$, on its arctic boundary.
	\begin{assumption} 
		\label{pa} 
		For a rational polygonal set $\mathfrak{P} \subset \mathbb{R}^2$, assume the following four properties hold. 
		\begin{enumerate} 
			\item The arctic boundary $\mathfrak{A} = \mathfrak{A} (\mathfrak{P})$ has no tacnode singularities. 
			\item No cusp singularity of $\mathfrak{A}$ is also a tangency location of $\mathfrak{A}$. 
			\item  There exists an axis $\ell$ of $\mathbb{T}$ such that any line connecting two distinct cusp singularities of $\mathfrak{A}$ is not parallel to $\ell$. 
			\item  Any intersection point between $\mathfrak{A}$ and $\partial \mathfrak{P}$ must be a tangency location of $\mathfrak{A}$. Moreover, $\nabla H^*(x,t)$ is continuous at any point on $\mathfrak{A}$ that is not a tangency location.
		\end{enumerate}	
	\end{assumption} 
As discussed in \cite[Remark 2.8]{aggarwal2021edge}, these assumptions are believed to hold for a generic rational polygonal set with a given number of sides, as violating each assumption is equivalent to that the side lengths satisfy a certain algebraic equation; but here we do not provide a rigorous proof of this.
 
	\subsection{Pearcey process}  \label{s:pp}
As another preparation for our main results, we formally define the Pearcey process $\cP$ as a time-dependent random collection of infinitely many particles on $\R$, with the multi-time gap probability given by the Fredholm determinant
\[
\PP[\cP(t_i)\cap E_i = \emptyset, \forall 1\le i \le m]
= \det(I-\chi K^{\textnormal{Pearcey}})_{L^2(\{t_1,\ldots,t_m\}\times \R)}
\]
for any $t_1<\cdots <t_m$ and finite unions of intervals $E_1, \ldots, E_m$. Here, $\chi$ is the projection operator, acting as $\chi f(t_i,x)=\don[x\in E_i]f(t_i,x)$ for $f:\{t_1,\ldots, t_m\}\times \R\to\R$,
and $K^{\textnormal{Pearcey}}$ is the integral operator, acting as
\[
K^{\textnormal{Pearcey}}f(t_i,x)=\sum_{j=1}^m \int K^{\textnormal{Pearcey}}(t_i,x;t_j,y) f(t_j,y) dy,
\]
with the extended Pearcey kernel
\begin{multline}  \label{eq:pearcey}
K^{\textnormal{Pearcey}}(s,x;t,y)\\
= -\frac{\mathds{1}[s<t]}{\sqrt{2\pi(t-s)}}\exp\Big(-\frac{(x-y)^2}{2(t-s)}\Big)
+\frac{1}{(2\pi\i)^2} \iint \frac{\rd\pcz \rd\pcw}{\pcz-\pcw}\exp\Big(\frac{-\pcz^4+\pcw^4}{4}+\frac{t\pcz^2-s\pcw^2}{2}-y\pcz + x\pcw\Big),
\end{multline}
for any $s,x,t,y\in \R$; see e.g., \cite{adler2010universality}.
The $\pcz$ contour is taken to be the straight vertical line $\Re(\pcz)=0$ traversed upwards (from $-\infty\i$ to $\infty\i$), and the $\pcw$ contour contains the straight lines from $\infty e^{\pi\i/4}$ and $-\infty e^{\pi\i/4}$ to $0$, and from $0$ to $\infty e^{-\pi\i/4}$ and $-\infty e^{-\pi\i/4}$.

 \subsection{Main results}

To state our result on the Pearcey process in tiling, we need to define the scaling parameters.
\begin{definition}	
		\label{sr} 
		For a rational polygonal set $\fP$, fix a cusp point $(x_c, t_c) \in \mathfrak{A} = \mathfrak{A} (\mathfrak{P})$ that is not a tangency location. We say that $(x_c,t_c)$ is \emph{upward oriented}, if the slope of the tangent line through $(x_c,t_c)$ is in $(1,\infty)$, and there exist $(\mathfrak{r}, \mathfrak{q}) = \big( \mathfrak{r} (x_c, t_c; \mathfrak{A}), \mathfrak{q} (x_c, t_c; \mathfrak{A}) \big) \in \mathbb{R}^2$  so that 
		\begin{flalign}
			\label{xyql}
			x - x_c = \frac{ (t - t_c)}{\mathfrak{r}} \pm \frac{2\mathfrak{q}}{3\sqrt 3} (t_c - t)^{3/2} + \mathcal{O} \big( (t_c - t)^2 \big),
		\end{flalign} 
 for all $(x, t) \in \mathfrak{A}$ in a sufficiently small neighborhood of $(x_c, t_c)$. 
We note that these can always be achieved by rotating $\fP$.
We call $(\fr, \fq)$ the \emph{curvature parameters} $(\mathfrak{r}, \mathfrak{q})$ associated with $(x_c, t_c)$.	
Note that $\fr=\frac{f^*_{t_c}(x_c)+1}{f^*_{t_c}(x_c)}$, according to \eqref{e:slope}.
\end{definition}

Our main cusp universality result is as follows.
\begin{thm}  \label{thm:univ}
	Take a rational polygonal set $\mathfrak{P} \subset \mathbb{R}^2$ satisfying \Cref{pa}, and let $H^*$ be a limiting height function of it. Fix some point $(x_c, t_c)$ that is a cusp location of $\mathfrak{A} (\mathfrak{P})$. Assume (without loss of generality) that this cusp is upward oriented as stated in \Cref{sr}. Denote the associated curvature parameters by $(\mathfrak{r}, \mathfrak{q})$, with $\fr\in (1,\infty)$ and $\fq>0$. 

 Take $n\in\N$ such that $\mathsf{P}=n\mathfrak{P}$ is a tileable domain.
	Let $\mathscr{M}$ denote a uniformly random tiling of $\mathsf{P}$. It is associated with a (random) family of non-intersecting Bernoulli paths (as defined in Section \ref{ssec:nbp}), which we denote as a function $\sfM$ from $\Z$ to the set of finite subsets of $\Z$ (by ignoring the indices of paths). Then as $n\to\infty$, the process
		\begin{flalign} 
			\label{xit1}
	t\mapsto \frac{\sfM(\lfloor\lsca\cust - \sqrt{\fr-1}\lsca^{1/2}t/(\fr \fq)\rfloor)-\lsca\cusx+\sqrt{\fr-1}\lsca^{1/2}t/(\fr^2 \fq)}{(\fr-1)^{3/4}\lsca^{1/4}/\sqrt{\fq\fr^3}} 
		\end{flalign}
converges to the Pearcey process $\mathcal{P}$, in the sense of convergence as point processes, in any set of the form $\{t_1,\ldots, t_m\}\times E$ with $t_1<\cdots <t_m$ and $E$ being a compact interval.
\end{thm}
\begin{rem}\label{rem:rescale-to-Pearcey}
Here, we have used the Pearcey process whose boundary is like $x=2(t/3)^{3/2}$ (see e.g.~\cite{adler2011pearcey}). In our setting, the arctic boundary around the cusp $(x_c,t_c)$ is parametrized by \eqref{xyql}. Hence, we need to rescale it to $x=2(t/3)^{3/2}$. Each path in the Pearcey process locally behaves like a Brownian motion. Locally around the cusp, the non-intersecting Bernoulli paths have drift $1/\fr$, so each step has variance $(1/\fr)(1-1/\fr)$. 
To make them behave like Brownian motions without drift, we need to do the following Brownian scaling 
\begin{align}\label{e:Rescale}
    \widehat \sfM=a\frac{(\sfM-nx_c)-(\sft-nt_c)/\fr}{\sqrt{(1/\fr)(1-1/\fr)}},\quad
    \widehat \sft=a^2 (nt_c-\sft),\quad a=\sqrt{\frac{\fq \fr}{\sqrt{\fr-1}}},
\end{align}
where $a$ is determined by $\widehat\sfM/n=2(\widehat\sft/3n)^{3/2}$. 
To get the Pearcey process, we further rescale the space by $n^{-1/4}$ and time by $n^{-1/2}$ so that the gaps between two paths are of order one:
\begin{align}
    \widetilde \sfM=n^{-1/4}\widehat \sfM, \quad \widetilde \sft=n^{-1/2}\widehat \sft,
\end{align}
which leads to \eqref{xit1}.
\end{rem}

\subsubsection{Universality of non-intersecting Bernoulli random walks (NBRW)}  \label{ssec:unbrw}

As already indicated, in proving Theorem \ref{thm:univ}, a key step is to understand the universality of the Pearcey process in the related model of NBRW, which we now define formally. 

\medskip
\noindent\textbf{NBRW as a Markov chain.} 
The NBRW $\sfA:\qq{0,\infty}\to \Z^{\qq{-M,N}}$ that we will consider can be defined as a Markov chain on time $\qq{0,\infty}$, with state space being the Weyl chamber
\[
\left\{\{\sfx_i\}_{i\in\qq{-M,N}}\in\Z^{\llbracket -M, N\rrbracket}: \sfx_{-M}<\cdots<\sfx_N \right\}
\]
for some $M,N\in \N$. The transition probability is given as follows. Take $\beta\in (0,1)$, which is the \emph{drift parameter}. For any $\sft\in\qq{0, \infty}$, let $\PP\left[\sfA(\sft+1)=\{\sfy_i\}_{i\in\qq{-M,N}} \mid \sfA(\sft)=\{\sfx_i\}_{i\in\qq{-M,N}}\right]$ equal
\[
(1-\beta)^{M+N+1} \prod_{-M\le i \le N} \Big(\frac{\beta}{1-\beta}\Big)^{\sfy_i-\sfx_i} \prod_{-M\le i<j\le N}\frac{(\sfy_i-\sfy_j)}{(\sfx_i-\sfx_j)},
\]
when each $\sfy_i-\sfx_i\in\{0,1\}$; and $0$ otherwise.
Alternatively, $\sfA$ can be defined as a collection of $M+N+1$ independent Bernoulli$(\beta)$ random walks on $\qq{0,\infty}$, conditioned on never intersect. It can also be viewed as a discrete analog of the \emph{Dyson Brownian motion} with parameter $2$.

With the relation between tilings and non-intersecting Bernoulli paths given in Section \ref{ssec:nbp}, we can view NBRW on $\llbracket 0,\infty\rrbracket$ as a random tiling of the upper-half plane, where the boundary height function on the horizontal axis is in correspondence with the initial configuration $\sfA(0)$.

\medskip

We next describe a universal convergence of NBRW to the Pearcey process. Roughly speaking, it says that if the initial configuration of NBRW contains two separated groups of particles, with the gap between them and their density growth being of `proper' orders, then the Pearcey process appears when these two groups of particles merge together.

We start with the setup. Fix any $\phi \in (0, 1/2)$. Let $\errh>0$ be a small enough constant (depending on $\phi$), and then $\err>0$ be a small enough constant (depending on $\phi$ and $\errh$).
To state the asymptotic result, we consider a sequence of NBRWs:
for each integer $\lsca>0$, we consider NBRW $\sfA$ on $\qq{0,\infty}$ with drift parameter $\beta\in (\phi, 1-\phi)$ and (possibly random) initial condition $\sfA(0)=\{\sfd_i\}_{i\in \llbracket -M, N\rrbracket}$ for some $M, N \asymp \lsca$.
We assume that $\{\sfd_i\}_{i\in \llbracket -M, N\rrbracket}$ (with scaling $n^{-1}$) can be \emph{approximated by the quantiles of a density function $\den:\R\to[0,1]$ up to order $\lsca^{-1+\err}$}, and $\den$ satisfies certain \emph{cusp growth at scale $\lsca$ and up to distance $\sct$, with $\lsca^{-1/2+\errh}<\sct\lesssim 1$}, in the sense to be specified in Assumption \ref{ass:dtsncusp} below.
Let $\cusx$, $\cust$, $A$, $B$ be real numbers determined by $\den$ and $\beta$, via Lemma \ref{lem:xtzdef} and \eqref{eq:defA} below (in particular, we have $\cust\asymp \sct$).
We remark that all of $\beta$, $M$, $N$, $\{\sfd_i\}_{i\in \llbracket -M, N\rrbracket}$, $\den$, $\sct$, $\cusx$, $\cust$, $A$, $B$ can depend on $\lsca$.

\begin{thm}  \label{thm:univNBRWlim}
As $\lsca\to \infty$, the process
\[
\cP^{\textnormal{Bernoulli}}:t\mapsto \frac{\sfA(\lfloor\lsca\cust - 2A^{1/2}B(1-B)\lsca^{1/2}t\rfloor)-\lsca\cusx}{\sqrt{2}A^{1/4}B(1-B)\lsca^{1/4}} +\sqrt{2}A^{1/4}B\lsca^{1/4}t ,
\]
converges to the Pearcey process, in the sense of convergence as point processes, in any set of the form $\{t_1,\ldots, t_m\}\times E$ with $t_1<\cdots <t_m$ and $E$ being a compact interval.
\end{thm}

We note that this is a `normal and smaller distance' result, in the sense that while the Pearcey process has temporal and  spatial scalings of order $n^{1/2}\times n^{1/4}$, the time when it appears is of order $\lsca\cust\asymp\lsca\sct$, which is $\gtrsim n^{1/2+\errh}$ and $\lesssim n$.
We cannot expect a Pearcey process of order $n^{1/2}\times n^{1/4}$ at any time much beyond this window: on one hand, the time to the boundary must be much larger than the temporal scaling $\lsca^{1/2}$; on the other hand, at any time much larger than $\lsca$, the spatial fluctuation of the paths should be much larger than $\lsca^{1/4}$ around a cusp.
Therefore, Theorem \ref{thm:univNBRWlim} covers almost the whole possible time window where a Pearcey process of order $n^{1/2}\times n^{1/4}$ could appear.

Theorem \ref{thm:univNBRWlim} is an immediate consequence of Proposition \ref{prop:Kconvconj} below, which gives a stronger pre-limit estimate of the NBRW determinantal kernel at the cusp.

\subsubsection{On the continuous theory of the Pearcey process and convergence}
\label{s:dccp}
Intuitively, for the non-intersecting Bernoulli paths from tiling or NBRW around a cusp, they should converge to a family of continuous processes, under e.g.~the topology of uniform convergence in any compact interval.
This limiting family should be a continuous path version of the Pearcey process $\cP$, which has been expected to exist (see e.g.~\cite{Tracy2006}, at the end of the introduction), and should have Brownian Gibbs property, as that of the Airy line ensemble given in \cite{PLE} (see e.g.~\cite[Problem 2.34]{AIM}).
Such an object could be called the `Pearcey line ensemble' (PLE), following the naming convention of the Airy and Bessel line ensembles, constructed in \cite{PLE} and \cite{wu2021bessel}. 
However, as far as we know, such a construction has not yet been accomplished in the literature, despite that the Pearcey limit has been established for various probabilistic models, such as random matrices, non-intersecting Brownian motions, and tilings, as stated in the introduction.
Compared to the Airy and Bessel cases, one additional difficulty is that paths in the PLE are indexed by $\Z$ rather than $\N$.
This causes a labeling issue: in Airy or Bessel, the point process distribution at a fixed time gives the distribution of the continuous paths at this time, since the $i$-th highest point must be in the $i$-th path. However, for Pearcey, given the point process at one time, additional information is needed to determine which points correspond to the paths that would $\to \infty$ or $-\infty$ as $t\to\infty$.

In terms of the convergence to the Pearcey process, all the proven results are (more or less equivalently) in the sense of convergence as point processes at finitely many times, as our Theorems \ref{thm:univ}, \ref{thm:univNBRWlim}; and this is what one can hope for without having the PLE defined.
We expect that once the PLE is built, there should be a general theorem upgrading all such point process convergence to uniform in compact convergence, as long as the prelimiting model has some local Gibbs properties (such as \Cref{lem:Gibbstiling} below for tiling).
For the Airy line ensemble such a theorem exists; see \cite[Theorem 4.2]{DNV}.

\section{Monotonicity and Gibbs properties}\label{s:NBwe}

In the study of uniformly random tiling and related models of random non-intersecting paths, 
an important and widely used \emph{monotonicity} property roughly says that: for two random configurations, if they are `close to each other' at the boundary of a region, they should also be `close to each other' inside the region.
It has various versions in the literature (see e.g.\ \cite[Lemma 18]{LSRT}, \cite[Lemmas 2.6 and 2.7]{PLE}, \cite[Lemmas 2.6 and 2.7]{corwin2016kpz}, and \cite[Lemma 5.6]{dimitrov2021characterization}). Here, we record some that will be used later.

The first one is for random non-intersecting Bernoulli paths. To proceed, we need some more notations. Take a family of non-intersecting Bernoulli paths $\sfB=\{\sfb_i\}_{i\in\qq{1,m}}$, consisting of $m$ paths, with each $\sfb_i$ having the same time span $\qq{0,\sfr}$.
Given functions $\mathsf{f}, \mathsf{g}: \qq{0, \sfr} \rightarrow \mathbb{R}$, we say that $\sfB$ has $\mathsf{f}$ and $\mathsf{g}$ as \emph{boundary conditions} if $\mathsf{f} (\sft) \le \sfb_i (\sft) \le \mathsf{g} (\sft)$ for each $\sft \in \qq{0, \sfr}$ and $i\in\qq{1,m}$. We refer to $\mathsf{f}$ and $\mathsf{g}$ as the \emph{left boundary} and the \emph{right boundary}, respectively, and allow $\mathsf{f}$ and $\mathsf{g}$ to be $-\infty$ or $\infty$. We say that $\sfB$ has \emph{entrance condition} $\mathsf{d} = (\mathsf{d}_1, \mathsf{d}_{2}, \ldots , \mathsf{d}_m)$ and \emph{exit condition} $\mathsf{e} = (\mathsf{e}_1, \mathsf{e}_{2}, \ldots , \mathsf{e}_m)$ if $\sfB (0) = \mathsf{d}$ and $\sfB (\sfr) = \mathsf{e}$. There is a finite number of non-intersecting Bernoulli paths with given entrance, exit, and (possibly infinite) boundary conditions.

In what follows, for any functions $\mathsf{f}, \mathsf{f}': \qq{0,\sfr} \rightarrow \mathbb{R}$ we write $\mathsf{f} \le \mathsf{f}'$ if $\mathsf{f} (\sft) \le \mathsf{f}' (\sft)$ for each $\sft \in \qq{0,\sfr}$ and denote $|\mathsf{f}-\mathsf{f}'|=\max_{\sft\in \llbracket 0, \sfr\rrbracket}|\mathsf{f}(\sft)-\mathsf{f}'(\sft)|$. Similarly, for any $m$-tuples $\mathsf{d} = (\mathsf{d}_1, \mathsf{d}_{2}, \ldots , \mathsf{d}_m) \in \mathbb{R}^m$ and $\mathsf{d}' = (\mathsf{d}_1', \mathsf{d}_{2}', \ldots , \mathsf{d}_m') \in \mathbb{R}^m$, we write $\mathsf{d} \le \mathsf{d}'$ if $\mathsf{d}_i \le \mathsf{d}_i'$ for each $i \in \qq{1,m}$ and denote $|\mathsf{d}-\mathsf{d}'|=\max_{i\in \llbracket 1, m\rrbracket}|\mathsf{d}_i - \mathsf{d}_i'|$.

	\begin{lem}
		
		\label{comparewalks}

		Fix integers $\sfr,m\ge 1$, functions $\mathsf{f}, \mathsf{f}', \mathsf{g}, \mathsf{g}' : \qq{0,\sfr} \rightarrow \mathbb{R}$, and $m$-tuples $\mathsf{d}, \mathsf{d}', \mathsf{e}, \mathsf{e}'$ with coordinates indexed by $\qq{1,m}$. Let $\sfQ=\{\sfq_i\}_{i\in\qq{1,m}}$ denote uniformly random non-intersecting Bernoulli paths with boundary, entrance, and exit conditions given by $\mathsf{f}$, $\mathsf{g}$, $\mathsf{d}$, $\mathsf{e}$; define $\sfQ'=\{\sfq'_i\}_{i\in\qq{1,m}}$ similarly, but using $\mathsf{f}'$, $\mathsf{g}'$, $\mathsf{d}'$, $\mathsf{e}'$ instead. If $|\mathsf{f} - \mathsf{f}'|\le K$, $|\mathsf{g} - \mathsf{g}'|\le K$, $|\mathsf{d} - \mathsf{d}'|\le K$, and $|\mathsf{e} - \mathsf{e}'|\le K$, for some $K>0$, then there exists a coupling between $\sfQ$ and $\sfQ'$ such that $|\mathsf{q}_i  - \mathsf{q}'_i|\le K$ almost surely for each $i \in \qq{1,m}$.
	\end{lem}
This lemma is in the spirit of \cite[Lemmas 2.6 and 2.7]{PLE} and can be proved using the same idea of constructing the coupling using the Glauber dynamics of the paths. 
We give a sketch here for completeness. 
\begin{proof}[Proof of Lemma \ref{comparewalks}]
We introduce a continuous-time Markovian dynamic on the non-intersecting Bernoulli paths (which is the Glauber dynamics). We write the non-intersecting Bernoulli paths at time $\tau$ as $\sfY_\tau=(\{\sfy_{i, \tau}\}_{i\in\qq{1,m}})_\tau$ and $\sfY'_\tau=(\{\sfy'_{i, \tau}\}_{i\in\qq{1,m}})_\tau$, with the time $0$ configurations $\sfY_0$ and $\sfY'_0$ being the lowest possible non-intersecting Bernoulli paths with boundary, entrance, and exit conditions being $\mathsf{f}$, $\mathsf{g}$, $\mathsf{d}$, $\mathsf{e}$, and $\mathsf{f}'$, $\mathsf{g}'$, $\mathsf{d}'$, $\mathsf{e}'$, respectively.
 It is clear that such lowest configurations exist, are unique, and satisfy $|\sfy_{i,0}'-\sfy_{i,0}|\le K$ for all $i\in\qq{1,m}$. 
For simplicity of notations, denote $\sfy_{0,\tau}=\sff$, $\sfy_{m+1,\tau}=\sfg$, $\sfy'_{0,\tau}=\sff'$, $\sfy'_{m+1,\tau}=\sfg'$, for any $\tau\ge 0$.
The dynamics are as follows: for each $\sft\in \qq{1,\sfr-1}$, $i\in \qq{1,m}$ and $e\in \{1,-1\}$, there is an independent exponential clock which rings at rate $1$.
If the clock labeled $(\sft, i, e)$ rings at time $\tau$, one attempts to set $\sfy_{i,\tau}(\sft)=\sfy_{i,\tau-}(\sft)+e$ (where $\sfy_{i,\tau-}(\sft)$ is the limit of $\sfy_{i,\tau'}(\sft)$ as $\tau'\to \tau$ from the left). This setting is only successful if $\sfy_{i,\tau}$ remains a Bernoulli path, and the condition of non-intersection with $\sfy_{i-1,\tau}$ and $\sfy_{i+1,\tau}$ is not broken.
One also attempts to set $\sfy'_{i,\tau}(\sft)=\sfy'_{i,\tau-}(\sft)+e$, and
the same conditions apply.

The first key fact is that the maximum difference $\max_{i\in \qq{0,m+1}}|\sfy_{i,\tau}-\sfy'_{i,\tau}|$ is non-increasing in $\tau$. As a consequence, for all $\tau \ge 0$, $|\sfy_{i,\tau}-\sfy'_{i,\tau}|\le K$
for each $i\in \qq{1,m}$. The second key fact is that the distributions of these non-intersecting Bernoulli paths converge to the invariant measures for this Markovian dynamics, which are given by the non-intersecting Bernoulli paths randomly sampled under the uniform measure on the set of paths with prescribed entrance, exit, and boundary conditions. This fact is true since these dynamics have finite state spaces which are irreducible with the obvious invariant measures. Then, Lemma \ref{comparewalks} follows immediately from these two facts.

For the rest of this proof, we prove the first key fact above, i.e., the maximum difference is non-increasing in time. 
Suppose that a clock labeled $(\sft^*, i^*, e)$ rings at some time $\tau>0$.
We denote by $\{\sfy_{i,\tau-}\}_{i\in\qq{0,m+1}}$, $\{\sfy'_{i,\tau-}\}_{i\in\qq{0,m+1}}$ the paths before the ringing, and $\{\sfy_{i,\tau}\}_{i\in\qq{0,m+1}}$, $\{\sfy_{i,\tau}\}_{i\in\qq{0,m+1}}$ the paths after the ringing.
If $(\sft^*, i^*)$ is not an $\argmax$ of $|\sfy_{i,\tau-}(\sft)-\sfy'_{i,\tau-}(\sft)|$ for $\sft\in\qq{0,\sfr}$ and $i\in\qq{0,m+1}$, then the maximum difference is obviously non-increasing at the instant $\tau$. Hence, below we assume that $|\sfy_{i,\tau-}(\sft)-\sfy'_{i,\tau-}(\sft)|$ achieves maximum at $(\sft^*, i^*)$.

Without loss of generality, we assume that $\sfy_{i^*,\tau-}(\sft^*)-\sfy'_{i^*,\tau-}(\sft^*)\ge 0$ and $e=1$.
It suffices to prove that the following scenario is impossible: $\sfy_{i^*,\tau}(\sft^*) = \sfy_{i^*,\tau-}(\sft^*)+1$ 
    and
    $\sfy'_{i^*,\tau}(\sft^*)
    = \sfy'_{i^*,\tau-}(\sft^*)$. 
    Assume the contrary, there are two cases: 
\begin{itemize}[leftmargin=*]
    \item[(i)] $\sfy'_{i^*,\tau-}(\sft^*+1)=\sfy_{i^*,\tau-}'(\sft^*)$ or $\sfy'_{i^*,\tau-}(\sft^*-1)=\sfy'_{i^*}(\sft^*)-1$. Then, we have $\sfy_{i^*,\tau-}(\sft^*+1)-\sfy'_{i^*,\tau-}(\sft^*+1)>\sfy_{i^*,\tau-}(\sft^*)-\sfy_{i^*,\tau-}'(\sft^*)$ or $\sfy_{i^*,\tau-}(\sft^*-1)-\sfy'_{i^*,\tau-}(\sft^*-1)>\sfy_{i^*,\tau-}(\sft^*)-\sfy_{i^*,\tau-}'(\sft^*)$, because we must have $\sfy_{i^*,\tau-}(\sft^*+1)=\sfy_{i^*,\tau-}(\sft^*)+1$ and $\sfy_{i^*,\tau-}(\sft^*)=\sfy_{i^*,\tau-}(\sft^*-1)$ in order for the update $\sfy_{i^*,\tau}(\sft^*) = \sfy_{i^*,\tau-}(\sft^*)+1$ to be permissible.
    This contradicts the assumption that $(\sft^*, i^*)$ is an $\argmax$ of the difference.
    \item[(ii)] $\sfy'_{i^*,\tau-}(\sft^*+1)=\sfy'_{i^*,\tau-}(\sft^*)+1$ and  $\sfy'_{i^*,\tau-}(\sft^*-1)=\sfy'_{i^*,\tau-}(\sft^*)$. In this case, since we have assumed that $\sfy'_{i^*,\tau}(\sft^*)=\sfy'_{i^*,\tau-}(\sft^*)$, i.e., the attempt to set $\sfy'_{i^*,\tau}(\sft^*)=\sfy'_{i^*,\tau-}(\sft^*)+e$ fails, we must have $\sfy'_{i^*+1,\tau-}(\sft^*)=\sfy'_{i^*,\tau-}(\sft^*)+1$. Moreover, since we have assumed that $\sfy_{i^*,\tau}(\sft^*) = \sfy_{i^*,\tau-}(\sft^*)+1$, we must have \[\sfy_{i^*+1,\tau-}(\sft^*)= \sfy_{i^*+1,\tau}(\sft^*) \ge \sfy_{i^*,\tau}(\sft^*)+1 = \sfy_{i^*,\tau-}(\sft^*)+2.\]
    This leads to $\sfy_{i^*+1,\tau-}(\sft^*)-\sfy'_{i^*+1,\tau-}(\sft^*)>\sfy_{i^*,\tau-}(\sft^*)-\sfy'_{i^*,\tau-}(\sft^*)$, which again contradicts the assumption that $(\sft^*, i^*)$ is an $\argmax$ of the difference.
\end{itemize}
Putting these cases together yields the first key fact, thereby the conclusion follows.
\end{proof}
We will also use the following version of monotonicity, in terms of the height function of tiling. For this purpose, we define uniformly random tilings on general subsets of $\R^2$, but with given boundary functions, in the sense of a uniformly chosen height function.

\begin{definition}  \label{defn:uhf}
Take any compact set $\sfR\subset\R^2$ with piecewise smooth boundary, and a function $\sfh:\partial\sfR\to\R$. If there exists a tileable domain $\sfR_+$ containing $\sfR$, and a tiling of $\sfR_+$ whose height function on $\partial \sfR$ equals $\sfh$, we call $\sfh$ a \emph{plausible boundary height function of $\sfR$}. 
In this case, there must be finitely many such height functions of $\sfR_+$, and for a uniformly chosen one, we call its restriction to $\sfR$ the \emph{uniformly random height function of $\sfR$ with boundary $\sfh$}. By the Gibbs property in \Cref{lem:Gibbstiling} below, it is straightforward to check that this uniformly chosen height function is independent of the choice of $\sfR_+$.
\end{definition}
 \begin{lem}[\protect{\cite[Lemma 18]{LSRT}}]  \label{lem:hfmon}
 Consider a compact set $\sfR_1\subset\R^2$ with piecewise smooth boundary, and its translation $\sfR_2=\sfR_1+v_0$ for some $v_0\in\R^2$.
 Take plausible boundary height functions $\sfh_1:\partial\sfR_1\to\R$ and $\sfh_2:\partial\sfR_2\to\R$.
 Let $\sfH_1$ and $\sfH_2$ be uniformly random height functions of $\sfR_1$ and $\sfR_2$ with boundaries $\sfh_1$ and $\sfh_2$, respectively.
 If $\sfh_1\le \sfh_2(\cdot+v_0)$, then there exists a coupling between $\sfH_1$ and $\sfH_2$, such that $\sfH_1\le \sfH_2(\cdot+v_0)$ almost surely.
 \end{lem}
We note that \cite[Lemma 18]{LSRT} is proved in the setting of random domino tiling, but the arguments carry over to lozenge tiling verbatim.

Finally, we record the Gibbs property for uniformly random tilings here, for the convenience of later reference. It is directly implied by the definition of uniformly random tilings.
\begin{lem}  \label{lem:Gibbstiling}
Take compact sets $\sfR, \sfR'\subset \R^2$ with piecewise smooth boundaries,
such that $\sfR\subset\sfR'$. Take plausible boundary height functions $\sfh:\partial \sfR\to\R$ and $\sfh':\partial\sfR'\to\R$, and let $\sfH$ and $\sfH'$ be uniformly random height functions of $\sfR$ and $\sfR'$ with boundaries $\sfh$ and $\sfh'$, respectively.
Consider the event where the restriction of $\sfH'$ on $\partial\sfR$ equals $\sfh$.
Suppose that this event happens with positive probability. Then, conditioning on this event, the restriction of $\sfH'$ on $\sfR$ has the same distribution as $\sfH$.
\end{lem}

\section{Tiling cusp universality: proof of \Cref{thm:univ}}
\label{sec:univproof}

In this section, we present the main steps for the proof of \Cref{thm:univ} as several lemmas and deduce  \Cref{thm:univ} from them.
The proofs of these lemmas will be given in subsequent sections.

\subsection*{Basic Setup}
Take any rational polygonal set $\fP$ satisfying \Cref{pa}, and recall that its liquid region and arctic curve are denoted by $\fL$ and $\fA$, respectively. Take a cusp point $(x_c, t_c) \in \fA$.
Let $n$ be any large enough integer such that $n\fP$ is a tileable domain.
As in \Cref{thm:univ}, by rotating $\fP$ if necessary, we assume that $(x_c, t_c)$ is upward oriented in the sense of \Cref{sr}, with curvature parameters $\fr,\fq$.
In this section, all the constants (including those implicitly used in $\lesssim,\gtrsim,\asymp,\OO$) can depend on $\fP$.

As indicated in the introduction, we will compare paths from tiling and NBRW in a region around $(x_c, t_c)$.
More precisely, we denote $\Delta t=n^{-\omega}$ for some constant $\omega\in(0,1/2)$.
Then we take $t_0<t_c<t_1$, such that $t_0, t_1\in n^{-1}\Z$, $t_c-t_0, t_1-t_c\asymp \Delta t$.
Take a small constant $\fc>0$.
We are mainly interested in the region $[x_c-\fc, x_c+\fc]\times [t_0, t_1]$, where $\fA$ contains two analytic pieces $\{(E_{-}(t),t): t_0\le t\le t_c\}$ and $\{(E_{+}(t),t): t_0\le t\le t_c\}$, with $E_-(t)<x_c+(t-t_c)/\fr<E_+(t)$ and $x_c+(t-t_c)/\fr-E_-(t), E_+(t)-x_c-(t-t_c)/\fr\asymp (t_c-t)^{3/2}$ for each $t\in [t_0, t_c]$.
Moreover, as pointed out in \Cref{sr}, we have $\fr=(f^*_{t_c}(x_c)+1)/f^*_{t_c}(x_c)$.
Then, we have $f^*_{t_c}(x_c)\in (0, \infty)$, implying that $\nabla H^*(x_c, t_c)=(0,0)$ by \eqref{fh}.
Therefore, we can assume that $H^*(x,t)=0$ for all $(x,t)$ in the frozen region with $t_0 \le t \le t_c$ and $E_-(t)\le x \le E_+(t)$.

\subsection{Tiling path estimates}
We next present estimates of the paths associated with tiling.
Let $\sfH:n\fP\to\R$ be the height function of the uniformly random tiling, satisfying $\sfH(nv)=nH^*(v)$ for each $v\in\partial \fP$.
We then consider a (random) family of non-intersecting Bernoulli paths as in \Cref{ssec:nbp}: for each $i, \sft\in\Z$, we define $\sfq_i(\sft)$ to be the number satisfying
$\sfH(\sfq_i(\sft), \sft)=i$ and $\sfH(\sfq_i(\sft)+1, \sft)=i+1$,
if such a number exists. 

The typical locations of the paths are deterministic numbers given by the quantiles of $H^*$, as follows.
Let $\fc>0$ be a small enough constant depending on $\fP$ and $(x_c, t_c)$.
Take $M, N\in\N$ such that 
\[\qq{-M, N}=\{i\in \Z: H^*(x_c-\fc,t_0) \le H^*(x_c,t_c)+i/n < H^*(x_c+\fc,t_0)\}.
\]
For each $t \in [t_0, t_1]$ and $i\in\qq{-M, N}$, we let 
\begin{equation}\label{eq:defgammai}
    \gamma_i(t) = \sup\{x: (x,t)\in \fP, H^*(x,t)-H^*(x_c,t_c)=i/n\}.
\end{equation}
In particular, notice that $\gamma_0(t)=E_+(t)$ when $t\le t_c$.
We have the following estimates on $\gamma_i$.
\begin{lem}  \label{lem:gammaies}
For any $i\in\N$, if $1\le i\lesssim \Delta t^2 \lsca$, we have
\begin{align}\label{e:gammailoc1}
\qut_i(t_0) - \Ep(t_0) \asymp \Delta t^{1/6} (i/\lsca)^{2/3}, \quad \Em(t_0) - \qut_{-i}(t_0) \asymp \Delta t^{1/6} (i/\lsca)^{2/3}.
\end{align}
If $i\ge C\Delta t^2 \lsca$ for a large enough constant $C>0$, we have that  for any $t\in [t_0, t_1]$, 
\begin{align}\label{e:gammailoc2}
\qut_i(t) - (x_c+(t-t_c)/\fr) \asymp (i/\lsca)^{3/4}, \ \ \ i\le N; \quad
(x_c+(t-t_c)/\fr) - \qut_{-i}(t) \asymp (i/\lsca)^{3/4},\ \ \ i\le M.
\end{align}
\end{lem}

We next give the estimate on the fluctuations of the tiling paths around these quantiles. 

\begin{lem}   \label{lem:sfqes}
For an arbitrarily small constant $\fd>0$,
with overwhelming probability, we have
\begin{align}\label{e:concentrate1}
&\sfq_{0}(nt_0)/n-E_+(t_0),\ \sfq_{-1}(nt_0)/n-E_-(t_0) \lesssim n^{-2/3+\fd}\Delta t^{1/6},\\
\label{e:concentrate2}
&|\{i \in \llbracket -M, N\rrbracket: \sfq_i(nt_0)<x\lsca\}| - |\{i \in \llbracket -M, N\rrbracket: \qut_i(t_0)<x\}|  \lesssim n^{\fd},
\end{align}
 uniformly for all $x\in \bR$.
Take a constant $ \delta \in (0, \omega/2)$, and let $L=\lceil n^{1+\delta}\Delta t^2\rceil$.
When $\fd$ is small enough (depending on $\omega$ and $\delta$), the following estimates hold with overwhelming probability:
\begin{align}\label{e:concentrate3}
 \sfq_{L}(nt)/n-\gamma_L(t), \ \sfq_{-L}(nt)/n-\gamma_{-L}(t)  \lesssim n^{-3/4-\fd},\quad &\forall t \in [t_0, t_1]\cap n^{-1}\Z,\\
\label{e:concentrate4}
 \sfq_{i}(nt_1)/n-\gamma_i(t_1) \lesssim n^{-3/4-\fd},\quad &\forall i\in\qq{-L, L}.
\end{align}
\end{lem}
The proofs of the above two lemmas rely on careful analysis of the limiting height function $H^*$ through the \emph{complex slope}. 
The proof \Cref{lem:sfqes} also uses the optimal rigidity estimate of the height function $\sfH$ around cusps, to be presented in \Cref{s:optimalR}. The detailed proofs will be presented in \Cref{s:det}.

\subsection{Construction and estimates of NBRW}  \label{ssec:cnbrw}
To prove that the random paths associated with tiling around $(x_c,t_c)$ converge to the Pearcey process, our strategy is to compare them with a certain NBRW starting from time $nt_0$.

We consider the NBRW $\widetilde\sfQ=\{\widetilde\sfq_i\}_{i=-M}^N:\qq{nt_0, \infty}\to \Z^{\qq{-M,N}}$, with initial data $\widetilde\sfQ(nt_0)=\{\widetilde\sfq_i(nt_0)\}_{i=-M}^N=\{\sfq_i(nt_0)\}_{i=-M}^N$.
Next, we explain the procedure to choose the drift parameter $\beta$. 
For any time $t\in[t_0, t_1]$,
 we denote the density $\rho^*_t(x)=\del_x H^*(x,t)$, which is defined almost everywhere and is in $[0,1]$ since $H^*$ is admissible.
We can interpret $\rho^*_t(x)$ as approximately the density of paths (or equivalently, type 2 and type 3 lozenges) around $(nx,nt)$.
We denote
\begin{align}\label{eq:deftildem}
\widetilde\rho_{t_0}=\rho_{t_0}^*\don([\gamma_{-M}(t_0), \gamma_N(t_0)]),\quad \widetilde m_{t_0}(z)=\int \frac{\widetilde\rho_{t_0}(x)\rd x}{z-x},
\end{align}
which are the restriction of $\rho_{t_0}^*$ on $[\gamma_{-M}(t_0), \gamma_N(t_0)]$ and its Stieltjes transform.
Denote\begin{align}\label{e:defzc}
    z_c=x_c-(t_c-t_0)/\fr,
\end{align}
which is the intersection of the tangent line at the cusp with the line $t=t_0$.
We take $\beta$ to satisfy
\begin{align}\label{e:chosebeta}
\frac{\beta}{1-\beta}=\frac{f_{t_c}^*(x_c)}{e^{\widetilde m_{t_0}(z_c)}}.  
\end{align}
It turns out that, by choosing such a $\beta$, the limit shape of NBRW will have a cusp at $(\widetilde x_c,\widetilde t_c)$ that is close to $(x_c,t_c)$.
Here, $\widetilde x_c$ and $\widetilde t_c$ are determined by $\widetilde\rho_{t_0}$ and $\beta$, through \Cref{lem:xtzdef} below.
More precisely, let $\widetilde f_{t_0}(z)=e^{\widetilde m_{t_0}(z)}\beta/(1-\beta)$, we then take $\widetilde x_c\in\R$, $\widetilde t_c>t_0$, and $\widetilde z_c\in (E_-(t_0), E_+(t_0))$ as the solutions to the following system of equations:
\begin{equation} \label{eq:tildedef}
\widetilde x_c =\widetilde z_c+\widetilde t_c \frac{\widetilde f_{t_0}(\widetilde  z_c)}{\widetilde f_{t_0}(\widetilde z_c)+1},\quad 
 1+\widetilde  t_c \frac{\widetilde f_{t_0}'(\widetilde z_c)}{(\widetilde f_{t_0}(\widetilde  z_c)+1)^2}=0,\quad 
\widetilde f_{t_0}''(\widetilde z_c)-\frac{2\widetilde f_{t_0}'(\widetilde z_c)^2}{\widetilde f_{t_0}(\widetilde z_c)+1}=0.
\end{equation}
By \Cref{lem:xtzdef}, these numbers exist and are unique, and there is $\widetilde t_c-t_0\asymp \Delta t$.
\begin{lem}  \label{lem:cuspclos}
We have $\widetilde x_c - x_c, \widetilde t_c - t_c, \widetilde z_c - z_c \lesssim \Delta t^2$, and $\frac{\widetilde f_{t_0}(\widetilde  z_c)}{\widetilde f_{t_0}(\widetilde z_c)+1} - \fr^{-1} \lesssim \Delta t$. 
\end{lem}

We next state a fluctuation estimate of $\widetilde{\sfQ}$ that is necessary for the proof.
\begin{lem}  \label{lem:tsfQes}
For $\omega\in (0,1/2)$ chosen in the basic setup above, assume that $\omega>3/8$ and take a constant $0<\delta < (\omega/2)\wedge(2\omega-3/4)$. Let $L=\lceil n^{1+\delta}\Delta t^2\rceil$.
Fix a constant $\fd>0$ that is small enough (depending on $\omega$ and $\delta$), then the following estimates hold with overwhelming probability:
\begin{align*}
    \widetilde\sfq_{L}(nt)/n-\gamma_L(t), \ \widetilde\sfq_{-L}(nt)/n-\gamma_{-L}(t)  \lesssim n^{-3/4-\fd},\quad &\forall t \in [t_0, t_1]\cap n^{-1}\Z,\\
 \widetilde\sfq_{i}(nt_1)/n-\gamma_i(t_1) \lesssim n^{-3/4-\fd},\quad &\forall i\in\qq{-L, L}.
\end{align*}

\end{lem}
The proofs of the above two lemmas are deferred to \Cref{s:det}.

\subsection{Pearcey limit and the comparison between tiling and NBRW}

Given the construction of the NBRW $\widetilde{\sfQ}$, using the estimates of the fluctuations of $\widetilde{\sfQ}$ (\Cref{lem:tsfQes}) and the paths from tiling (\Cref{lem:sfqes}), we can now prove \Cref{thm:univ} through comparison, by using the following convergence of $\widetilde{\sfQ}$ to the Pearcey process.
\begin{lem}   \label{lem:nbrwcco}
Fix the constant $\omega \in (3/8,1/2)$. Recall the curvature parameters $\fr,\fq$ in \Cref{sr}.
For each $\sft\in\qq{nt_0,\infty}$, we regard $\widetilde{\sfQ}(\sft)$ as a finite subset of $\Z$. Then, as $n\to\infty$, the process
\[
	t\mapsto \frac{\widetilde\sfQ(\lfloor\lsca\cust - \sqrt{\fr-1}\lsca^{1/2}t/(\fr \fq)\rfloor)-\lsca\cusx+\sqrt{\fr-1}\lsca^{1/2}t/(\fr^2 \fq)}{(\fr-1)^{3/4}\lsca^{1/4}/\sqrt{\fq\fr^3}}  , 
		\]
converges to the Pearcey process in the same sense as in \Cref{thm:univ}.
\end{lem}
This lemma is deduced from \Cref{thm:univNBRWlim}, the Pearcey universality for NBRW.
There, the scaling of the Pearcey process is described by the following two parameters:
\[
\widetilde A = (\widetilde t_c-t_0)^4 \int \frac{\widetilde\rho_{t_0}(x)}{4(\widetilde z_c-x)^4} \rd x - \frac{(\widetilde t_c-t_0)^4}{12(\widetilde t_c-t_0+\widetilde z_c-\widetilde x_c)^3} - \frac{(\widetilde t_c-t_0)^4}{12(\widetilde x_c-\widetilde z_c)^3},\quad \widetilde B =\frac{\widetilde x_c-\widetilde z_c}{\widetilde t_c-t_0},
\]
as defined in \eqref{eq:defA} below.
We need the following relation between $\widetilde A, \widetilde B$ and $\fr, \fq$. We defer its proof to \Cref{s:det}.
\begin{lem}  \label{lem:relapears}
As $n\to\infty$, we have $\widetilde B \to \fr^{-1}$ and $\widetilde A \to \fr^2(\fr-1)^{-1}\fq^{-2}/4$.
\end{lem}

\begin{proof}[Proof of \Cref{lem:nbrwcco}]
Take $\fd>0$ to be arbitrarily small depending on $\omega$.
By \Cref{lem:gammaies} and \Cref{lem:sfqes}, 
the quantiles of $\widetilde\rho_{t_0}$ (which are precisely $\gamma_i(t_0)$) satisfy the growth specified in \Cref{ass:dtsncusp}, with $\sct=\Delta t=n^{-\omega}$; and with overwhelming probability, the initial data $\widetilde\sfQ(nt_0)=\{\widetilde\sfq_i(nt_0)\}_{i=-M}^N=\{\sfq_i(nt_0)\}_{i=-M}^N$ is approximated by $\widetilde\rho_{t_0}$ in the sense stated in \Cref{ass:dtsncusp}, with $\err=\fd$.
In addition, we know that $\beta$ is bounded away from $0$ and $1$, uniformly in $n$. Indeed, by \eqref{e:chosebeta}, it suffices to show that $\widetilde m_{t_0}(z_c)\lesssim 1$.
From the definition of $z_c$, we obtain $E_+(t_0)-z_c, z_c-E_-(t_0)\asymp \Delta t^{3/2}$. Then, by decomposing the integral $\int  {\widetilde\rho_{t_0}(x)}/({z_c-x})\rd x$ according to the quantiles $\gamma_i(t_0)$, we can readily deduce $\widetilde m_{t_0}(z_c)\lesssim 1$ with the help of \Cref{lem:gammaies}.

With the above preparations, we can now apply \Cref{thm:univNBRWlim} to conclude the convergence of
\[
t\mapsto \frac{\widetilde\sfQ(\lfloor\lsca\widetilde t_c - 2\widetilde A^{1/2}\widetilde B(1-\widetilde B)\lsca^{1/2}t\rfloor)-\lsca\widetilde x_c}{\sqrt{2}\widetilde A^{1/4}\widetilde B(1-\widetilde B)\lsca^{1/4}} +\sqrt{2}\widetilde A^{1/4}\widetilde B\lsca^{1/4}t
\]
to the Pearcey process. Then, by \Cref{lem:cuspclos} and \Cref{lem:relapears}, and noticing that $\Delta t=\oo(n^{-3/8})$ (so $n\widetilde x_c-nx_c, n\widetilde t_c-nt_c=\oo(n^{1/4})$), the conclusion follows.
\end{proof}

We now finish the comparison arguments.
\begin{proof}[Proof of \Cref{thm:univ}]
We take constants $\omega \in (3/8,1/2)$ and $0<\delta < (\omega/2)\wedge(2\omega-3/4)$, and let $\Delta t=n^{-\omega}$ and $L=\lceil n^{1+\delta}\Delta t^2\rceil$.
For the random paths $\{\sfq_i(\sft)\}_{i\in\qq{-L,L}, \sft\in\qq{nt_0, nt_1}}$ associated with tiling around $(nx_c, nt_c)$ and the NBRW paths $\{\widetilde\sfq_i(\sft)\}_{i\in\qq{-L,L}, \sft\in\qq{nt_0, n_1}}$, by \Cref{lem:sfqes}, \Cref{lem:tsfQes}, and the monotonicity property in  \Cref{comparewalks}, we can couple them so that with overwhelming probability,
\begin{align}\label{e:couple}
\max\{|\sfq_i(\sft)-\widetilde\sfq_i(\sft)|: i\in\qq{-L,L}, \sft\in\qq{nt_0, nt_1}\}= \oo(n^{1/4}).
\end{align}
By \Cref{lem:gammaies} and \Cref{lem:sfqes}, 
with overwhelming probability, we have 
\[
\max\{\sfq_{-L}(\sft)-(\sft-nt_c)/\fr: \sft\in\qq{nt_0, nt_1}\} < -n^{1/4+\fd},
\]
\[
\min\{\sfq_{L}(\sft)-(\sft-nt_c)/\fr: \sft\in\qq{nt_0, nt_1}\} > n^{1/4+\fd},
\]
for a small enough $\fd>0$ depending on $\delta$.
These ensure that the paths $\{\sfq_i\}_{i\in\qq{-L,L}}$ and $\{\widetilde\sfq_i\}_{i\in\qq{-L,L}}$ contain all paths around $(nx_c, nt_c)$ in $\sfQ$ and $\widetilde\sfQ$, respectively.
Then, the conclusion follows from \Cref{lem:nbrwcco}.
\end{proof}

\section{Cusp universality for non-intersecting Bernoulli random walks} \label{s:nibr}
In this section, we study non-intersecting Bernoulli random walks 
\[\sfA=\{\sfa_{i}\}_{i\in \llbracket -M, N\rrbracket}:\qq{0,\infty}\to \Z^{\qq{-M,N}},\] 
with drift parameter $\beta\in (0,1)$ and initial configuration $\sfA(0)=\{\sfd_i\}_{i\in\qq{-M,N}} \in\Z^{\llbracket -M, N\rrbracket}$ for $M, N \asymp \lsca$, as defined in Section \ref{ssec:unbrw}. 
We will assume that $\{\sfd_i\}_{i\in\qq{-M,N}}$ contains two separated parts $\{\sfd_i\}_{i\in\qq{-M,-1}}$ and $\{\sfd_i\}_{i\in\qq{0,N}}$, so that cusp forms when the two parts meet, and we prove that the Pearcey process appears around the cusp location.
Our proof uses the fact that both NBRW and the Pearcey process are determinantal point processes, and we bound the difference between their kernels (see Proposition \ref{prop:Kconvconj} below).

We now set up the parameters we will use.
First, we take $\phi\in(0,1/2)$, and we assume that the drift parameter $\beta\in(\phi, 1-\phi)$.
Then, we take small positive constants $\errh$, $\ere$, $\errd$, $\err$, in the following way:
\begin{enumerate}
    \item[(1)] $\errh>0$ is any number small enough depending on $\phi$;
    \item[(2)] $\ere>0$ is any number small enough depending on $\phi, \errh$;
    \item[(3)] $\errd>0$ is any number small enough depending on $\phi, \errh, \ere$;
    \item[(4)] $\err>0$ is any number small enough depending on $\phi, \errh, \ere, \errd$.
\end{enumerate}
The precise requirements for the choice of these parameters will be clear in the proofs below.
All other constants ($c,C>0$ and those implicitly used in $\lesssim$, $\gtrsim$, $\asymp$, $\OO$) can depend on these parameters.

We next describe the assumptions on the initial configuration $\{\sfd_i\}_{i\in\qq{-M,N}}$.
We approximate it with a density function $\den:\R\to [0,1]$, when rescaled by $\lsca$.
Take $\sct>0$ satisfying that $\lsca^{-1/2+\errh}<\sct\lesssim 1$.
The density function $\den$ can depend on $\lsca$, and needs to satisfy certain assumptions to form a cusp at the distance of order $\lsca\sct$.

\begin{assumption}  \label{ass:dtsncusp}
We assume that $\den=0$ on $[\Em, \Ep]$, for some $\Em<\Ep$ with $\Ep-\Em\asymp\sct^{3/2}$.
Let $\{\qut_i\}_{i\in\llbracket-M,N\rrbracket}$ be the scale $\lsca^{-1}$ quantiles starting from $\Em$ and $\Ep$. Namely, we let $\qut_0=\Ep$; for $i\in\llbracket 1, N\rrbracket$ or $i\in\llbracket 1, M\rrbracket$, let $\qut_i$ or $\qut_{-i}$ satisfy that,
\[
\int_{\Ep}^{\qut_i} \den(x) \rd x = i/\lsca, \quad \text{or}\quad
\int_{\qut_{-i}}^{\Em} \den(x) \rd x = i/\lsca.
\]
We assume that $\den=0$ on $(-\infty, \qut_{-M})\cup (\qut_N, \infty)$.
In addition, we assume that for $i\in\N$, when $i\lesssim \sct^2 \lsca$,
\[
\qut_i - \Ep \asymp \sct^{1/6} (i/\lsca)^{2/3}, \quad \Em - \qut_{-i} \asymp \sct^{1/6} (i/\lsca)^{2/3};
\]
and when $i\gtrsim \sct^2 \lsca$, $i\le N$ or $M$,
\[
\qut_i - \Ep \asymp (i/\lsca)^{3/4},\quad
\Em - \qut_{-i} \asymp (i/\lsca)^{3/4}.
\]
For $\{\sfd_i\}_{i\in\qq{-M,N}}$, we assume that it is approximated by $\den$, in the following sense:
\begin{itemize}
    \item $\sfd_0-\Ep \lsca\lesssim \lsca^{1/3+\err} \sct^{1/6}$ and $\sfd_{-1}-\Em \lsca\lesssim \lsca^{1/3+\err} \sct^{1/6}$; 
    \item for any $x\in\R$, we have
\begin{equation}  \label{eq:heightclose}
 |\{i \in \llbracket -M, N\rrbracket: \sfd_i<x\lsca\}| - |\{i \in \llbracket -M, N\rrbracket: \qut_i<x\}|  \lesssim \lsca^{\err}.  
\end{equation}
\end{itemize}

\end{assumption}

\subsection{Cusp location}
We now determine the cusp location of NBRW under Assumption \ref{ass:dtsncusp}, in a way indicated by Lemma \ref{lem:xtzdeft}.
For this purpose, we consider the Stieltjes transform
$\mmm(\zzz)=\int \frac{\den(x)}{\zzz-x}\rd x$.
Below are some basic properties of $\mmm$ for $z\in (\Em, \Ep)$, which are straightforward to check.
\begin{enumerate}
    \item $\mmm'(\zzz)=-\int \frac{\den(x)}{(z-x)^2} \rd x<0$ for any $\zzz \in (\Em, \Ep)$, and
    \begin{equation}  \label{eq:mmmfir}
        |\mmm'(\zzz)| = \int \frac{\den(x)}{(z-x)^2} \rd x \asymp \frac{1}{\lsca}\sum_{i=-M}^N \frac{1}{(\zzz-\qut_i)^2} \asymp \sct^{-1/4}((\zzz-\Em)\wedge (\Ep-\zzz))^{-1/2},
    \end{equation}
    when $\zzz \in (\Em+\lsca^{-2/3}\sct^{1/6}, \Ep-\lsca^{-2/3}\sct^{1/6})$.
    \item When $\zzz \in (\Em+\lsca^{-2/3}\sct^{1/6}, \Ep-\lsca^{-2/3}\sct^{1/6})$, we have
        \begin{equation}  \label{eq:mmmzer}
        |\mmm(\zzz)|\le \int \frac{\den(x)}{|z-x|}\rd x \asymp \frac{1}{\lsca}\sum_{i=-M}^N \frac{1}{|\zzz-\qut_i|} \lesssim 1.
    \end{equation}
    \item With $\mmm''(\zzz)=\int \frac{2\den(x)}{(\zzz-x)^3}\rd x$, we can deduce that $\mmm''(\Em+\lsca^{-2/3}\sct^{1/6})>0$, $\mmm''(\Ep-\lsca^{-2/3}\sct^{1/6}) < 0$, and
    \begin{equation}  \label{eq:mmmsec}
    \mmm''(\zzz)\asymp \sct^{-1/4}((\zzz-\Em)\wedge (\Ep-\zzz))^{-3/2},
    \end{equation}
    for any $\zzz \in (\Em, \Ep)$ with $(\zzz-\Em)\wedge (\Ep-\zzz) < c\sct^{3/2}$ for a small enough constant $c>0$.
    \item $\mmm'''(\zzz)=-\int \frac{6\den(x)}{(\z-x)^4} \rd x<0$ for any $\zzz \in (\Em, \Ep)$, so $\mmm''$ is decreasing in $(\Em, \Ep)$.
    Moreover, 
    \begin{equation}  \label{eq:mmmthi}
    \mmm'''(\zzz)=-\int \frac{6\den(x)}{(\z-x)^4} \rd x\asymp \frac{1}{\lsca}\sum_{i=-M}^N \frac{1}{(\zzz-\qut_i)^4} \asymp \sct^{-4},
    \end{equation}
whenever $(\zzz-\Em)\wedge (\Ep-\zzz) \asymp \sct^{3/2}$.
\item For any $\zzz\in \C$ with $\Re \zzz\in (E_-,E_+)$ and $(\Re \zzz -\Em)\wedge (\Ep-\Re \zzz ) \asymp \sct^{3/2}$, we have
\begin{equation}  \label{eq:mmmfo}
    \mmm''''(\zzz) = \int \frac{24\den(x)}{(\zzz-x)^5}\rd x \lesssim \sct^{-11/2}.
\end{equation}
\end{enumerate}
We further denote (see \Cref{prop:slopeNBRW} below)
$\fff(\zzz)=e^{\mmm(\zzz)}\frac{\beta}{1-\beta}$.
We then find the cusp location, using formulas inspired by the complex slope (see \Cref{lem:xtzdeft} below).
\begin{lem}  \label{lem:xtzdef}
There exist $\cusx \in \R$, $\cust>0$, and $\zzc \in (\Em, \Ep)$, such that 
\begin{equation}  \label{eq:defcusx}
\cusx=\zzc+\cust \frac{\fff(\zzc)}{\fff(\zzc)+1},
\end{equation}
\begin{equation}  \label{eq:defcust}
 1+\cust \frac{\fff'(\zzc)}{(\fff(\zzc)+1)^2}=0,
\end{equation}
\begin{equation}  \label{eq:defzzzc}
\fff''(\zzc)-\frac{2\fff'(\zzc)^2}{\fff(\zzc)+1}=0.
\end{equation}
In addition, we have 
\begin{equation}\label{eq:zbtcloser}
\cust\asymp\sct, \quad (\Ep-\zzc)\wedge (\zzc-\Em) \asymp \sct^{3/2}, \quad
-\cust < \zzc-\cusx < 0,\quad
\cusx-\zzc,\, \cust+\zzc-\cusx\asymp \cust.
\end{equation}
\end{lem}
The cusp should be present around the location $(\lsca\cusx, \lsca\cust)$. 
The strategy to prove this lemma is to first determine $\zzc$ using \eqref{eq:defzzzc}, then $\cust$ using \eqref{eq:defcust}, and finally $\cusx$ using \eqref{eq:defcusx}.
\begin{proof}[Proof of Lemma \ref{lem:xtzdef}]
Note that \eqref{eq:defzzzc} is equivalent to $g(\zzc)=0$, where for $\zzz\in(\Em, \Ep)$, $g(\zzz)$ is defined as
\[
g(\zzz):=[(\mmm'(\zzz))^2 + \mmm''(\zzz)]\beta (\beta+(1-\beta)e^{-\mmm(\zzz)})-2(\mmm'(\zzz))^2\beta^2.
\]
Using the above basic properties \eqref{eq:mmmfir}, \eqref{eq:mmmzer}, and \eqref{eq:mmmsec}, we get that $g(\Em+\lsca^{-2/3}\sct^{1/6})>0$ and $g(\Ep-\lsca^{-2/3}\sct^{1/6})<0$.
Thus, we can find $\zzc\in(\Em+\lsca^{-2/3}\sct^{1/6}, \Ep-\lsca^{-2/3}\sct^{1/6})$ such that $g(\zzc)=0$.
For such $\zzc$, \eqref{eq:mmmfir} and $g(\zzc)=0$ imply that 
\[
\mmm''(\zzc) \lesssim (\mmm'(\zzc))^2 \lesssim \sct^{-1/2}((\zzc-\Em)\wedge (\Ep-\zzc))^{-1}.
\]
Then by \eqref{eq:mmmsec}, we must have that $(\zzc-\Em)\wedge (\Ep-\zzc) \asymp\sct^{3/2}$.
So by \eqref{eq:mmmfir}, we have $\mmm'(\zzc)\asymp\sct^{-1}$, and by \eqref{eq:mmmzer}, we have $\mmm(\zzc)\lesssim 1$.
The number $\cust$ is then determined by \eqref{eq:defcust}, which yields that
\[
\cust = -\frac{(e^{\mmm(\zzc)}\beta+(1-\beta))^2}{\mmm'(\zzc)\beta(1-\beta)e^{\mmm(\zzc)}} \asymp \sct.
\]
Finally, $\cusx$ is solved from \eqref{eq:defcusx}.
In particular, \eqref{eq:defcusx} and the fact that $\mmm(\zzc)\lesssim 1$ imply that $-\cust < \zzc-\cusx < 0$ and $\cusx-\zzc,\, \cust+\zzc-\cusx\asymp \cust$.
\end{proof}

\subsection{Kernel approximation}  \label{ssec:kapprox}
As discussed before, for the NBRW $\sfA$, the set $\{(\sfa_i(\sft),\sft)\}_{i\in\llbracket- M,N\rrbracket, \sft\in\Z_{\ge 0}}$ is a determinantal point process on $\Z^2$.
The kernel is given in \cite[Theorem 2.1]{gorin2019universality} as
    \begin{multline}  \label{eq:KBernoulli}
        K^{\textnormal{Bernoulli}}(\oot_1,\oox_1;\oot_2,\oox_2)
        =
        \don_{\oox_1\ge \oox_2}\don_{\oot_1>\oot_2}
        (-1)^{\oox_1-\oox_2+1}
        \binom{\oot_1-\oot_2}{\oox_1-\oox_2}
        \\+
        \frac{\oot_1!}{(\oot_2-1)!}\frac1{(2\pi\i)^{2}}
        \int\limits_{\oox_2-\oot_2+\frac12-\i\infty}
        ^{\oox_2-\oot_2+\frac12+\i\infty}\rd \ooz
        \oint\limits_{\textnormal{all $\oow$ poles}}\rd \oow\,
        \frac{(\ooz-\oox_2+1)_{\oot_2-1}}{(\oow-\oox_1)_{\oot_1+1}}
        \frac{1}{\oow-\ooz}
        \frac{\sin(\pi \oow)}{\sin(\pi \ooz)}
        \left(\frac{1-\beta}{\beta}\right)^{\oow-\ooz}
        \prod_{i=-M}^{N}\frac{\ooz-\sfd_i}{\oow-\sfd_i},
    \end{multline}
for any $\oox_1, \oox_2\in \Z$ and $\oot_1, \oot_2\in \Z_+$. 
Here, we recall the Pochhammer symbols and the binomial coefficients defined at the beginning of \Cref{Walks}. 
The integration contours for $\ooz$ and $\oow$ are as follows: the $\ooz$ contour is the straight vertical line $\Re(\ooz)=\oox_2-\oot_2+\frac12$ traversed upwards, and the $\oow$ contour is a positively (counter-clockwise) oriented circle or a union of two circles encircling all the $\oow$ poles $\{\oox_1-\oot_1,\oox_1-\oot_1+1,\ldots,\oox_1-1,\oox_1\}\cap \{\sfd_i\}_{i\in\qq{-M,N}}$ of the integrand, except the pole at $\oow=\ooz$.

With proper scaling, this kernel $K^{\textnormal{Bernoulli}}$ should be approximated by the Pearcey kernel $K^{\textnormal{Bernoulli}}$, given by \eqref{eq:pearcey}. This is the main task of this section.
For this purpose, we denote
\begin{equation}  \label{eq:defA}
\begin{split}
A &=\cust^4 \int \frac{\den(x)}{4(\zzc-x)^4} \rd x - \frac{\cust^4}{12(\cust+\zzc-\cusx)^3} - \frac{\cust^4}{12(\cusx-\zzc)^3},\quad \BE =(\cusx-\zzc)\cust^{-1}.
\end{split}
\end{equation}
By \eqref{eq:zbtcloser}, we have $B>0$ and $B\asymp 1$.
We also have $A>0$ and $A\asymp 1$, which will be proved later as Lemma \ref{lem:Sderiv}.
Then, \Cref{thm:univNBRWlim} is an immediate consequence of the next proposition.
\begin{prop}  \label{prop:Kconvconj}
Suppose Assumption \ref{ass:dtsncusp} holds. Take $\cusx \in \R$ and $\cust>0$ from Lemma \ref{lem:xtzdef}, and define $A, B$ as in \eqref{eq:defA}.
Take any $\pct_1, \pct_2, \pcx_1, \pcx_2\in\R$, $\oox_1, \oox_2\in\Z$, and $\oot_1, \oot_2\in \Z_+$ such that $|\pct_1|, |\pct_2|, |\pcx_1|, |\pcx_2| \lesssim 1$, and $\oot_1=\lsca\cust +\lsca^{1/2}\pct_1+\OO(1)$, $\oot_2=\lsca\cust +\lsca^{1/2}\pct_2+\OO(1)$, $\oox_1=\lsca\cusx+B \lsca^{1/2}\pct_1+\lsca^{1/4}\pcx_1+\OO(1)$, $\oox_2=\lsca\cusx+B \lsca^{1/2}\pct_2+\lsca^{1/4}\pcx_2+\OO(1)$. 
Then, we have
\begin{multline*}
\Big|(-1)^{\oox_1-\oox_2}\BE^{\oox_1-\oox_2}(1-\BE)^{\oot_1-\oot_2+\oox_2-\oox_1}
K^{\textnormal{Bernoulli}}(\oot_1,\oox_1;\oot_2,\oox_2)
    - \frac{1}{\sqrt{2}\lsca^{1/4}A^{1/4}\BE(1-\BE)} \\ \times K^{\textnormal{Pearcey}}\Big(\frac{-\pct_1}{2A^{1/2}\BE(1-\BE)},\frac{\pcx_1}{\sqrt{2}A^{1/4}\BE(1-\BE)};\frac{-\pct_2}{2A^{1/2}\BE(1-\BE)},\frac{\pcx_2}{\sqrt{2}A^{1/4}\BE(1-\BE)}\Big)\Big|\lesssim \lsca^{-1/4-\errd},
\end{multline*}
if $\pct_1= \pct_2$ and $\oot_1=\oot_2$, or $|\pct_1-\pct_2|>\lsca^{-\errd}$.
\end{prop}

The rest of this section is devoted to proving this proposition. 
Without loss of generality, below we assume that $\cusx=0$, by shifting $\den$ and $\{\sfd_i\}_{i\in\qq{-M,N}}$.
Then by \eqref{eq:zbtcloser} and Lemma \ref{lem:Sderiv} below, 
we have $-\cust < \Em < \zzc < \Ep < 0$ and $\Ep, \Em+\cust\asymp \cust$, and $\Ep-\zzc, \zzc-\Em \asymp \sct^{3/2}$.
Therefore, we have (see Figure \ref{fig:veca})
\begin{equation}  \label{eq:vecal}
(\oox_1-\oot_1+1)/\lsca,\, (\oox_2-\oot_2+1)/\lsca<\sfd_{-1}/\lsca < \zzc < \sfd_0/\lsca < (\oox_1-1)/\lsca, (\oox_2-1)/\lsca.    
\end{equation}
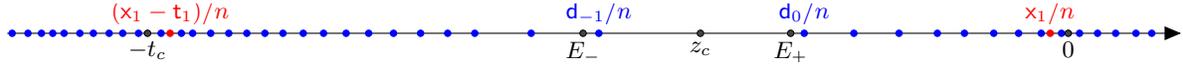
\begin{figure}[hbt!]
    \centering
\begin{tikzpicture}[line cap=round,line join=round,>=triangle 45,x=0.3cm,y=0.25cm]
\clip(-2.5,-1.5) rectangle (51,1.55);

\draw [->] (-2,0) -- (50,0);

\begin{small}
\draw (29.7-1,0) node[anchor=north]{$\zzc$};

\draw (25.5-2,0) node[anchor=north]{$\Em$};
\draw (32.7,0) node[anchor=north]{$\Ep$};
\draw [blue] (33.3,0) node[anchor=south]{$\sfd_0/\lsca$};
\draw [blue] (26.2-2,0) node[anchor=south]{$\sfd_{-1}/\lsca$};

\draw (6.2-2,0) node[anchor=north]{$-\cust$};
\draw (45,0) node[anchor=north]{$0$};

\draw [red] (44.2,0) node[anchor=south]{$\oox_1/\lsca$};

\draw [red] (7.2-2,0) node[anchor=south]{$(\oox_1-\oot_1)/\lsca$};
\end{small}

\draw [blue] [fill=blue] (26.2-2,0) circle (1.2pt);
\draw [blue] [fill=blue] (33.3,0) circle (1.2pt);

\draw [blue] [fill=blue] (23.2-2,0) circle (1.2pt);
\draw [blue] [fill=blue] (20.7-2,0) circle (1.2pt);
\draw [blue] [fill=blue] (19.5-2,0) circle (1.2pt);
\draw [blue] [fill=blue] (18.3-2,0) circle (1.2pt);
\draw [blue] [fill=blue] (17.1-2,0) circle (1.2pt);
\draw [blue] [fill=blue] (16-2,0) circle (1.2pt);

\draw [red] [fill=red] (7.2-2,0) circle (1.2pt);

\draw [red] [fill=red] (44.2,0) circle (1.2pt);

\draw [blue] [fill=blue] (15-2,0) circle (1.2pt);
\draw [blue] [fill=blue] (14-2,0) circle (1.2pt);
\draw [blue] [fill=blue] (13.1-2,0) circle (1.2pt);
\draw [blue] [fill=blue] (12.2-2,0) circle (1.2pt);
\draw [blue] [fill=blue] (11.4-2,0) circle (1.2pt);
\draw [blue] [fill=blue] (10.6-2,0) circle (1.2pt);
\draw [blue] [fill=blue] (9.8-2,0) circle (1.2pt);
\draw [blue] [fill=blue] (9-2,0) circle (1.2pt);
\draw [blue] [fill=blue] (8.2-2,0) circle (1.2pt);
\draw [blue] [fill=blue] (7.7-2,0) circle (1.2pt);
\draw [blue] [fill=blue] (6.8-2,0) circle (1.2pt);
\draw [blue] [fill=blue] (5.8-2,0) circle (1.2pt);
\draw [blue] [fill=blue] (5.1-2,0) circle (1.2pt);
\draw [blue] [fill=blue] (4.5-2,0) circle (1.2pt);
\draw [blue] [fill=blue] (3.8-2,0) circle (1.2pt);
\draw [blue] [fill=blue] (3.2-2,0) circle (1.2pt);
\draw [blue] [fill=blue] (2.5-2,0) circle (1.2pt);
\draw [blue] [fill=blue] (2.-2,0) circle (1.2pt);
\draw [blue] [fill=blue] (1.5-2,0) circle (1.2pt);
\draw [blue] [fill=blue] (0.9-2,0) circle (1.2pt);
\draw [blue] [fill=blue] (0.2-2,0) circle (1.2pt);

\draw [blue] [fill=blue] (35.5,0) circle (1.2pt);
\draw [blue] [fill=blue] (37.5,0) circle (1.2pt);
\draw [blue] [fill=blue] (39.2,0) circle (1.2pt);
\draw [blue] [fill=blue] (40.4,0) circle (1.2pt);
\draw [blue] [fill=blue] (41.7,0) circle (1.2pt);
\draw [blue] [fill=blue] (42.8,0) circle (1.2pt);
\draw [blue] [fill=blue] (43.8,0) circle (1.2pt);
\draw [blue] [fill=blue] (44.7,0) circle (1.2pt);
\draw [blue] [fill=blue] (45.5,0) circle (1.2pt);
\draw [blue] [fill=blue] (46.3,0) circle (1.2pt);
\draw [blue] [fill=blue] (47.1,0) circle (1.2pt);
\draw [blue] [fill=blue] (48.,0) circle (1.2pt);
\draw [blue] [fill=blue] (48.7,0) circle (1.2pt);

\draw [fill=uuuuuu] (29.7-1,0) circle (1.2pt);

\draw [fill=uuuuuu] (25.5-2,0) circle (1.2pt);
\draw [fill=uuuuuu] (32.7,0) circle (1.2pt);

\draw [fill=uuuuuu] (6.2-2,0) circle (1.2pt);
\draw [fill=uuuuuu] (45,0) circle (1.2pt);

\end{tikzpicture}
\caption{An illustration of the initial configuration $\{\sfd_i\}_{i\in\qq{-M,N}}$ (blue), and the locations of $\Em$, $\Ep$, $\zzc$, $-\cust$, $0$, and $(\oox_1-\oot_1)/\lsca$, $\oox_1/\lsca$.}
\label{fig:veca}
\end{figure}

Our main task is now to analyze the contour integral in \eqref{eq:KBernoulli}.
By separating the terms containing $\ooz$ or $\oow$, we need to study the integral of $\mathsf f(\ooz,\oow)=\frac{1}{\oow-\ooz}\exp(\lsca\Dfin_2(\ooz/\lsca)-\lsca\Dfin_1(\oow/\lsca))$, with the same contours.
Here, $\Dfin_1$ and $\Dfin_2$ are two \emph{key functions} to be defined in Section \ref{ssec:kfunc}.
These functions have three critical points near $\zzc$. From that, we will show that $\Dfin_1(\z)-\Dfin_1(\zzc)$ and $\Dfin_2(\z)-\Dfin_2(\zzc)$ are approximately $-\cust^{-4}A(\z-\zzc)^4$ around $\zzc$.
In light of this, we then use the steepest descent method as follows: we deform the contours of $\ooz$ and $\oow$, so that $\ooz/\lsca$ passes through $\zzc$ vertically, and $\oow/\lsca$ passes through $\zzc$ in the $e^{\pi\i/4}$, $e^{3\pi\i/4}$, $e^{5\pi\i/4}$, $e^{7\pi\i/4}$ directions, and these contours roughly follow the steepest descent curves of $\Re(\Dfin_2)$ and $\Re(\Dfin_1)$ away from $\zzc$.
Then, for the integral of $\mathsf f(\ooz,\oow)$, the main contribution comes from the part of the contours where $|\ooz/\lsca-\zzc|$ and $|\oow/\lsca-\zzc|$ are of order $\OO( \cust\lsca^{-1/4})$. We will later call this part the `inner part', and the remaining part the `outer part'.  
We will do a careful asymptotic analysis of the inner part to obtain the Pearcey kernel \eqref{eq:pearcey}, and show that (under appropriate scaling) the outer part decays to zero as $\lsca\to\infty$.

As already mentioned in the introduction, the steepest descent method has been extensively used to do asymptotic analysis for determinantal point processes. In particular, it has been used in \cite{RSPPP} around a triple critical point to obtain the extended Pearcey kernel for weighted random tilings of special domains; in \cite{gorin2019universality}, it was used to prove convergence to the extended discrete Sine kernel in the bulk of NBRW.
Our task here is more intricate than these previous works, due to the following reasons. (1) Compared to \cite{RSPPP}, we work with general initial configurations rather than special ones. (2) Compared to \cite{gorin2019universality} where the key saddle points are a pair of complex conjugate critical points away from the real line, here we need to handle three critical points near $\zzc\in\R$, which can lead to more complicated behaviors for $\Dfin_1$ and $\Dfin_2$. (3) The fact that we seek for a `small distance' result (i.e., having a cusp at time $\lsca\cust$, which is allowed to be much smaller than $\lsca$) adds to the technical difficulty.
Therefore, delicate computations are needed to achieve the desired approximation of $\Dfin_1$ and $\Dfin_2$ in the inner part of the contours.
For the outer part of the contours, which are taken to be the steepest descent curves of $\Re \Dfin_1 $ and $\Re \Dfin_2 $ outside a ball of radius $\cust\lsca^{-1/4}$ around $\zzc$, it is hard to precisely describe them, as controlling the behavior of $\Dfin_1$, $\Dfin_2$ is challenging in this region. For example, at distance $\asymp\sct^{3/2}$ from $\zzc$, there already exist singular points (i.e., $\sfd_{-1}/\lsca$ and $\sfd_{0}/\lsca$).

In the next two subsections, we will introduce and analyze the \emph{key functions} $\Dfin_1$, $\Dfin_2$, and study their steepest descent curves.
Using them, we finish the proof of Proposition \ref{prop:Kconvconj} in Sections \ref{ssec:convKP} and \ref{ssec:small}.

\subsection{Key functions} \label{ssec:kfunc}
We now define two functions $\Dfin_2$ and $\Dfin_1$, through the following expressions:
\begin{align*}
\Dfin_2(\z) &= \frac{1}{\lsca}\sum_{i=-M}^N\log\Big(\z-\frac{\sfd_i}{\lsca}\Big)
    +\frac{1}{\lsca}\sum_{i=-\oox_2+1}^{-\oox_2+\oot_2-1}\log\Big(\z+\frac{i}{\lsca}\Big)
    -\frac{1}{\lsca}\log(\sin(\pi\lsca\z))+\z \log(\beta/(1-\beta)),\\
\Dfin_1(\z) &= \frac{1}{\lsca}\sum_{i=-M}^N\log\Big(\z-\frac{\sfd_i}{\lsca}\Big)
    +\frac{1}{\lsca}\sum_{i=-\oox_1}^{-\oox_1+\oot_1}\log\Big(\z+\frac{i}{\lsca}\Big)
    -\frac{1}{\lsca}\log(\sin(\pi\lsca\z))+\z \log(\beta/(1-\beta)).
\end{align*}
We note that these expressions define $\Dfin_2$ and $\Dfin_1$ as holomorphic functions in the upper-half complex plane $\HH$, up to adding a pure imaginary constant.
They are also analytically extended to $\R\setminus E(\Dfin_2)$ and $\R \setminus E(\Dfin_1)$, respectively, where 
\begin{align*}
E(\Dfin_2)&=\lsca^{-1}\left[(\{\sfd_i\}_{i=-M}^N \cap [\oox_2-\oot_2+1, \oox_2-1] ) \cup ( \qq{-\infty,\oox_2-\oot_2}\cup\qq{\oox_2,\infty}  \setminus \{\sfd_i\}_{i=-M}^N)\right],\\
E(\Dfin_1)&=\lsca^{-1}\left[(\{\sfd_i\}_{i=-M}^N \cap [\oox_1-\oot_1, \oox_1] ) \cup ( \qq{-\infty,\oox_1-\oot_1-1}\cup\qq{\oox_1+1,\infty}  \setminus \{\sfd_i\}_{i=-M}^N)\right]
\end{align*}
are the sets of poles of $\Dfin'_2$ and $\Dfin'_1$, respectively.
We note that $\Im \Dfin_2 $ (resp.~$\Im \Dfin_1 $) is constant in each interval of $\R\setminus E(\Dfin_2)$ (resp.~$\R \setminus E(\Dfin_1)$);
 in particular, $\Im \Dfin_2 $ and $\Im \Dfin_1 $ are both constant in $(\sfd_{-1}/\lsca, \sfd_0/\lsca)$.
 We then choose the pure imaginary constants for $\Dfin_2$ and $\Dfin_1$ such that for $z\in \R$,
\begin{align*}
\Dfin_2(\z) &= \frac{1}{\lsca}\sum_{i=-M}^N\log \left|\z-\frac{\sfd_i}{\lsca}\right| 
    +\frac{1}{\lsca}\sum_{i=-\oox_2+1}^{-\oox_2+\oot_2-1}\log \left|\z+\frac{i}{\lsca}\right| 
    -\frac{1}{\lsca}\log(\sin(\pi\lsca\z))+\z \log(\beta/(1-\beta)),\\
\Dfin_1(\z) &= \frac{1}{\lsca}\sum_{i=-M}^N\log\left|\z-\frac{\sfd_i}{\lsca}\right|
    +\frac{1}{\lsca}\sum_{i=-\oox_1}^{-\oox_1+\oot_1}\log\left|\z+\frac{i}{\lsca}\right|
    -\frac{1}{\lsca}\log(\sin(\pi\lsca\z))+\z \log(\beta/(1-\beta)).
\end{align*}
In particular, under this choice, we have $\Im \Dfin_2 =\Im \Dfin_1 =0$ in $(\sfd_{-1}/\lsca, \sfd_0/\lsca)$. 
Finally, we analytically extend $\Dfin_2$ and $\Dfin_1$ to the lower-half plane $\HH^-$ from $\HH\cup(\sfd_{-1}/\lsca, \sfd_0/\lsca)$.

Now, by a change of variables of $\z=\ooz/\lsca$ and $\w=\oow/\lsca$, we can rewrite \eqref{eq:KBernoulli} as
    \begin{multline}  \label{eq:kberD}
        K^{\textnormal{Bernoulli}}(\oot_1,\oox_1;\oot_2,\oox_2)
        =
        \mathbf{1}_{\oox_1\ge \oox_2}\mathbf{1}_{\oot_1>\oot_2}
        (-1)^{\oox_1-\oox_2+1}
        \binom{\oot_1-\oot_2}{\oox_1-\oox_2} \\ +
        (-1)^{\oox_1-\oox_2+1}\frac{\oot_1!}{(\oot_2-1)!}\frac{\lsca^{\oot_2-\oot_1-1}}{(2\pi\i)^{2}}
        \int\limits_{(\oox_2-\oot_2+\frac12)/\lsca-\i\infty}
        ^{(\oox_2-\oot_2+\frac12)/\lsca+\i\infty}
        \rd\z
        \oint\limits_{\textnormal{all $\w$ poles}}
       \frac{\rd\w }{\w-\z}\exp(\lsca\Dfin_2(\z)-\lsca\Dfin_1(\w)),
    \end{multline}
where the $\z$ contour is the straight vertical line $\Re  \z =(\oox_2-\oot_2+\frac12)/\lsca$ traversed upwards, and the $\w$ contour encircles all the $\w$ poles, except the pole at $\w=\z$.

For the rest of this subsection, we derive some estimates of $\Dfin_1$ and $\Dfin_2$ near $\zzc$, by approximating them with some easier-to-analyze functions in several steps.

\subsubsection{The function $\Sfin$} 
It would be more convenient to work with a `continuous version' of the functions $\Dfin_1$ and $\Dfin_2$, defined as \begin{equation}   \label{eq:defSf}
\begin{split}
\Sfin(\zzz) =& \int_{-\infty}^{E_-} \log(\zzz-x)\den(x) \rd x + \int_{E_+}^{\infty} \log(x-\zzz)\den(x) \rd x \\
&+ (\zzz+\cust)\log(\zzz+\cust) - \zzz\log(-\zzz) + \zzz \log(\beta/(1-\beta)).  
\end{split}
\end{equation}
We first define this function $\Sfin$ for $\zzz \in (\Em,\Ep)$, and then analytically extend it to $\C \setminus ( (-\infty, \Em] \cup [\Ep, \infty) )$.
A key advantage of $\Sfin$ is that, while each of $\Dfin_1$ and $\Dfin_2$ should have three critical points very close to $\zzc$,
 $\Sfin$ has a triple critical point precisely at $\zzc$, due to our choice of $(\cusx,\cust)$ in Lemma \ref{lem:xtzdef}.

\begin{lem}   \label{lem:Sderivs}
We have $\Sfin'(\zzc)=\Sfin''(\zzc)=\Sfin'''(\zzc)=0$, and $\Sfin''''(\zzc)=-24\cust^{-4}A$. 
\end{lem}
\begin{proof}
The fact that  $\Sfin'(\zzc)=\Sfin''(\zzc)=\Sfin'''(\zzc)=0$ is follows straightforwardly from \eqref{eq:defcusx}, \eqref{eq:defcust}, and \eqref{eq:defzzzc}.
Also, we have
\[
\Sfin''''(\zzz)=\mmm'''(\zzz) + \frac{2}{(\cust+\zzz)^3} - \frac{2}{\zzz^3}=- \int \frac{6\den(x)}{(\zzz-x)^4} \rd x + \frac{2}{(\cust+\zzz)^3} - \frac{2}{\zzz^3},
\]
so $\Sfin''''(\zzc)=-24\cust^{-4}A$ by \eqref{eq:defA}. 
\end{proof}

Lemma \ref{lem:Sderivs} indicates that, when $|\z-\zzc|$ is small, we can approximate $\Sfin(\z)$ by $\Sfin(\zzc)-\cust^{-4}A(\z-\zzc)^4$. 
\begin{lem}  \label{lem:appxSzsm}
There exists a constant $c>0$, such that for any $\z \in \C$ with $|\z-\zzc|\le c\cust^{3/2}$, we have
\begin{align*}
\Sfin'(\z)+4\cust^{-4}A(\z-\zzc)^3 &\lesssim \cust^{-11/2}|\z-\zzc|^4,\\
\Sfin(\z)-\Sfin(\zzc)+\cust^{-4}A(\z-\zzc)^4 &\lesssim  \cust^{-11/2}|\z-\zzc|^5.
\end{align*}
\end{lem}
\begin{proof}
It suffices to bound the fifth derivative of $\Sfin$. 
For any $\z\in\C \setminus ( (-\infty, \Em] \cup [\Ep, \infty) )$, we have
\[
\Sfin'''''(\zzz)=\mmm''''(\zzz) - \frac{6}{(\cust+\zzz)^4} + \frac{6}{\zzz^4}.
\]
Take any $\zzz \in \C$ with $|\z-\zzc|\le c\cust^{3/2}$. For a sufficiently small $c$, we have $(\Re(\zzz)-\Em) \wedge (\Ep-\Re(\zzz))\asymp \cust^{3/2}$ by \eqref{eq:zbtcloser}. Then, with \eqref{eq:mmmfo} and \eqref{eq:zbtcloser}, we get that $\mmm''''(\zzz)\lesssim \cust^{-11/2}$ and  $- \frac{6}{(\cust+\zzz)^4} + \frac{6}{\zzz^4} \lesssim \cust^{-4}$. These two estimates imply that $\Sfin'''''(\zzz)\lesssim \cust^{-11/2}$, which give the bound $\Sfin''''(\z)- \Sfin''''(\zzc) \lesssim \cust^{-11/2}|\z-\zzc|$ for any $\zzz \in \C$ with $|\z-\zzc|\le c\cust^{3/2}$.
Then, the conclusion follows from the Taylor expansion of $\Sfin(z)$ and Lemma \ref{lem:Sderivs}.
\end{proof}
We next use Lemma \ref{lem:Sderivs} to deduce some results that were mentioned earlier.
\begin{lem}   \label{lem:Sderiv}
We have $A>0$ and $A\asymp 1$. In addition, we have $\Ep<0$, $\Em>-\cust$, and $\Ep, \Em+\cust\gtrsim \sct$.
\end{lem}
\begin{proof}
By \eqref{eq:mmmthi} and \eqref{eq:zbtcloser}, we have $\mmm'''(\zzc)<0$ and $\mmm'''(\zzc)\asymp \sct^{-4}$. If the following estimate holds,
\begin{equation}\label{eq:m3rd}
    -\mmm'''(\zzc)>\frac{3}{2}\left[\frac{2}{(\cust+\zzz_c)^3} - \frac{2}{\zzz_c^3}\right],
\end{equation}
then $\Sfin''''(\zzc)<0$ and $\Sfin''''(\zzc)\asymp \sct^{-4}$, which gives that $A>0$ and $A\asymp 1$.

Otherwise, if \eqref{eq:m3rd} fails, then we must have $\sct\asymp \cust\asymp 1$.
Define 
$$U_-=\int_{-\infty}^{\zzc} \frac{\den(x)}{(\zzc-x)^2} \rd x,\quad V_-=\int_{-\infty}^{\zzc} \frac{\den(x)}{(\zzc-x)^3} \rd x,\quad U_+=\int_{\zzc}^\infty \frac{\den(x)}{(x-\zzc)^2} \rd x,\quad V_+=\int_{\zzc}^\infty \frac{\den(x)}{(x-\zzc)^3} \rd x.$$
For simplicity of notations, we denote $P=(\cust+\zzc)^{-1}$ and $Q=-\zzc^{-1}$, which are positive by \eqref{eq:zbtcloser}.
Then, equations $\Sfin''(\zzc)=0$ and $\Sfin'''(\zzc)=0$ can be written as
\begin{equation}  \label{eq:UVrel}
U_-+U_+ = P+Q, \quad 2V_--2V_+ = P^2-Q^2.    
\end{equation}
Also, denote $G_-=2V_--U_-^2$ and $G_+=2V_+-U_+^2$ (which are positive by \Cref{lem:Sderivsh} below).
Then, 
\begin{equation}  \label{eq:Grel}
G_--G_+=P^2-Q^2-U_-^2+U_+^2.
\end{equation}
By Lemma \ref{lem:Sderivsh} below, using $V_-^2\ge U_-^4/4+U_-^2G_-/2$ and $V_+^2\ge U_+^4/4+U_+^2G_+/2$, we get that
\begin{equation}  \label{eq:Glowb}
\begin{split}
 -\Sfin''''(\zzc) & \ge 
 (U_-^3+U_+^3)/2 + 6(V_-^2/U_-+V_+^2/U_+) - 2P^3-2Q^3\\
& \ge 2(U_-^3+U_+^3) + 3(U_-G_- + U_+G_+) - 2P^3-2Q^3.
\end{split}    
\end{equation}
Without loss of generality, we assume that $P\le Q$. If $U_- \le P$ or $U_-\ge Q$, we have $-\Sfin''''(\zzc) \ge 3U_-G_- + 3U_+G_+$.
Otherwise, we have $U_-, U_+ \in (P, Q)$.
Then, we can write the last line of \eqref{eq:Glowb} as
 \[
\begin{split}
&2(U_-^3+U_+^3-P^3-Q^3) + 3 U_+(G_+-G_-) + 3 (U_-+ U_+)G_-\\
=&2(U_--P)(U_-^2+U_-P+P^2-U_+^2-U_+Q-Q^2 +3 U_+(U_-+U_+)) + 3(U_-+ U_+)G_-
\ge 3(U_-+ U_+)G_-,
\end{split}    
 \]
where we used \eqref{eq:Grel} and the first identity in \eqref{eq:UVrel} in the first step, and the second step simply uses $P\ge 0$ and $Q\le U_-+U_+$.
From Assumption \ref{ass:dtsncusp} and \eqref{eq:zbtcloser}, it is straightforward to check that similarly to \eqref{eq:mmmfir}, we have $U_-, U_+\asymp \sct^{-1}$.
Using Lemma \ref{lem:Sderivsh} below and the fact that $\den$ is supported in an order $1$ interval, we obtain that $G_-\gtrsim U_-\gtrsim \sct^{-1}$ and $G_+\gtrsim U_+\gtrsim \sct^{-1}$. 
Thus, we have $-\Sfin''''(\zzc)>0$ and $-\Sfin''''(\zzc)\gtrsim \sct^{-2}$. 
Since we have assumed that  $\sct\asymp \cust\asymp 1$, we still get that $A>0$ and $A\asymp 1$.

We now deduce that $\Ep<0$ with $\Ep\gtrsim \sct$. Otherwise, 
using that $\den(x)\in [0,1]$ and the growth of $\den$ in Assumption \ref{ass:dtsncusp},
we can deduce that $U_+<\int_0^\infty \frac{1}{(x-\zzc)^2} \rd x=Q$ and $V_+<\int_0^\infty \frac{1}{(x-\zzc)^3} \rd x=Q^2/2$.
Therefore, by \eqref{eq:UVrel}, we have $U_->P$ while $2V_-<P^2$, contradicting Lemma \ref{lem:Sderivsh} below. Similarly, we can deduce that $\Em>-\cust$ with $\Em+\cust\gtrsim \sct$. 
\end{proof}
The following elementary lemma is used in the proof of Lemma \ref{lem:Sderiv}.
\begin{lem}   \label{lem:Sderivsh}
Take any $r>0$ and function $\eta:(0, r)\to [0,1]$, denote $U=\int_0^r \frac{\eta(x)\rd x}{x^2}$, $V=\int_0^r \frac{\eta(x)\rd x}{x^3}$, and $W=\int_0^r \frac{\eta(x)\rd x}{x^4}$.
Then, we have $2V-U^2\ge 2Ur^{-1}$ and $W\ge U^3/12+V^2/U$.
\end{lem}
\begin{proof}
Note that $\eta\mapsto \int_0^r \frac{\eta(x)}{x^i}\rd x$, for $i\in\{2,3,4\}$, are linear functionals. It would then be straightforward to deduce that, given $U<\infty$, $V$ is minimized when $\eta$ is the indicator function of an interval $(a, r)$, for some $0<a<r$.
In this case, we have $U=a^{-1}-r^{-1}$ and $V=a^{-2}/2-r^{-2}/2$, so $2V-U^2=2(ar)^{-1}-2r^{-2}=2Ur^{-1}$.

Similarly, given $U,V<\infty$, $W$ is minimized when $\eta$ is the characteristic function of an interval $(a, b)$, for some $0<a<b\le \infty$. Then, we have $U=a^{-1}-b^{-1}$, $V=(a^{-2}-b^{-2})/2$, $W=(a^{-3}-b^{-3})/3$.
These imply that $a^{-1}=U/2+V/U$ and $b^{-1}=-U/2+V/U$, so $W=U^3/12+V^2/U$.
\end{proof}

\subsubsection{Discrete approximation}
We next use the above-obtained information on $\Sfin$ to extract properties of $\Dfin_2$ and $\Dfin_1$.
As a first step, we discretize $\Sfin$.
For the rest of this section, we denote $\oot_c=\lfloor\cust \lsca \rfloor$.
Define
\begin{equation}  \label{eq:defosfin}
 \oSfin(\z)=
    \frac{1}{\lsca}\sum_{i=-M}^N\log\Big(\z-\frac{\sfd_i}{\lsca}\Big)
    +\frac{1}{\lsca}\sum_{i=0}^{\oot_c-1}\log\Big(\z+\frac{i}{\lsca}\Big)
    -\frac{1}{\lsca}\log(\sin(\pi\lsca\z))+\z\log(\beta/(1-\beta)). 
\end{equation}
As in the case of $\Dfin_2$ and $\Dfin_1$, this expression defines a holomorphic function on $\HH$, up to adding a pure imaginary constant, and it can be analytically extended to $\R\setminus E(\oSfin)$ from $\HH$, where
\[
E(\oSfin)=\lsca^{-1}\left[(\{\sfd_i\}_{i=-M}^N \cap [-\oot_c+1, 0] ) \cup ( (-\infty,-\oot_c]\cup[1,\infty)   \setminus \{\sfd_i\}_{i=-M}^N)\right].
\]
We then choose the pure imaginary constant to ensure $\Im \oSfin=0$ in the interval $(\sfd_{-1}/\lsca, \sfd_0/\lsca)$. Finally, we analytically extend $\oSfin$ to $\HH^-$ from $\HH\cup (\sfd_{-1}/\lsca, \sfd_0/\lsca)$.
For the next two lemmas, we show that $\oSfin$ is a good approximation of $\Sfin$, by bounding the difference between their derivatives.
\begin{lem}  \label{lem:dercomp}
For any $\z\in \HH$, we have
\[
\Sfin'(\z)-\oSfin'(\z) \lesssim \frac{\lsca^{-1+\err}}{\inf_{x \in (-\infty, \Em\vee (\sfd_{-1}/\lsca)]\cup [\Ep\wedge (\sfd_0/\lsca), \infty)} |x-\z| }.
\]
\end{lem}
\begin{proof}
By \eqref{eq:defSf} and \eqref{eq:defosfin}, we can write $\Sfin'(\z)-\oSfin'(\z)$ for $\z\in\HH$ as
\begin{equation}  \label{eq:SoSddiff}
\int \frac{\den(x)}{\zzz-x} \rd x + \log(\zzz+\cust) - \log(-\zzz) - \sum_{i=-M}^N \frac{1}{\lsca\z - \sfd_i} - \sum_{i=0}^{\oot_c-1}\frac{1}{\lsca\z+i} + \pi\cot(\pi\lsca\z),    
\end{equation}
which is defined first for $(\sfd_{-1}/\lsca, \sfd_0/\lsca)\cap(\Em, \Ep)$, and then analytically extended to $\HH$.
We note that
\begin{equation}   \label{eq:cotapp}
\pi\cot(\pi\lsca\z) = \lim_{m\to\infty} \sum_{i\in\llbracket -m, m\rrbracket} \frac{1}{\lsca\z+i},    
\end{equation}
for any $\z\in\HH$, so we have
\begin{align*}
&\Big| \log(\zzz+\cust) - \log(-\zzz) - \sum_{i=0}^{\oot_c-1}\frac{1}{\lsca\z+i} + \pi\cot(\pi\lsca\z)  \Big| \\
=&
\lim_{m\to\infty} \Big| -\int_{-m}^0 \frac{1}{\lsca\z+x} \rd x - \int_{\lsca\cust}^{m} \frac{1}{\lsca\z+x} \rd x + \sum_{i=-m}^{-1}\frac{1}{\lsca\z+i} + \sum_{i=\oot_c}^{m}\frac{1}{\lsca\z+i}  \Big| \\
\le &\sum_{i=-\infty}^{-1} \Big| \int_i^{i+1}\frac{1}{\lsca\z+x}\rd x - \frac{1}{\lsca\z+i}\Big|
+ \sum_{i=\oot_c}^{\infty} \Big| \int_{i+\lsca\cust-\oot_c}^{i+\lsca\cust-\oot_c+1}\frac{1}{\lsca\z+x}\rd x - \frac{1}{\lsca\z+i}\Big|
\\
\le &\sum_{i=-\infty}^{-1} \frac{1}{\inf_{x\in [i,i+1]}|\lsca\z+x|^2}
+ \sum_{i=\oot_c}^{\infty} \frac{1}{\inf_{x\in [i,i+2]}|\lsca\z+x|^2}  \lesssim  \frac{1}{\lsca \inf_{x \in (-\infty,\oot_c/\lsca] \cup [0,\infty) } |x-\z| }.
\end{align*}
We next bound the remaining terms in \eqref{eq:SoSddiff}.
Recall the quantiles $\qut_i$, $i\in\qq{-M,N}$, defined in Assumption \ref{ass:dtsncusp}.
We have
\begin{equation}  \label{eq:rhodiffz}
\Big| \int \frac{\den(x)}{\zzz-x} \rd x - \sum_{i=-M}^N \frac{1}{\lsca\z - \sfd_i}   \Big|
\le \Big| \sum_{i=-M}^{N-1} \int_{\qut_i}^{\qut_{i+1}} \den(x)\Big(\frac{1}{\zzz-x} - \frac{1}{\z - \sfd_i/\lsca}\Big) \rd x  \Big| + \Big|\frac{1}{\lsca\z-\sfd_N}\Big| .    
\end{equation}
For each $i\in\llbracket -M, N-1\rrbracket$ and $x\in[\qut_i, \qut_{i+1}]$, we have
\[
\Big|\frac{1}{\zzz-x} - \frac{1}{\z - \sfd_i/\lsca} \Big| \le \int_{x\wedge (\sfd_i/\lsca)}^{x\vee (\sfd_i/\lsca)} \frac{\rd y}{|\z-y|^2}.
\]
Thus, \eqref{eq:rhodiffz} can be bounded by
\begin{equation}  \label{eq:rhodiffzb}
\sum_{i=-M}^{N-1} \int_{\qut_i\wedge (\sfd_i/\lsca)}^{\qut_{i+1}' \vee (\sfd_i/\lsca)} \frac{\rd y}{\lsca |\z-y|^2}+ \Big|\frac{1}{\lsca\z-\sfd_N}\Big|,
\end{equation}
where we denote $\qut_0'=\Em$ (in contrast to $\gamma_0=E_+$) and $\qut_i'=\qut_i$ for $i\in\llbracket-M, N\rrbracket\setminus\{0\}$.

For each $i\in\qq{-M, N}$ and $y\in [\qut_i\wedge (\sfd_i/\lsca), \qut_{i+1}' \vee (\sfd_i/\lsca)]$, we have either $\qut_i\le y \le \sfd_i/\lsca$ or $\sfd_i/\lsca \le y \le \qut_{i+1}'$.
We note that by \eqref{eq:heightclose}, for any $y\in (-\infty, \Em \vee (\sfd_{-1}/\lsca)] \cup [\Ep\wedge (\sfd_0/\lsca), \infty)$ we have
\[
|\{i\in \llbracket -M, N-1\rrbracket: \qut_i\le y \le \sfd_i/\lsca\}|+
|\{i\in \llbracket -M, N-1\rrbracket: \sfd_i/\lsca \le y \le \qut_{i+1}'\}| \lesssim   \lsca^{\err};
\]
and for $y\in (\Em \vee (\sfd_{-1}/\lsca), \Ep\wedge (\sfd_0/\lsca))$ we have
\[
|\{i\in \llbracket -M, N-1\rrbracket: \qut_i\le y \le \sfd_i/\lsca\}|+
|\{i\in \llbracket -M, N-1\rrbracket: \sfd_i/\lsca \le y \le \qut_{i+1}'\}|= 0.
\]
Thus, we can further bound \eqref{eq:rhodiffzb} by 
\[
\int_{(-\infty, \Em\vee (\sfd_{-1}/\lsca)]\cup [\Ep\wedge (\sfd_0/\lsca), \infty)} \frac{\lsca^{-1+\err}\rd y}{|\z-y|^2}+ \Big|\frac{1}{\lsca\z-\sfd_N}\Big|,
\]
so the conclusion follows.
\end{proof}
Combining Lemmas \ref{lem:appxSzsm} and \ref{lem:dercomp}, we can deduce the following result.
\begin{lem}  \label{lem:ospoc}
For any $\z\in\C$ with $|\z-\zzc|<c\cust^{3/2}$, we have that
\[
\oSfin(\z)-\oSfin(\zzc)+\cust^{-4}A(\z-\zzc)^4 \lesssim  \cust^{-11/2}|\z-\zzc|^5 + \lsca^{-1+\err} \cust^{-3/2}|\z-\zzc|.
\]
\end{lem}
\begin{proof}
By Lemma \ref{lem:dercomp}, for any $\z\in\C$ with $|\z-\zzc|<c\cust^{3/2}$, we have $|\oSfin'(\z)-\Sfin'(\z)| \lesssim \lsca^{-1+\err} \cust^{-3/2}$. Then, by integrating over $\z$, we get $|\oSfin(\z)-\Sfin(\z)-\oSfin(\zzc)+\Sfin(\zzc)| \lesssim \lsca^{-1+\err} \cust^{-3/2}|\z-\zzc|$. Together with Lemma \ref{lem:appxSzsm}, this concludes the proof.
\end{proof}

\subsubsection{Estimates of $\Dfin_1$ and $\Dfin_2$}
We next show that $\oSfin$ is close to $\Dfin_1$ and $\Dfin_2$.
\begin{lem}  \label{lem:dDS}
For any $\z\in\C$ with $|\z-\zzc|\le c\cust$, we have
\[
\oSfin'(\z)-\Dfin'_1(\z),\ \oSfin'(\z)-\Dfin'_2(\z) \lesssim \lsca^{-3/4}\cust^{-1} + \lsca^{-1/2}\cust^{-2}|\z-\zzc|.
\]
\end{lem}
\begin{proof}
For $\z\in\C$ with $|\z-\zzc|\le c\cust$, we can write that
\[
\oSfin'(\z)-\Dfin'_2(\z) = \sum_{i=0}^{\oot_c-1}\frac{1}{\lsca\z+i} - \sum_{i=-\oox_2+1}^{-\oox_2+\oot_2-1}\frac{1}{\lsca\z+i}.
\]
Since $\oox_2, \oot_2-\oot_c \lesssim \lsca^{1/2}$ and $|\z-\zzc|\le c\cust$ (so $\z, \z+\cust\asymp \cust$ according to \eqref{eq:zbtcloser}), we estimate it as
\begin{align*}
&\log\left(\frac{\lsca\z-\oox_2+1}{\lsca\z}\right) - \log\left(\frac{\lsca\z-\oox_2+\oot_2}{\lsca\z+\oot_c}\right)
+ \OO\Big(\int_{0\wedge(-\oox_2+1)}^{0\vee(-\oox_2+1)} \frac{\rd y}{|\lsca \z + y|^2}\Big)
+ \OO\Big(\int_{\oot_c\wedge(-\oox_2+\oot_2)}^{\oot_c\vee(-\oox_2+\oot_2)} \frac{\rd y}{|\lsca \z + y|^2}\Big)\\
=& 
\log\left(\frac{(\lsca\z-\oox_2+1)(\lsca\z+\oot_c)}{\lsca\z(\lsca\z-\oox_2+\oot_2)}\right) + \OO(\lsca^{-3/2}\cust^{-2}).
\end{align*}
We note that
\[
\log\left(\frac{(\lsca\z-\oox_2+1)(\lsca\z+\oot_c)}{\lsca\z(\lsca\z-\oox_2+\oot_2)}\right) = \log\left(1+\frac{\lsca\z(\oot_c-\oot_2+1) + (-\oox_2+1)\oot_c}{\lsca\z(\lsca\z-\oox_2+\oot_2)}\right).
\]
Again, using that $\oox_2, \oot_2-\oot_c \lesssim \lsca^{1/2}$ and $\z, \z+\cust\asymp \cust$, we get that
\[
\frac{\lsca\zzc(\oot_c-\oot_2+1) + (-\oox_2+1)\oot_c}{\lsca^2\cust^2} + \OO(\lsca^{-1/2}\cust^{-2}|\z-\zzc|) = \OO(\lsca^{-3/4}\cust^{-1} + \lsca^{-1/2}\cust^{-2}|\z-\zzc|),
\]
where we used that $\oox_2+\frac{\zzc}{\cust}(\oot_2-\oot_c) = \oox_2-B(\oot_2-\oot_c) \lesssim \lsca^{1/4}.$
Thus, we have
\[
\oSfin'(\z)-\Dfin'_2(\z) \lesssim \lsca^{-3/2}\cust^{-2} + \lsca^{-3/4}\cust^{-1} + \lsca^{-1/2}\cust^{-2}|\z-\zzc|.
\]
Since $\lsca^{-1/2+\errh}<\sct\asymp\cust$, we have $\lsca^{-3/2}\cust^{-2}\lesssim  \lsca^{-3/4}\cust^{-1}$, which gives the desired bound for $\oSfin'(\z)-\Dfin'_2(\z)$.
The bound for $\oSfin'(\z)-\Dfin'_1(\z)$ is proved similarly.
\end{proof}
Putting together Lemmas \ref{lem:appxSzsm}, \ref{lem:dercomp}, and \ref{lem:dDS}, we get the following estimate for $\Dfin'_1$ and $\Dfin'_2$ near $\zzc$.
\begin{lem}  \label{lem:dDzc}
For any $\z\in\C$ with $|\z-\zzc|\le c\cust^{3/2}$, we have
\[
\Dfin'_1(\z)+4\cust^{-4}A(\z-\zzc)^3,\ \Dfin'_2(\z)+4\cust^{-4}A(\z-\zzc)^3  \lesssim \lsca^{-3/4}\cust^{-1} + \lsca^{-1/2}\cust^{-2}|\z-\zzc| +  \cust^{-11/2}|\z-\zzc|^4.
\]
\end{lem}
Note that here we use that $\lsca^{-1+\err}\cust^{-3/2}\lesssim \lsca^{-3/4}\cust^{-1}$ (since $\lsca^{-1/2+\errh}<\sct\asymp\cust$, and $\err$ is chosen small enough depending on $\errh$).

\subsubsection{Away from $\zzc$} 
So far, we have obtained estimates on $\Dfin_1'$ and $\Dfin_2'$ near $\zzc$, using the approximations $\oSfin$ and $\Sfin$.
We will also need some estimates on $\Dfin_1'$ and $\Dfin_2'$ away from $\zzc$, which are stated as follows.
\begin{lem}  \label{lem:dD}
For any $\z\in \HH$, suppose that $x_*$ is the element in $E(\Dfin_2)$ with the smallest $|\z-x_*|$. Then, we have that
\[
\Dfin'_2(\z) - \frac{s}{\lsca(\z-x_*)} \lesssim \log \lsca ,
\]
for some $s\in\{1, -1\}$, depending on the residue of $\Dfin'_2$ at $x_*$.
A similar estimate holds for $\Dfin_1'$.
\end{lem}
\begin{proof}
We only prove the bound for $\Dfin'_2(\z)$, and the bound for $\Dfin'_1(\z)$ follows from a similar argument.
With the definition of $\Dfin_2$ and \eqref{eq:cotapp}, we can write that
\begin{align}
\Dfin'_2(\z) &= \sum_{i=-M}^N \frac{1}{\lsca\z - \sfd_i} + \sum_{i=-\oox_2+1}^{-\oox_2+\oot_2-1}\frac{1}{\lsca\z+i} - \pi\cot(\pi\lsca\z) + \log(\beta/(1-\beta)) \nonumber\\
&= \sum_{i=-M}^N \frac{1}{\lsca\z - \sfd_i} - \lim_{m\to\infty} \Big( \sum_{i=-m}^{-\oox_2}\frac{1}{\lsca\z+i} + \sum_{i=-\oox_2+\oot_2}^{m}\frac{1}{\lsca\z+i}\Big) + \log(\beta/(1-\beta)).\label{eq:D2'z}
\end{align}
Thus, we have
\[
\Big| \Dfin'_2(\z) - \frac{s}{\lsca(\z-x_*)}  \Big| \le
\sum_{x\in E(\Dfin_2)\cap (-C_0,C_0), x\neq x_*} \frac{1}{\lsca|\z-x|} + \lim_{m\to\infty}\Big| \sum_{i\in\qq{-m,m}, |i|\ge C_0\lsca, i\neq x_*\lsca} \frac{1}{\lsca \z-i}\Big| + |\log(\beta/(1-\beta))|,
\]
where $C_0>0$ is a large enough constant so that $-C_0\lsca < 2\sfd_{-M}<2\sfd_N<C_0\lsca$.
The first term is bounded by
\[
\sum_{x\in E(\Dfin_2)\cap (-C_0,C_0), x\neq x_*} \frac{1}{\lsca|\z-x|} \lesssim \sum_{i=1}^{\lceil 2C_0\lsca \rceil} \frac{1}{i-1/2} \lesssim \log \lsca.
\]
We also have that
\[
\sum_{i\in\qq{-m,m}, |i|\ge C_0\lsca, i\neq x_*\lsca} \frac{1}{\lsca \z-i}
= \int_{I} \frac{\rd y}{\lsca\z+y} 
+ \OO\Big( \int_{I} \frac{\rd y}{|\lsca\z+y|^2} \Big),
\]
where $I$ is the union of $[i-1, i]$ for all $i\in\qq{-m,m}, |i|\ge C_0\lsca, i< x_*\lsca$, and $[i, i+1]$ for all $i\in\qq{-m,m}, |i|\ge C_0\lsca, i> x_*\lsca$.
It is straightforward to check that 
\[
\lim_{m\to\infty} \int_{I} \frac{\rd y}{\lsca\z+y}  = \OO(\log \lsca), \quad \int_{I} \frac{\rd y}{|\lsca\z+y|^2} = \OO(1),
\]
and the conclusion follows.
\end{proof}

When $|\z|$ is large, we can directly approximate $\Dfin_1(\z)$ and $\Dfin_2(\z)$ using the linear function $\z\mapsto \z[\pi\i+\log(\beta/(1-\beta))]$, 
as follows.
\begin{lem}  \label{lem:Dlz}
For any $\z\in\HH$ with $|\z|>\lsca$, we have
\[
\Dfin_1(\z)-\z[\pi\i+\log(\beta/(1-\beta))],\ \Dfin_2(\z)-\z[\pi\i+\log(\beta/(1-\beta))]  \lesssim [-\lsca^{-1}\log(\min_{i\in\Z}|\lsca\z-i|)]\vee 0 +  \log |\z|   .
\]
\end{lem}
\begin{proof}
We only prove the bound for $\Dfin_2$, while the proof for $\Dfin_1$ is similar.
Since $M, N\asymp \lsca$, $|\sfd_{-M}|, |\sfd_N| \asymp \lsca$, $|\oox_2|, |\oot_2| \lesssim \lsca$, and $|\z|>\lsca$, we have
\begin{equation}  \label{eq:logs}
\frac{1}{\lsca}\sum_{i=-M}^N\log\Big(\z-\frac{\sfd_i}{\lsca}\Big)
    +\frac{1}{\lsca}\sum_{i=-\oox_2+1}^{-\oox_2+\oot_2-1}\log\Big(\z+\frac{i}{\lsca}\Big) \lesssim \log |\z|,    
\end{equation}
where the left-hand side is defined to be holomorphic in $\HH$ and real near $(2\lsca)^{-1}$.
Next, we consider $-\frac{1}{\lsca}\log(\sin(\pi\lsca\z))$, which is defined first for $\z\in (0, \lsca^{-1})$, and then analytically extended to $\C\setminus ((-\infty, 0]\cup [\lsca^{-1}, \infty))$.
We note that $\frac{1}{\lsca}\log(\sin(\pi\lsca\z)) + \z\pi\i$ is periodic: 
for any $\z\in\HH$, we have $\frac{1}{\lsca}\log(\sin(\pi\lsca\z)) + \z\pi\i=\frac{1}{\lsca}\log(\sin(\pi\lsca(\z+2\lsca^{-1}))) + (\z+2\lsca^{-1})\pi\i$.
This gives $\Im[\frac{1}{\lsca}\log(\sin(\pi\lsca\z)) +\z \pi\i] \lesssim \lsca^{-1}$ for any $\z\in \HH$.
For $\Re[\frac{1}{\lsca}\log(\sin(\pi\lsca\z)) + \z\pi\i]$, it is equal to
\[
\frac{1}{\lsca}\log(|\sin(\pi\lsca\z )e^{\i\pi\lsca\z}|) = \frac{1}{\lsca}\log\Big(\Big|\frac{e^{2\pi\i\lsca\z} - 1}{2\i}\Big|\Big),
\]
which is of order $\OO(\lsca^{-1})$ when $\Im \z >\lsca^{-1}/5$, and of order $\OO(\lsca^{-1}(1-\log(\min_{i\in \Z}|\lsca\z-i|)))$ when $0<\Im(\z)\le \lsca^{-1}/5$.
Thus, we conclude that for any $\z\in \HH$,
\[
\frac{1}{\lsca}\log(\sin(\pi\lsca\z)) +\z \pi\i
\lesssim  [-\lsca^{-1}\log(\min_{i\in\Z}|\lsca\z-i|)]\vee 0  + \lsca^{-1}.
\]
Finally, to get $\Dfin_2$ from the left-hand side of \eqref{eq:logs} (which is taken to be real near $(2\lsca)^{-1}$) and $-\frac{1}{\lsca}\log(\sin(\pi\lsca\z))$ (which is taken to be real in $(0, \lsca^{-1})$), one only needs to add a pure imaginary number, which is $\OO(1)$ since $M, N\asymp \lsca$. 
Therefore, the conclusion follows.
\end{proof}

\subsection{Contour deformation}
To obtain the estimates of kernels and prove Proposition \ref{prop:Kconvconj}, we will use the steepest descent method. For this purpose, we need to deform the contours for $K^{\textnormal{Bernoulli}}$ in \eqref{eq:kberD}.
In this subsection, we construct these deformed contours.

From the above computations on $\Dfin_1$ and $\Dfin_2$, we expect to deform the contours so that both pass through $\zzc$, and the integrand in \eqref{eq:kberD} decays fast for $\z$ and $\w$ away from $\zzc$.
Specifically, we consider the contours inside and outside $\{\z\in\C: |\z-\zzc|\le \lsca^{-1/4+\ere}\cust\}$ separately. (Recall that $\ere$ is one of the parameters defined at the beginning of this section.)
We now record a useful lemma.
\begin{lem}  \label{lem:dDzce}
For any $\z\in\C$ with $|\z-\zzc|\lesssim  \lsca^{-1/4+\ere}\cust$, we have
\[
\Dfin'_1(\z)+4\cust^{-4}A(\z-\zzc)^3, \;\Dfin'_2(\z)+4\cust^{-4}A(\z-\zzc)^3  \lesssim \lsca^{-3/4+\ere}\cust^{-1},
\]
\[
\Dfin_1(\z)-\Dfin_1(\zzc)+\cust^{-4}A(\z-\zzc)^4, \;\Dfin_2(\z)-\Dfin_2(\zzc)+\cust^{-4}A(\z-\zzc)^4  \lesssim \lsca^{-1+2\ere}.
\]
\end{lem}
The first estimate is directly implied by Lemma \ref{lem:dDzc} and the facts that  $\lsca^{-1/2+\errh}<\sct\asymp\cust$. 
The second estimate is obtained by integrating over $\z$.

For the rest of this section, we use $[\z_1\to\cdots\to\z_k]$ to denote the contour obtained by connecting $\z_1, \ldots, \z_k \in \C$ sequentially using straight line segments.
In such notations, we may also take $\z_1$ or $\z_k$ to be $\infty e^{\pi\i\theta}$ for some $\theta \in \R$, in which case the first or last segment is an infinite straight line in the corresponding direction.
Our first step is the following deformation.
\begin{lem}  \label{lem:deforma}
For \eqref{eq:kberD}, the contours can be replaced by the followings: the $\w$ contour is the union of
\begin{align}
\label{eq:wconta}
[\zzc\to \zzc+e^{3\pi\i/4} \lsca^{-1/4+\ere}\cust &\to -\cust-1 \to \zzc+e^{5\pi\i/4} \lsca^{-1/4+\ere}\cust \to \zzc],    \\
\label{eq:wcontb}
[\zzc \to \zzc+e^{7\pi\i/4} \lsca^{-1/4+\ere}\cust &\to 1 \to \zzc+e^{\pi\i/4} \lsca^{-1/4+\ere}\cust \to \zzc],
\end{align}
and the $\z$ contour is the straight vertical line passing through $\zzc$, traversed upwards.
\end{lem}
\begin{proof}
We first assume that the contour of $\w$ in \eqref{eq:kberD} is taken to be small circles around the $\w$ poles.
Then we fix $\w$ and deform the contour of $\z$, from the vertical line through $(\oox_2-\oot_2+1/2)/\lsca$ to the vertical line through $\zzc$.
It is straightforward to check that, by Lemma \ref{lem:Dlz}, the integral over $\z$ along $[(\oox_2-\oot_2+1/2)/\lsca+\i X\to \zzc+\i X]$ for some $X\in\R$ would $\to 0$ as $|X| \to \infty$.
Thus, it remains to consider the poles of $\z$.

We note that there is no pole of $\z$ in $[(\oox_2-\oot_2+1/2)/\lsca, \zzc]$, except for $\w$ (when $\w$ is in a small circle around a point in this interval).
For the residue at $\z=\w$, it can be written as a coefficient multiplying
\[
\oint\limits_{\textnormal{all $\w$ poles in }[(\oox_2-\oot_2+1/2)/\lsca, \zzc]}
\frac{(\lsca\w-\oox_2+1)_{\oot_2-1}}{(\lsca\w-\oox_1)_{\oot_1+1}}\rd\w,
\]
which vanishes since the integrand has no $w$ pole in $[(\oox_2-\oot_2+1/2)/\lsca, \zzc]$.
Thus we are done with deforming the contour of $\z$.

For $\w$, the contours \eqref{eq:wconta} and \eqref{eq:wcontb}  enclose all the $\w$ poles $(\{\oox_1-\oot_1,\oox_1-\oot_1+1,\ldots,\oox_1-1,\oox_1\}\cap \{\sfd_{-M},\ldots,\sfd_N\})/\lsca$, since $-\lsca\cust-\lsca<\oox_1-\oot_1<\oox_1<\lsca$.
Hence, $\w$ can be integrated along the union of \eqref{eq:wconta} and \eqref{eq:wcontb}.
\end{proof}

\subsubsection{Steepest descent curves}   \label{sssec:stpcur}
For the part of the contours \eqref{eq:wconta}, \eqref{eq:wcontb} and $\{\z\in\C: \Re(\z)=\zzc\}$ outside $\{\z\in\C: |\z-\zzc|\le \lsca^{-1/4+\ere}\cust\}$, 
we will further deform them to follow the steepest descent curves of $\Re(\Dfin_1)$ and $\Re(\Dfin_2)$.
For this purpose, we need to analyze the critical points of $\Dfin_1$ and $\Dfin_2$.
Define 
\[\Gamma_l=(-\infty, \zzc-\lsca^{-1/4+\ere}\cust]\cup [\zzc+\lsca^{-1/4+\ere}\cust, \infty),\;\Gamma_c=\{\z:\Im(\z)\ge 0, |\z-\zzc|=\lsca^{-1/4+\ere}\cust\},\]
and let $\Gamma=\Gamma_l\cup\Gamma_c$.
Let $U=\{\z:\Im(\z)\ge 0, |\z-\zzc|\ge \lsca^{-1/4+\ere}\cust\}$.
Note $\Gamma$ is the boundary of $U$.
\begin{lem}  \label{lem:Dnz}
The functions $\Dfin_1$ and $\Dfin_2$ have no critical point in the interior of $U$.
\end{lem}
\begin{proof}
Recall that $\Dfin'_2(\z)$ can be written as \eqref{eq:D2'z} for $\z\in\HH$.
By Hurwitz's theorem, it suffices to show that for all large enough $m$,
\begin{equation}  \label{eq:DXapp}
\sum_{i=-M}^N\frac{1}{\lsca\z-\sfd_i}
    - \sum_{i=-m}^{-\oox_2}\frac{1}{\lsca\z+i} - \sum_{i=-\oox_2+\oot_2}^{m}\frac{1}{\lsca\z+i}+\log \frac{\beta}{1-\beta}    
\end{equation}
has no zero in the interior of $U$.
For this purpose, we multiply \eqref{eq:DXapp} by $\prod_{x\in E(\Dfin_2)\cap[-m/\lsca,m/\lsca]} (\z-x)$, and obtain a polynomial with degree at most $|E(\Dfin_2)\cap[-m/\lsca,m/\lsca]|$.
So \eqref{eq:DXapp} has at most $|E(\Dfin_2)\cap[-m/\lsca,m/\lsca]|$ many zeros.

Consider the poles of \eqref{eq:DXapp}, i.e., $E(\Dfin_2)\cap[-m/\lsca,m/\lsca]$, which divide $\R$ into $|E(\Dfin_2)\cap[-m/\lsca,m/\lsca]|+1$ many intervals. By \eqref{eq:vecal}, except for at most four of these intervals (the left-most and right-most open intervals, and two intervals in the middle where the residues of the poles change signs), there is at least one zero in each interval. 
By Rouch{\'e}'s theorem and Lemma \ref{lem:dDzce}, we see that there are precisely three zeros inside $\{\z\in\C: |\z-\zzc|< \lsca^{-1/4+\ere}\cust\}$. 
Now, we have found at least $|E(\Dfin_2)\cap[-m/n,m/n]|$ many zeros of \eqref{eq:DXapp} in $\R\cup\{\z\in\C: |\z-\zzc|< \lsca^{-1/4+\ere}\cust\}$. Hence, $\Dfin_2'$ has no zero in the interior of $U$.

The statement for $\Dfin_1$ follows a similar argument.
\end{proof}

Since $\Re(\Dfin_2)$ and $\Im(\Dfin_2)$ are harmonic conjugates, the steepest descent curves of $\Re(\Dfin_2)$ starting from around $\zzc$ are given by the set $\Im \Dfin_2 =0$, which can be described as follows.
\begin{lem}  \label{lem:curveDo}
The set $\mathfrak A=\{\z\in \C\setminus E(\Dfin_2):\Im(\Dfin_2(\z))=0\}\cap U$ contains the following parts (cf.~Figure \ref{fig:stec}):
\begin{enumerate}
\item $\Gamma_l \cap (\sfd_{-1}/\lsca, \sfd_0/\lsca)$;
\item Two open intervals $P_{2,-} \subset (-\infty, (\oox_2-\oot_2)/\lsca)$ and $P_{2,+} \subset (\oox_2/\lsca, \infty)$, defined as
\[
P_{2,-}:=\left\{x\in \mathfrak A , x<(\oox_2-\oot_2)/\lsca: |\{-M\le i\le -1: \sfd_i > x\lsca \}| = |\{i\in \Z: x\lsca< i\le \oox_2-\oot_2 \}|\right\},
\]
\[
P_{2,+}:=\left\{x\in \mathfrak A , x>\oox_2/\lsca: |\{0\le i\le N: \sfd_i < x\lsca \}| = |\{i\in \Z: \oox_2  \le i<x\lsca\}|\right\}.
\]
By \eqref{eq:vecal}, $P_{2,-}$ and $P_{2,+}$ are non-empty finite intervals.
\item Three disjoint, smooth, and self-avoiding curves $\oell_{2,1}$, $\oell_{2,2}$, and $\oell_{2,3}$, such that $\oell_{2,1}$ is from $\zzc+e^{\i \theta_1} \lsca^{-1/4+\ere}\cust\in\Gamma_c$ to some $\z_-\in P_{2,-}$, $\oell_{2,3}$ is from $\zzc+e^{\i \theta_3} \lsca^{-1/4+\ere}\cust\in\Gamma_c$ to some $\z_+\in P_{2,+}$, and $\oell_{2,2}$ is from $\zzc+e^{\i \theta_2} \lsca^{-1/4+\ere}\cust\in \Gamma_c$ to $\infty$. 
Here, $\theta_1, \theta_2, \theta_3 \in \R$ and satisfy 
\begin{equation}  \label{eq:thetat}
|\theta_3-\pi/4|, |\theta_2-\pi/2|, |\theta_1-3\pi/4| \lesssim \lsca^{-2\ere}.
\end{equation}
Except for the endpoints, these curves ($\oell_{2,1}$, $\oell_{2,2}$, and $\oell_{2,3}$) are contained in the interior of $U$. 
\end{enumerate}
Starting from $\Gamma_c$, $\Re \Dfin_2 $ is strictly decreasing along $\oell_{2,2}$, and strictly increasing along each of $\oell_{2,1}$ and $\oell_{2,3}$. Moreover, we have
\begin{equation}\label{eq:inicur}
\begin{split}  
&\Dfin_2(\z_c+e^{\i \theta_1} \lsca^{-1/4+\ere}\cust)-\Dfin_2(\zzc)-A\lsca^{-1+4\ere} \lesssim \lsca^{-1+2\ere}, \\
&\Dfin_2(\z_c+e^{\i \theta_2} \lsca^{-1/4+\ere}\cust)-\Dfin_2(\zzc)+A\lsca^{-1+4\ere} \lesssim \lsca^{-1+2\ere}, \\
&\Dfin_2(\z_c+e^{\i \theta_3} \lsca^{-1/4+\ere}\cust)-\Dfin_2(\zzc)-A\lsca^{-1+4\ere} \lesssim \lsca^{-1+2\ere}.
\end{split}    
\end{equation}
Similar statements hold for $\Dfin_1$ and the set $\{\z\in \C\setminus E(\Dfin_1):\Im(\Dfin_1(\z))=0\}\cap U$.
\end{lem}
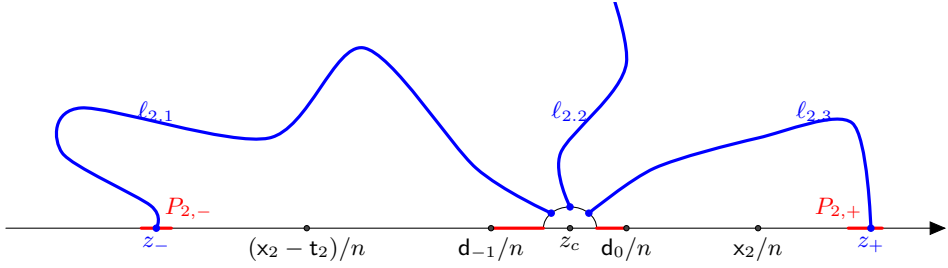
\begin{figure}[hbt!]
    \centering
\begin{tikzpicture}[line cap=round,line join=round,>=triangle 45,x=0.25cm,y=0.2cm]
\clip(-1,-2.2) rectangle (51,15);

\draw [->] (0,0) -- (50,0);

\begin{scope}
    \clip (28,0) rectangle (32,2);
    \draw (30,0) circle (1.4142);
\end{scope}

\draw [red] [very thick] (25.8,0) -- (30-1.4142,0);
\draw [red] [very thick] (30+1.4142,0) -- (33,0);
\draw [red] [very thick] (7.2,0) -- (8.8,0);
\draw [red] [very thick] (44.8,0) -- (46.6,0);

\begin{small}
\draw (30,0) node[anchor=north]{$\zzc$};

\draw (25.8,0) node[anchor=north]{$\sfd_{-1}/\lsca$};
\draw (33,0) node[anchor=north]{$\sfd_0/\lsca$};

\draw (16,0) node[anchor=north]{$(\oox_2-\oot_2)/\lsca$};
\draw (40,0) node[anchor=north]{$\oox_2/\lsca$};

\draw [blue] (8,0) node[anchor=north]{$\z_-$};
\draw [blue] (46,0) node[anchor=north]{$\z_+$};

\draw [red] (8,0) node[anchor=south west]{$P_{2,-}$};
\draw [red] (46,0) node[anchor=south east]{$P_{2,+}$};

\draw [blue] (8,9) node[anchor=north]{$\oell_{2.1}$};
\draw [blue] (30,9) node[anchor=north]{$\oell_{2.2}$};
\draw [blue] (43,9) node[anchor=north]{$\oell_{2.3}$};
\end{small}

\draw [blue] [very thick] plot [smooth] coordinates {(8,0) (7.7,1.5) (3,5) (4,8) (14,6) (19,12) (25,4) (29,1)};

\draw [blue] [very thick] plot [smooth] coordinates {(31,20) (33,11) (29.5,5) (30,1.4142)};

\draw [blue] [very thick] plot [smooth] coordinates {(46,0) (45,7) (40,6) (34.5,4) (31,1)};

\draw [blue] [fill=blue] (8,0) circle (1.2pt);
\draw [blue] [fill=blue] (46,0) circle (1.2pt);

\draw [blue] [fill=blue] (29,1) circle (1.2pt);
\draw [blue] [fill=blue] (31,1) circle (1.2pt);
\draw [blue] [fill=blue] (30,1.4142) circle (1.2pt);

\draw [fill=uuuuuu] (30,0) circle (1pt);

\draw [fill=uuuuuu] (25.8,0) circle (1pt);
\draw [fill=uuuuuu] (33,0) circle (1pt);

\draw [fill=uuuuuu] (16,0) circle (1pt);
\draw [fill=uuuuuu] (40,0) circle (1pt);

\end{tikzpicture}
\caption{An illustration of the set $\{\z\in \C\setminus E(\Dfin_2):\Im(\Dfin_2(\z))=0\}\cap U$.}
\label{fig:stec}
\end{figure}
\begin{proof}
It is straightforward to check (from the definition of $\Dfin_2$) that for $\z\in \Gamma_l \cap (\sfd_{-1}/\lsca, \sfd_0/\lsca)$ or $\z \in P_{2,-}\cup P_{2,+}$, we have $\Im(\Dfin_2(\z))=0$, and $\Im(\Dfin_2(\z))\neq 0$ for all other $\z \in \Gamma_l \setminus E(\Dfin_2)$.

We next consider the half-circle $\Gamma_c$.
Define the function $f:[0, \pi]\to \R$ as
\[
f:\theta \mapsto \Im(\Dfin_2(\zzc+e^{\i \theta} \lsca^{-1/4+\ere}\cust)).
\]
Note that $f(0)=f(\pi)=0$.
By Lemma \ref{lem:dDzce}, we have
\[
f'(\theta) + 4\cos(4\theta) \lsca^{-1+4\ere}A \lesssim \lsca^{-1+2\ere}.
\]
With this estimate, we can conclude that $f$ has five zeros $0<\theta_3<\theta_2<\theta_1<\pi$, and they satisfy \eqref{eq:thetat} and \eqref{eq:inicur}.
Using Lemma \ref{lem:dDzce} again, we get that
\begin{equation}\label{eq:devDt}
\begin{split}  
\Dfin'_2(\zzc+e^{\i \theta_3} \lsca^{-1/4+\ere}\cust)& = -4A e^{3\i \pi/4} \lsca^{-3/4+3\ere}\cust^{-1}(1+\OO(\lsca^{-2\ere})),\\
\Dfin'_2(\zzc+e^{\i \theta_2} \lsca^{-1/4+\ere}\cust)& = -4A e^{3\i \pi/2} \lsca^{-3/4+3\ere}\cust^{-1}(1+\OO(\lsca^{-2\ere})),\\
\Dfin'_2(\zzc+e^{\i \theta_1} \lsca^{-1/4+\ere}\cust)& = -4A e^{\i \pi/4} \lsca^{-3/4+3\ere}\cust^{-1}(1+\OO(\lsca^{-2\ere})).
\end{split}    
\end{equation}

It is straightforward to check that $\Dfin''_2>0$ on $P_{2,-}$ and $P_{2,+}$, so $\Dfin_2$ has a unique critical point inside each of $P_{2,-}$ and $P_{2,+}$, and we denote them by $\z_-$ and $\z_+$, respectively.

As $\Dfin_2$ is holomorphic and contains no critical point in the interior of $U$ by Lemma \ref{lem:Dnz}, 
for each $\z_*$ in the interior of $U$, if $\Im(\Dfin_2(\z_*))=0$, one can then take the steepest descent curve of $\Re \Dfin_2 $ from $\z_*$.
Along this curve $\Im \Dfin_2 =0$, and $\Re \Dfin_2$ is strictly monotone due to the absence of critical points.
This curve inside $U$ is smooth and self-avoiding; in each direction, it either ends at one of $\zzc+e^{\i \theta_1} \lsca^{-1/4+\ere}\cust$, $\zzc+e^{\i \theta_2} \lsca^{-1/4+\ere}\cust$, $\zzc+e^{\i \theta_3} \lsca^{-1/4+\ere}\cust$, or ends at a critical point of $\Dfin_2$ in $P_{2,-}$ or $P_{2,+}$ (i.e., one of $\z_-$ and $ \z_+$), or goes to $\infty$. All such curves do not intersect each other.

Consider the $\Re\Dfin_2$ steepest descent curves starting from $\zzc+e^{\i \theta_1} \lsca^{-1/4+\ere}\cust$, $\zzc+e^{\i \theta_2} \lsca^{-1/4+\ere}\cust$, $\zzc+e^{\i \theta_3} \lsca^{-1/4+\ere}\cust$, and $\z_-$, $\z_+$, towards the interior of $U$.
Due to the above estimates \eqref{eq:devDt} on $\Dfin'_2$, and the fact $\Dfin''_2>0$ on $P_{2,-}$ and $P_{2,+}$, we observe that $\Re\Dfin_2$ is decreasing along the curves from $\zzc+e^{\i \theta_2} \lsca^{-1/4+\ere}\cust$ and $\z_-$, $\z_+$, while $\Re\Dfin_2$ is increasing along the curves from $\zzc+e^{\i \theta_1} \lsca^{-1/4+\ere}\cust$ and $\zzc+e^{\i \theta_3} \lsca^{-1/4+\ere}\cust$.

By Lemma \ref{lem:Dlz}, we have that any $\Im \Dfin_2 =0$ curve in the ascending direction of $\Re\Dfin_2$ cannot go to $\infty$ in $U$, and thus must terminate at one of $\zzc+e^{\i \theta_2} \lsca^{-1/4+\ere}\cust, \z_-, \z_+$.
These imply that the $\Re \Dfin_2 $ steepest descent curves from $\zzc+e^{\i \theta_2} \lsca^{-1/4+\ere}\cust, \z_-, \z_+$ are all such $\Im \Dfin_2 =0$ curves,
and we denote them by $\oell_{2,2}$, $\oell_{2,1}$, and $\oell_{2,3}$.

For the $\Re(\Dfin_2)$ steepest descent curves from $\zzc+e^{\i \theta_1} \lsca^{-1/4+\ere}\cust$ and $\zzc+e^{\i \theta_3} \lsca^{-1/4+\ere}\cust$, they cannot end at $\zzc+e^{\i \theta_2} \lsca^{-1/4+\ere}\cust$.
This is because, by \eqref{eq:inicur} we have that
\[
\Dfin_2(\zzc+e^{\i \theta_1} \lsca^{-1/4+\ere}\cust) - \Dfin_2(\zzc+e^{\i \theta_2} \lsca^{-1/4+\ere}\cust) - 2A\lsca^{-1+4\ere} \lesssim \lsca^{-1+2\ere},
\]
so $\Dfin_2(\zzc+e^{\i \theta_1} \lsca^{-1/4+\ere}\cust) > \Dfin_2(\zzc+e^{\i \theta_2} \lsca^{-1/4+\ere}\cust)$; similarly, we have $\Dfin_2(\zzc+e^{\i \theta_3} \lsca^{-1/4+\ere}\cust) > \Dfin_2(\zzc+e^{\i \theta_2} \lsca^{-1/4+\ere}\cust)$.
Thus, we conclude that $\oell_{2,2}$ must go to $\infty$ in $U$.
Since $\oell_{2,2}$ does not intersect $\oell_{2,1}$ or $\oell_{2,3}$ in the interior of $U$, we must have that $\oell_{2,1}$ connects $\z_-$ and $e^{\i \theta_1} \lsca^{-1/4+\ere}\cust$, and $\oell_{2,3}$ connects $\z_+$ and $e^{\i \theta_3} \lsca^{-1/4+\ere}\cust$.
\end{proof}
For the convenience of later applications, for the similar statement on $\Dfin_1$, we denote the corresponding intervals as $P_{1,-}$ and $P_{1,+}$, and the corresponding curves as $\oell_{1,1}$ from $\zzc+e^{\i \vartheta_1} \lsca^{-1/4+\ere}\cust\in\Gamma_c$ to some $\w_-\in P_{1,-}$, $\oell_{1,3}$ from $\zzc+e^{\i \vartheta_3} \lsca^{-1/4+\ere}\cust\in\Gamma_c$ to some $\w_+\in P_{1,+}$, and $\oell_{1,2}$ from $\zzc+e^{\i \vartheta_2} \lsca^{-1/4+\ere}\cust\in \Gamma_c$ to $\infty$. 
Here, $\vartheta_1, \vartheta_2, \vartheta_3$ are real numbers satisfying that 
\[
|\vartheta_3-\pi/4|, |\vartheta_2-\pi/2|, |\vartheta_1-3\pi/4| \lesssim \lsca^{-2\ere}.
\]

We next provide a technical lemma that will be used later.
\begin{lem}  \label{lem:oodis}
The curve $\oell_{1,1}$ (resp. $\oell_{1,2}$, $\oell_{1,3}$, $\oell_{2,1}$, $\oell_{2,2}$, $\oell_{2,3}$) is disjoint from the $\lsca^{-3}$ neighborhood of $\R\setminus P_{1,-}$ (resp. $\R$, $\R\setminus P_{1,+}$, $\R\setminus P_{2,-}$, $\R$, $\R\setminus P_{2,+}$).
\end{lem}
\begin{proof}
We first show that $\oell_{2,1}$ is disjoint from the $\lsca^{-2}$ neighborhood of $E(\Dfin_2)$.
For any $x_*\in E(\Dfin_2)$, take $\z\in\HH$ such that $|\z-x_*|<\frac{1}{2\lsca}$. By Lemma \ref{lem:dD}, we have that
\[
\Dfin'_2(\z) = \frac{s}{\lsca(\z-x_*)} + \OO(\log \lsca ),
\]
where $s\in\{1,-1\}$, depending on the residue of $\Dfin'_2$ at $x_*$.
Without loss of generality, assume that $\Re \z \ge x_*$.
Then, by integrating over $z$, we get that
\[
\Im(\Dfin_2(\z)) - \Im(\Dfin_2(x_*+(2\lsca)^{-1})) = \lsca^{-1}s\arg(\z-x_*) + \OO(\log\lsca\cdot \im z),
\]
where $\arg$ takes value in $(0, \pi/2]$ as $\z-x_*\in\HH$ and $\Re(\z-x_*)\ge 0$. 
Note that $\Im\z< \arg(\z-x_*)|\z-x_*|$. Hence, when $|\z-x_*|\lesssim \lsca^{-2}$, we have that
\[
\Im(\Dfin_2(\z)) - \Im(\Dfin_2(x_*+(2\lsca)^{-1})) = \lsca^{-1}s\arg(\z-x_*) (1+\OO(\log(\lsca)/\lsca)) \in (-\pi\lsca^{-1}, \pi\lsca^{-1})\setminus\{0\}.
\]
Since $x_*+(2\lsca)^{-1}\in\R$, we have $\Im(\Dfin_2(x_*+(2\lsca)^{-1}))\in \pi\lsca^{-1}\Z$. Thus, we conclude that $\Im(\Dfin_2(\z))\not\in \pi\lsca^{-1}\Z$ whenever $|\z-x_*|\lesssim \lsca^{-2}$.
This implies that $\oell_{2,1}$ is at least $\lsca^{-2}$ away from $E(\Dfin_2)$.

For any $\z\in\HH$ and $y\in \R$, by Lemma \ref{lem:dD}, integrating $\Dfin_2'(z)$ along a curve connecting $y$ and $\z$ yields that
\begin{equation}  \label{eq:appdDd}
\Dfin_2(\z)-\Dfin_2(y) \lesssim \frac{|\z-y|}{\inf_{x\in E(\Dfin_2)}\lsca(|\z-x|\wedge |y-x|)} + \log(\lsca)|\z-y|.    
\end{equation}
If $\z$ is in the $\lsca^{-3}$ neighborhood of $\R\setminus (P_{2,-}\cup P_{2,+}\cup (\sfd_{-1}/\lsca,\sfd_0/\lsca))$ but not the $\lsca^{-2}$ neighborhood of $E(\Dfin_2)$, then we must have that $\inf_{x\in E(\Dfin_2)} |\Re\z-x| \gtrsim \lsca^{-2}$ and $\Re\z\in \R\setminus (P_{2,-}\cup P_{2,+}\cup (\sfd_{-1}/\lsca,\sfd_0/\lsca))$, which give that $|\Im(\Dfin_2(\Re \z))|\ge \pi\lsca^{-1}$. Then, \eqref{eq:appdDd} for $y=\Re\z$ implies that $|\Dfin_2(\z)-\Dfin_2(\Re z)|\lesssim \lsca^{-2}$, so $\Im(\Dfin_2(\z))\neq 0$, meaning that $\z\not\in\oell_{2,1}$.

So far, we have shown that $\oell_{2,1}$ is disjoint from the $\lsca^{-3}$ neighborhood of $\R\setminus (P_{2,-}\cup P_{2,+}\cup (\sfd_{-1}/\lsca,\sfd_0/\lsca))$, and the $\lsca^{-2}$ neighborhood of $E(\Dfin_2)$.
Similarly, these statements also hold for $\oell_{2,2}$ and  $\oell_{2,3}$.

We next consider $\z\in\HH$ in the $\lsca^{-3}$ neighborhood of $(\sfd_{-1}/\lsca, \sfd_0/\lsca)$, while $|\z-\zzc|>\lsca^{-1/4+\ere}\cust$ and $|\z-\sfd_{-1}/\lsca|, |\z-\sfd_0/\lsca|\ge \lsca^{-2}$.
Take $y$ to be some point in $(\sfd_{-1}/\lsca, \zzc-\lsca^{-1/4+\ere}\cust)\cup (\zzc+\lsca^{-1/4+\ere}\cust, \sfd_0/\lsca)$ with $\z-y\lesssim \lsca^{-3}$.
By Lemma \ref{lem:dDzce}, we have that $\Dfin_2'>0$ in $(\sfd_{-1}/\lsca, \zzc-\lsca^{-1/4+\ere}\cust)$ and $\Dfin_2'<0$ in $(\zzc+\lsca^{-1/4+\ere}\cust, \sfd_0/\lsca)$, and that $\Dfin_2(y)-\Dfin_2(\zzc)<-c\lsca^{-1+4\ere}$ for a constant $c>0$.
So by \eqref{eq:appdDd}, we have 
\[\Re(\Dfin_2(\z)) < \Re(\Dfin_2(y)) + \OO(\lsca^{-2}) < \Re(\Dfin_2(\zzc))-c\lsca^{4\ere-1}+ \OO(\lsca^{-2}).
\]
On the other hand, by Lemma \ref{lem:curveDo} 
(more precisely, \eqref{eq:inicur} and the fact that $\Re(\Dfin_2)$ is increasing along $\oell_{2,1}$ and $\oell_{2,3}$ from $\Gamma_c$), we conclude that $\z\not\in \oell_{2,1}\cup \oell_{2,3}$.
In other words, $\oell_{2,1}$ and $\oell_{2,3}$ are disjoint from the $\lsca^{-3}$ neighborhood of $(\sfd_{-1}/\lsca, \sfd_0/\lsca)$.
By similar arguments, we can show that $\oell_{2,2}$ is disjoint from the $\lsca^{-3}$ neighborhood of $P_{2,-}\cup P_{2,+}$.

Since $\oell_{2,1}$, $\oell_{2,2}$, $\oell_{2,3}$ are disjoint, by planarity we must have that $\oell_{2,2}$ is disjoint from the $\lsca^{-3}$ neighborhood of $(\sfd_{-1}/\lsca, \sfd_0/\lsca)$, 
$\oell_{2,1}$ is disjoint from the $\lsca^{-3}$ neighborhood of $P_{2,+}$, and  
$\oell_{2,3}$ is disjoint from the $\lsca^{-3}$ neighborhood of $P_{2,-}$.

The results for $\oell_{1,1}$, $\oell_{1,2}$, $\oell_{1,3}$ can be proved analogously. Putting all these together concludes the proof. 
\end{proof}

With all the above preparations (on the properties of $\Dfin_1$, $\Dfin_2$ around $\zzc$ and their steepest descent curves) in the above two subsections, we are now ready to prove the approximation of the NBRW kernel by the Pearcey kernel, i.e., Proposition \ref{prop:Kconvconj}.

\subsection{Convergence to Pearcey}  \label{ssec:convKP}
Using \eqref{eq:kberD} and Lemma \ref{lem:deforma}, Proposition \ref{prop:Kconvconj} can be deduced from the following lemma, plus an estimate of the binomial term $\binom{\oot_1-\oot_2}{\oox_1-\oox_2}$ in Lemma \ref{lem:gaussp} below.

\begin{lem}  \label{lem:Kconvconj}
Under the setting of Proposition \ref{prop:Kconvconj},
consider the integral
\begin{equation}  \label{eq:pearcov}
\BE^{\oox_1-\oox_2}(1-\BE)^{\oot_1-\oot_2+\oox_2-\oox_1}\frac{\oot_1!}{(\oot_2-1)!}\frac{\lsca^{\oot_2-\oot_1-1}}{(2\pi\i)^{2}} 
        \iint
       \frac{d\w d\z}{\w-\z}\exp(\lsca\Dfin_2(\z)-\lsca\Dfin_1(\w)).
\end{equation}
We divide the contours into two parts: inside or outside $\{\z\in\C: |\z-\zzc|\le \lsca^{-1/4+\ere}\cust\}$.
\begin{enumerate}
    \item[\textbf{Inner part:}] When the $\w$ contour is taken to be $[\zzc+e^{\pi\i/4} \lsca^{-1/4+\ere}\cust \to \zzc \to \zzc+e^{3\pi\i/4} \lsca^{-1/4+\ere}\cust]$ and $[\zzc+e^{5\pi\i/4} \lsca^{-1/4+\ere}\cust \to \zzc \to \zzc+e^{7\pi\i/4} \lsca^{-1/4+\ere}\cust]$, and the $\z$ contour is taken to be $[\zzc+e^{3\pi\i/2} \lsca^{-1/4+\ere}\cust \to \zzc+e^{\pi\i/2} \lsca^{-1/4+\ere}\cust]$, the integral \eqref{eq:pearcov} is equal to
\begin{multline}  \label{eq:pczwio}
\frac{1}{\sqrt{2}\lsca^{1/4} A^{1/4}\BE(1-\BE)} \frac{1}{(2\pi\i)^2}
\iint \frac{d\pcw d\pcz}{\pcw-\pcz}\exp\Big(\frac{-\pcz^4 + \pcw^4}{4}
+ \frac{\pcx_1\pcw-\pcx_2\pcz}{\sqrt{2}A^{1/4}\BE(1-\BE)}
+ \frac{\pct_1\pcw^2-\pct_2\pcz^2}{4A^{1/2}\BE(1-\BE)} \Big)\\ +\OO(\lsca^{-1/4-\errd}),
\end{multline}
where the $\pcw$ and $\pcz$ contours are respectively
\begin{equation}\label{eq:pcwzcon}
    [\infty e^{\pi\i/4} \to 0 \to \infty e^{3\pi\i/4}] \cup [\infty e^{5\pi\i/4} \to 0 \to \infty e^{7\pi\i/4}], \quad
    [\infty e^{3\pi\i/2}  \to  0\to \infty e^{\pi\i/2}].
\end{equation}
    \item[\textbf{Outer part:}]
    When either (i) the $\w$ contour is taken to be $[\zzc+e^{3\pi\i/4} \lsca^{-1/4+\ere}\cust \to -\cust-1 \to \zzc+e^{5\pi\i/4} \lsca^{-1/4+\ere}\cust]$ and $[\zzc+e^{7\pi\i/4} \lsca^{-1/4+\ere}\cust \to 1 \to \zzc+e^{\pi\i/4} \lsca^{-1/4+\ere}\cust]$, and the $\z$ contour is taken to be $[\infty e^{3\pi\i/2}\to \zzc \to \infty e^{\pi\i/2}]$, or (ii) the $\w$ contour is taken to be $[\zzc+e^{\pi\i/4} \lsca^{-1/4+\ere}\cust \to \zzc \to \zzc+e^{3\pi\i/4} \lsca^{-1/4+\ere}\cust]$ and $[\zzc+e^{5\pi\i/4} \lsca^{-1/4+\ere}\cust \to \zzc \to \zzc+e^{7\pi\i/4} \lsca^{-1/4+\ere}\cust]$, and the $\z$ contour is taken to be $[\infty e^{3\pi\i/2}\to \zzc+e^{3\pi\i/2} \lsca^{-1/4+\ere}\cust]$ and $[\zzc+e^{\pi\i/2} \lsca^{-1/4+\ere}\cust\to \infty e^{\pi\i/2}]$, the integral \eqref{eq:pearcov} is $\lesssim\exp(-c\lsca^{4\ere})$.
\end{enumerate}
\end{lem}
We next prove the inner part of Lemma \ref{lem:Kconvconj}, and the outer part will be proved in \Cref{ssec:small}.

\subsubsection{Inner contour integral}
We first analyze the factor in front of the integral in \eqref{eq:pearcov} at $z=z_c$ .
\begin{lem} \label{lem:conspart}
We have
\begin{multline} \label{eq:conspart}
\log\left(\BE^{\oox_1-\oox_2}(1-\BE)^{\oot_1-\oot_2+\oox_2-\oox_1}\frac{\oot_1!}{(\oot_2-1)!}\right) + \lsca \Dfin_2(\zzc)-\lsca \Dfin_1(\zzc) \\ = (\oot_1-\oot_2+1)\log \lsca  - \log(\cust\BE(1-\BE)) + \OO(\lsca^{-1/2}\cust^{-1}).
\end{multline}
\end{lem}
\begin{proof}
The left-hand side of \eqref{eq:conspart} is equal to
\begin{multline}  \label{eq:conspartle}
(\oox_1-\oox_2)\log \BE  + (\oot_1-\oot_2+\oox_2-\oox_1)\log(1-\BE) + \log(\oot_1!/(\oot_2-1)!) \\ + \sum_{i=-\oox_2+1}^{-\oox_2+\oot_2-1}\log \Big|\zzc+\frac{i}{\lsca}\Big|  - \sum_{i=-\oox_1}^{-\oox_1+\oot_1}\log \Big|\zzc+\frac{i}{\lsca}\Big| .    
\end{multline}
Using $\oox_1, \oox_2, \oot_1-\lsca\cust, \oot_2-\lsca\cust\lesssim \lsca^{1/2}$ and $\zzc, \cust+\zzc\asymp \cust$, we obtain that
\begin{align*}
&\sum_{i=-\oox_2+1}^{-\oox_2+\oot_2-1}\log \Big|\zzc+\frac{i}{\lsca}\Big|  - \sum_{i=-\oox_1}^{-\oox_1+\oot_1}\log \Big|\zzc+\frac{i}{\lsca}\Big|  \\
=&
\lsca \int_{-\frac{\oox_2}{\lsca}}^{\frac{-\oox_2+\oot_2}{\lsca}} \log |\zzc+x|  \rd x - 
\lsca \int_{-\frac{\oox_1}{\lsca}}^{\frac{-\oox_1+\oot_1}{\lsca}} \log |\zzc+x|  \rd x - \log(-\zzc) - \log(\cust+\zzc) + \OO(\lsca^{-1/2}\cust^{-1}).
\end{align*}
The first two terms in the last line are equal to
\begin{align*}
& (-\lsca\zzc+\oox_2)\log\Big(-\zzc+\frac{\oox_2}{\lsca}\Big) - (-\lsca\zzc+\oox_2)
+ (\lsca\zzc-\oox_2+\oot_2)\log\Big(\zzc+\frac{-\oox_2+\oot_2}{\lsca}\Big) - (\lsca\zzc-\oox_2+\oot_2)
\\
- & (-\lsca\zzc+\oox_1)\log\Big(-\zzc+\frac{\oox_1}{\lsca}\Big) + (-\lsca\zzc+\oox_1)
-(\lsca\zzc-\oox_1+\oot_1)\log\Big(\zzc+\frac{-\oox_1+\oot_1}{\lsca}\Big) + (\lsca\zzc-\oox_1+\oot_1),
\end{align*}
which further simplifies to (recall that $\BE=-\zzc\cust^{-1}$)
\begin{align*}
& (-\lsca\zzc+\oox_2)\log\Big(1+\frac{\oox_2}{-\lsca\zzc}\Big) 
-
(-\lsca\zzc+\oox_1)\log\Big(1+\frac{\oox_1}{-\lsca\zzc}\Big) \\
+ &
(\lsca\zzc-\oox_2+\oot_2)\log\Big(1+\frac{-\oox_2+\oot_2-\lsca\cust}{\lsca(\zzc+\cust)}\Big) 
-(\lsca\zzc-\oox_1+\oot_1)\log\Big(1+\frac{-\oox_1+\oot_1-\lsca\cust}{\lsca(\zzc+\cust)}\Big) \\
+ & (\oox_2-\oox_1)\log(\BE) + (-\oox_2+\oox_1+\oot_2-\oot_1) \log(1-\BE) + (\oot_2-\oot_1)\log(\cust) 
+\oot_1-\oot_2.
\end{align*}
Using Stirling's approximation and that $\oot_1-\lsca\cust, \oot_2-\lsca\cust\lesssim \lsca^{1/2}$, we also get that 
\begin{multline*}
\log(\oot_1!/(\oot_2-1)!) = \oot_1\log(\oot_1) - \oot_2\log(\oot_2) - \oot_1 + \oot_2 + \log(\lsca\cust) + \OO(\lsca^{-1/2}\cust^{-1}) \\
= (\oot_1-\oot_2+1)\log(\lsca\cust)+
\oot_1\log\Big(1+\frac{\oot_1-\lsca\cust}{\lsca\cust}\Big) - \oot_2\log\Big(1+\frac{\oot_2-\lsca\cust}{\lsca\cust}\Big) - \oot_1 + \oot_2  + \OO(\lsca^{-1/2}\cust^{-1}).
\end{multline*}
By summing up the above expressions, we have that \eqref{eq:conspartle} is equal to 
\begin{equation}  \label{eq:zxtot}
\begin{split}
&(-\lsca\zzc+\oox_2)\log\Big(1+\frac{\oox_2}{-\lsca\zzc}\Big) 
-
(-\lsca\zzc+\oox_1)\log\Big(1+\frac{\oox_1}{-\lsca\zzc}\Big) \\
+ &
(\lsca\zzc-\oox_2+\oot_2)\log\Big(1+\frac{-\oox_2+\oot_2-\lsca\cust}{\lsca(\zzc+\cust)}\Big) 
-(\lsca\zzc-\oox_1+\oot_1)\log\Big(1+\frac{-\oox_1+\oot_1-\lsca\cust}{\lsca(\zzc+\cust)}\Big) \\
+ & \oot_1\log\Big(1+\frac{\oot_1-\lsca\cust}{\lsca\cust}\Big) - \oot_2\log\Big(1+\frac{\oot_2-\lsca\cust}{\lsca\cust}\Big)
\\
+&(\oot_1-\oot_2+1)\log \lsca 
 - \log(\cust\BE(1-\BE)) + \OO(\lsca^{-1/2}\cust^{-1}).
 \end{split}
\end{equation}
The last line matches the right-hand side of \eqref{eq:conspart}. 
It remains to show that the first three lines are of order $\OO(\lsca^{-1/2}\cust^{-1})$.
We first consider 
\begin{equation}  \label{eq:zlzt}
(-\lsca\zzc+\oox_2)\log\Big(1+\frac{\oox_2}{-\lsca\zzc}\Big)  
+ (\lsca\zzc-\oox_2+\oot_2)\log\Big(1+\frac{-\oox_2+\oot_2-\lsca\cust}{\lsca(\zzc+\cust)}\Big) - \oot_2\log\Big(1+\frac{\oot_2-\lsca\cust}{\lsca\cust}\Big) .
\end{equation}
For this expression, its derivative with respect to $\oox_2$ is
\[\log\Big(1+\frac{\oox_2}{-\lsca\zzc}\Big) - \log\Big(1+\frac{-\oox_2+\oot_2-\lsca\cust}{\lsca(\zzc+\cust)}\Big)
= \log\Big( 1+\frac{\cust(\oox'_2-\oox_2)}{\zzc(\lsca\zzc-\oox_2+\oot_2)} \Big)\lesssim \lsca^{-3/4}\cust^{-1},\]
where $\oox_2'=\frac{-\zzc}{\cust}(\oot_2-\lsca\cust)$ and we used that $\oox_2'-\oox_2\lesssim \lsca^{1/4}$. In addition, we have that that value of \eqref{eq:zlzt} vanishes when we replace $\oox_2$ by $\oox_2'$, i.e.,
\[
(-\lsca\zzc+\oox_2')\log\Big(1+\frac{\oox_2'}{-\lsca\zzc}\Big)  
+ (\lsca\zzc-\oox_2'+\oot_2)\log\Big(1+\frac{-\oox_2'+\oot_2-\lsca\cust}{\lsca(\zzc+\cust)}\Big) - \oot_2\log\Big(1+\frac{\oot_2-\lsca\cust}{\lsca\cust}\Big) = 0.
\]
By integrating over $\oox_2$, we obtain that \eqref{eq:zlzt} is of order $\OO( \lsca^{-1/2}\cust^{-1})$.
Similarly, we have
\[
(-\lsca\zzc+\oox_1)\log\Big(1+\frac{\oox_1}{-\lsca\zzc}\Big)  
+ (\lsca\zzc-\oox_1+\oot_1)\log\Big(1+\frac{-\oox_1+\oot_1-\lsca\cust}{\lsca(\zzc+\cust)}\Big) - \oot_1\log\Big(1+\frac{\oot_1-\lsca\cust}{\lsca\cust}\Big)\lesssim \lsca^{-1/2}\cust^{-1}.
\]
Plugging these two estimates into \eqref{eq:zxtot}, the conclusion follows.
\end{proof}

We now finish the estimate on the contour integral inside $\{\z\in\C: |\z-\zzc|\le \lsca^{-1/4+\ere}\cust\}$.
\begin{proof}[Proof of Lemma \ref{lem:Kconvconj}: Inner part]
By Lemma \ref{lem:conspart}, using $\lsca^{-1/2}\cust^{-1}\lesssim \lsca^{-\errd}$ (due to that $\lsca^{-1/2+\errh}<\sct\asymp\cust$), it suffices to prove the same estimate for
\begin{equation}  \label{eq:sfwz}
\frac{1}{\cust\BE(1-\BE)} \frac{1}{(2\pi\i)^2}
\iint \frac{d\w d\z}{\w-\z}\exp\left(\lsca\Dfin_2(\z)-\lsca\Dfin_1(\w)-\lsca \Dfin_2(\zzc) + \lsca \Dfin_1(\zzc)\right),
\end{equation}
where the contours for $\w$ and $\z$ are as stated in Lemma \ref{lem:Kconvconj} (Inner part).

We want to approximate $\Dfin_1$ and $\Dfin_2$ by $\oSfin$. For this purpose, we estimate $\Dfin'_2 - \oSfin'$. Compared to Lemma \ref{lem:dDS}, here we consider $\z$ much closer to $\zzc$, thereby getting more refined estimates.
Take any $\z\in\C$ with $|\z-\zzc|\lesssim \cust^{3/2}$, we have that
\begin{align*}
\Dfin'_2(\z) - \oSfin'(\z) &=\sum_{i=-\oox_2+1}^{-\oox_2+\oot_2-1} \frac{1}{\lsca\z + i} - \sum_{i=0}^{\oot_c-1} \frac{1}{\lsca\z + i}
\\ &=\log(\lsca\z/(\lsca\z-\oox_2)) + \log((\lsca\z-\oox_2+\oot_2)/(\lsca\z+\oot_c)) + \OO(\lsca^{-1}\cust^{-1})\\
&=
\log\left(1+\frac{\lsca\z(\oot_2-\oot_c)+\oox_2\oot_c}{(\lsca\z-\oox_2)(\lsca\z+\oot_c)}\right)+
\OO(\lsca^{-1}\cust^{-1}).
\end{align*}
Using $\frac{\zzc}{\cust}(\oot_2-\lsca\cust)+\oox_2=\lsca^{1/4}\pcx_2+\OO(1)\lesssim \lsca^{1/4}$, $|\z-\zzc|\lesssim \cust^{3/2}$, and $\zzc, \zzc+\cust\asymp \cust$, we obtain that  
$$\frac{\lsca\z(\oot_2-\oot_c)+\oox_2\oot_c}{(\lsca\z-\oox_2)(\lsca\z+\oot_c)}=\OO(\lsca^{-3/4}\cust^{-1}).$$ 
With this estimate, we can derive that
\begin{equation}   \label{eq:oSDdiff}
\Dfin'_2(\z) - \oSfin'(\z) =\frac{\lsca\z(\oot_2-\oot_c)+\oox_2\oot_c}{(\lsca\z-\oox_2)(\lsca\z+\oot_c)}+
\OO(\lsca^{-1}\cust^{-1})
=
\frac{\lsca^{1/2}(\z-\zzc)\pct_2+\lsca^{1/4}\cust\pcx_2}{(\lsca\z-\oox_2)(\z+\cust)}+
\OO(\lsca^{-1}\cust^{-1}),
\end{equation}
where we also used $\oot_2-\oot_c=\lsca^{1/2}\pct_2+\OO(1)$ for the second equality. 
Then, using $|\z-\zzc|\lesssim \cust^{3/2}$, $\oox_2\lesssim \lsca^{1/2}$, and $\zzc, \zzc+\cust\asymp \cust$, we get that  
\[
\frac{\lsca^{1/2}(\z-\zzc)\pct_2+\lsca^{1/4}\cust\pcx_2}{(\lsca\z-\oox_2)(\z+\cust)} = \frac{\lsca^{-1/2}(\z-\zzc)\pct_2+\lsca^{-3/4}\cust\pcx_2}{\zzc(\zzc+\cust)}\left[1+\OO\left(|\z-\zzc|\cust^{-1}+\lsca^{-1/2}\cust^{-1}\right)\right].
\]
If we further assume that $|\z-\zzc|\le \lsca^{-1/4+\ere}\cust$, this expression reduces to
\[
\frac{\lsca^{1/2}(\z-\zzc)\pct_2+\lsca^{1/4}\cust\pcx_2}{(\lsca\z-\oox_2)(\z+\cust)}
=
-\frac{\lsca^{-1/2}(\z-\zzc)\pct_2}{\cust^2\BE(1-\BE)} - \frac{\lsca^{-3/4}\pcx_2}{\cust\BE(1-\BE)} + 
\OO\left(\lsca^{-3/4+\ere}\cust^{-1}(\lsca^{-1/4+\ere}+\lsca^{-1/2}\cust^{-1})\right).
\]
Plugging this estimate into \eqref{eq:oSDdiff} and taking an integration over $\z$, we obtain that
\[
\Dfin_2(\z) - \Dfin_2(\zzc) - \oSfin(\z) + \oSfin(\zzc)
+ \frac{\lsca^{-1/2}(\z-\zzc)^2\pct_2}{2\cust^2\BE(1-\BE)} + \frac{\lsca^{-3/4}(\z-\zzc)\pcx_2}{\cust\BE(1-\BE)}
 \lesssim \lsca^{-5/4+3\ere}+\lsca^{-3/2+2\ere}\cust^{-1},
\]
when $|\z-\zzc|\le \lsca^{-1/4+\ere}\cust$.
By Lemma \ref{lem:ospoc}, when $|\z-\zzc|\le \lsca^{-1/4+\ere}\cust$, we have that
\[\oSfin(\z) - \oSfin(\zzc)+\cust^{-4}A(\z-\zzc)^4\lesssim 
\lsca^{-5/4+5\ere}\cust^{-1/2}.\]
Then, using $\lsca^{-1/2+\errh}<\sct\asymp\cust$ and the fact that $\ere$ is small enough depending on $\errh$, we conclude that
\[
\Dfin_2(\z) - \Dfin(\zzc) + \cust^{-4}A(\z-\zzc)^4
+ \frac{\lsca^{-1/2}(\z-\zzc)^2\pct_2}{2\cust^2\BE(1-\BE)} + \frac{\lsca^{-3/4}(\z-\zzc)\pcx_2}{\cust\BE(1-\BE)}
 \lesssim \lsca^{-1-\errd},
\]
where $\errd>0$ is a small enough constant depending on $\errh$ and $\ere$. Similarly, we have
\[
\Dfin_1(\w) - \Dfin(\zzc) + \cust^{-4}A(\w-\zzc)^4
+ \frac{\lsca^{-1/2}(\w-\zzc)^2\pct_1}{2\cust^2\BE(1-\BE)} + \frac{\lsca^{-3/4}(\w-\zzc)\pcx_1}{\cust\BE(1-\BE)}
\lesssim \lsca^{-1-\errd},
\]
when $|\w-\zzc|\le \lsca^{-1/4+\ere}\cust$. Plugging the above two estimates into \eqref{eq:sfwz}, we obtain
\begin{multline*}
\frac{1}{\cust\BE(1-\BE)} \frac{1}{(2\pi\i)^2}
\iint \frac{d\w d\z}{\w-\z}\exp\Big(-\lsca\cust^{-4}A(\z-\zzc)^4 + \lsca\cust^{-4}A(\w-\zzc)^4\\
-\lsca^{1/2}\frac{ \pct_2(\z-\zzc)^2-\pct_1(\w-\zzc)^2}{2\cust^2\BE(1-\BE)}
-\lsca^{1/4}\frac{ \pcx_2(\z-\zzc)-\pcx_1(\w-\zzc)}{\cust\BE(1-\BE)}
+\OO(\lsca^{-\errd})
\Big).
\end{multline*}
Introducing the rescaled variables $\pcw=\sqrt{2}(\w-\zzc)\lsca^{1/4}\cust^{-1}A^{1/4}$ and $\pcz=\sqrt{2}(\z-\zzc)\lsca^{1/4}\cust^{-1}A^{1/4}$, we get \eqref{eq:pczwio}, with the $\pcw$ contour being 
\begin{equation*}
    [\sqrt{2}e^{\pi\i/4} \lsca^{\ere}A^{1/4} \to 0 \to \sqrt{2}e^{3\pi\i/4} \lsca^{\ere}A^{1/4}]\cup [\sqrt{2}e^{5\pi\i/4} \lsca^{\ere}A^{1/4} \to 0 \to \sqrt{2}e^{7\pi\i/4} \lsca^{\ere}A^{1/4}],
\end{equation*}
and the $\pcz$ contour being
\begin{equation*}
[\sqrt{2}e^{3\pi\i/2} \lsca^{\ere}A^{1/4} \to \sqrt{2}e^{\pi\i/2} \lsca^{\ere}A^{1/4}].
\end{equation*}
We note that by replacing the $\pcw$ and $\pcz$ contours with \eqref{eq:pcwzcon}, the integral in \eqref{eq:pczwio} changes by $\OO(\exp(-\lsca^{4\ere}A/2))$, because the integrand along these contours is at most $\lesssim \exp((-\pcz^4+\pcw^4)/8)$. 
This concludes the proof.
\end{proof}

\subsubsection{Binomial and Gaussian}
By classical CLT, it is expected that the binomial term in the NBRW kernel would lead to the Gaussian term in the Pearcey kernel.
We provide a detailed derivation here.
\begin{lem}  \label{lem:gaussp}
If $\pct_1-\pct_2>\lsca^{-\errd}$, then we have that
\begin{multline*}
\BE^{\oox_1-\oox_2}(1-\BE)^{\oot_1-\oot_2+\oox_2-\oox_1}\binom{\oot_1-\oot_2}{\oox_1-\oox_2}\\=
\frac{\lsca^{-1/4}}{\sqrt{2\pi (\pct_1-\pct_2) \BE(1-\BE)}} \exp\left[ - \frac{(\pcx_1-\pcx_2)^2}{2\BE(1-\BE)(\pct_1-\pct_2)}\right] + \OO\left(\lsca^{-1/2+3\errd}\right).
\end{multline*}
\end{lem}
\begin{proof}
Denote $\oot'=\oot_1-\oot_2$ and $\oox'=\oox_1-\oox_2$. If $\pct_1-\pct_2>\lsca^{-\errd}$, then we have $\oot'-\oox'=\oot_1-\oot_2 - \oox_1+\oox_2> (1-B)\lsca^{1/2-\errd}-C\lsca^{1/4}$ and $\oox'=\oox_1-\oox_2>B\lsca^{1/2-\errd}-C\lsca^{1/4}$. Then, by Stirling's approximation, we have
\[
\binom{\oot_1-\oot_2}{\oox_1-\oox_2} = \binom{\oot'}{\oox'}= \frac{\oot'!}{\oox'!(\oot'-\oox')!}= \left[1+\OO(\lsca^{-1/2+\errd})\right] \sqrt{\frac{\oot'}{2\pi \oox'(\oot'-\oox')}} (\oot')^{\oot'} (\oox')^{-\oox'} (\oot'-\oox')^{-\oot'+\oox'}.
\]
Using $\oot'=(\pct_1-\pct_2)\lsca^{1/2}+\OO(1)$ and $\oox'=\BE(\pct_1-\pct_2)\lsca^{1/2} + (\pcx_1-\pcx_2)\lsca^{1/4}+ \OO(1)$, we get that
\begin{equation}  \label{eq:binori}
\sqrt{\frac{\oot'}{2\pi \oox' (\oot'-\oox')}} 
=
\frac{\lsca^{-1/4}(1+\OO(\lsca^{-1/4+\errd}))}{\sqrt{2\pi (\pct_1-\pct_2)\BE(1-\BE)}} ,
\end{equation}
and 
\begin{align}  
&\log\Big(\BE^{\oox'}(1-\BE)^{\oot'-\oox'} (\oot')^{\oot'} (\oox')^{-\oox'} (\oot'-\oox')^{-\oot'+\oox'} \Big)
=\oox'\log \frac{\BE \oot'}{\oox'} +(\oot'-\oox')\log\frac{(1-\BE)\oot'}{\oot'-\oox'} \nonumber\\
=&-\left[\BE(\pct_1-\pct_2)\lsca^{1/2}+(\pcx_1-\pcx_2)\lsca^{1/4}\right] \log\left(1+\BE^{-1}\frac{\pcx_1-\pcx_2}{\pct_1-\pct_2}\lsca^{-1/4}\right) \nonumber\\
&  -\left[(1-\BE)(\pct_1-\pct_2)\lsca^{1/2}-(\pcx_1-\pcx_2)\lsca^{1/4}\right]  \log\left(1-(1-\BE)^{-1}\frac{\pcx_1-\pcx_2}{\pct_1-\pct_2}\lsca^{-1/4}\right) + \OO(\lsca^{-1/4+\errd}),\label{eq:xpr}
\end{align}
where for the second estimate, we used that the derivatives of the left-hand side with respect to $\oox'$ and $\oot'$ are
\[
\log \frac{\BE}{1-\BE}  - \log \frac{\oox'}{\oot'-\oox'} \lesssim \lsca^{-1/4+\errd},\quad \log(1-\BE) - \log \frac{\oot'-\oox'}{\oot'} \lesssim \lsca^{-1/4+\errd}.
\]
Taking the Taylor expansion of the logarithms in \eqref{eq:xpr}, we get
\[
\eqref{eq:xpr}=-\frac{(\pcx_1-\pcx_2)^2}{2\BE(1-\BE)(\pct_1-\pct_2)}
+\OO(\lsca^{-1/4+3\errd}).
\]
Multiplying the exponential of this estimate with the right-hand side of \eqref{eq:binori} concludes the proof.
\end{proof}

\subsection{Smallness of outer contour integral}  \label{ssec:small}

It remains to prove the outer part of Lemma \ref{lem:Kconvconj}.
By Lemma \ref{lem:conspart}, it suffices to show that the following integral 
\begin{equation}  \label{eq:contwz}
\iint\frac{d\w d\z}{\w-\z}\exp\left(\lsca\Dfin_2(\z)-\lsca\Dfin_1(\w) -\lsca\Dfin_2(\zzc) + \lsca\Dfin_1(\zzc) \right)    
\end{equation}
over the contours stated in the outer part of Lemma \ref{lem:Kconvconj} is bounded by $\OO(\exp(-c\lsca^{4\ere}))$.

We recall from Section \ref{sssec:stpcur} the curves $\oell_{2,1}$, $\oell_{2,2}$, $\oell_{2,3}$ on which $\Im(\Dfin_2)=0$, and the curves $\oell_{1,1}$, $\oell_{1,2}$, $\oell_{1,3}$ on which $\Im(\Dfin_1)=0$.
We will deform the $\w$ contours 
$[\zzc+e^{3\pi\i/4} \lsca^{-1/4+\ere}\cust \to -\cust-1 \to \zzc+e^{5\pi\i/4} \lsca^{-1/4+\ere}\cust]$ and $[\zzc+e^{7\pi\i/4} \lsca^{-1/4+\ere}\cust \to 1 \to \zzc+e^{\pi\i/4} \lsca^{-1/4+\ere}\cust]$
to curves that closely follow $\oell_{1,1}$, $\oell_{1,3}$ and their complex conjugates (denoted by $\ooell_{1,1}$ and $\ooell_{1,3}$). We will deform the $\z$ contours $[\infty e^{3\pi\i/2}\to \zzc+e^{3\pi\i/2} \lsca^{-1/4+\ere}\cust]$ and $[\zzc+e^{\pi\i/2} \lsca^{-1/4+\ere}\cust\to \infty e^{\pi\i/2}]$ to curves that closely follow $\oell_{2,2}$ and its complex conjugate (denoted by $\ooell_{2,2}$).
Along these curves, $\lsca (\Dfin_2(\z)-\Dfin_2(\zzc))$ and $-\lsca(\Dfin_1(\w)-\Dfin_1(\zzc))$ are (almost) real and negative, and of order at least $\gtrsim \lsca^{4\ere}$ by Lemma \ref{lem:curveDo}.

As indicated earlier, some issues may appear as we lose precise control of the behaviors of these curves.
First, it is unclear whether $\oell_{1,1}$ and $\oell_{1,3}$ are disjoint from $\oell_{2,2}$, so we need to consider the residues resulting from their possible intersections.
Second, it could be technical to control the length of these curves. We will circumvent this issue by discretizing these curves.

\subsubsection{Intersections}
For possible intersections between the curves $\oell_{1,1}$, $\oell_{1,3}$, and $\oell_{2,2}$, we consider the zero set of $\Im(\Dfin_1-\Dfin_2)$ (which must contain all these intersections). We denote 
\[
E(\Dfin_1-\Dfin_2)=\lsca^{-1}(\qq{\oox_1-\oot_1, \oox_1}\triangle \qq{\oox_2-\oot_2+1,\oox_2-1}),
\]
which is the set of poles of $\Dfin'_1-\Dfin'_2$ ($\triangle$ denotes the symmetric difference).
Similar to Lemma \ref{lem:dD}, we have that for any $\z\not\in E(\Dfin_1-\Dfin_2)$,
\begin{equation}  \label{eq:ddD}
\Dfin'_1(\z)-\Dfin'_2(\z) \lesssim \frac{1}{\min_{x\in E(\Dfin_1-\Dfin_2)}\lsca|\z-x|} + \log \lsca .    
\end{equation}
The zero set of $\Im(\Dfin_1-\Dfin_2)$ can be described as follows.
\begin{lem}   \label{lem:difcur}
The set $\{\z\in \HH\cup\R\setminus E(\Dfin_1-\Dfin_2) :\Im(\Dfin_1(\z)-\Dfin_2(\z))=0\}$ contains the following parts:
\begin{enumerate}
\item The connected component of $\R\setminus E(\Dfin_1-\Dfin_2)$ containing $\zzc$;
\item 
If either (i) $\oox_1-\oot_1<\oox_2-\oot_2+1$ and $\oox_1>\oox_2-1$, or (ii) $\oox_1-\oot_1>\oox_2-\oot_2+1$ and $\oox_1<\oox_2-1$, the set $\{\z\in \HH\cup\R\setminus E(\Dfin_1-\Dfin_2) :\Im(\Dfin_1(\z)-\Dfin_2(\z))=0\}$ also contains a smooth and self-avoiding curve $\oell_d$ from the connected component of $\R\setminus E(\Dfin_1-\Dfin_2)$ containing $\zzc$ to $\infty$. Except for the starting point, $\oell_d$ is contained in $\HH$.
Moreover, $\Re(\Dfin_1-\Dfin_2)$ is strictly monotone along $\oell_d$, and $\oell_d$ is disjoint from the $\lsca^{-3}$ neighborhood of $E(\Dfin_1-\Dfin_2)$.
\end{enumerate}
\end{lem}
For convenience, if neither (i) nor (ii) in (2) holds, we denote $\oell_d=\emptyset$. 
This lemma can be proved by analyzing the critical points of $\Dfin_1-\Dfin_2$ and using arguments similar to the proofs of Lemmas \ref{lem:curveDo} and \ref{lem:oodis}. We omit the details here.

\subsubsection{Curve discretization}

Choose $\ssca=\lsca^{-100}$, we denote
\[\Lat=2\ssca\Z+2\i\ssca\Z, \quad \Lat'=\Lat+(1+\i)\ssca.\]
For any discrete interval $I$ (i.e., $I$ is a set of consecutive integers) and $\{\z_i\}_{i\in I}$, we call $\{\z_i\}_{i\in I}$ a \emph{$\Lat$-path} (resp. \emph{$\Lat'$-path}) if the followings hold: all these $\z_i$ are different points in $ \Lat$ (resp. $\Lat'$), and every pair $\z_i$ and $\z_{i+1}$ are nearest neighbors on the lattice $\Lat$ (resp. $\Lat'$). The curve obtained by connecting every pair of points $\z_i$ and $\z_{i+1}$ using a line segment is called the corresponding \emph{$\Lat$-curve} (resp. \emph{$\Lat'$-curve}).
We first discretize (i.e., approximate) $\oell_{2,2}$ by a $\Lat'$-curve.
\begin{lem}   \label{lem:disa}
There exists a $\Lat'$-curve $\soell_{2,2}$ from a point in the intersection of $\Lat'$ and the $2\ssca$-neighborhood of $\zzc+e^{\i \theta_2} \lsca^{-1/4+\ere}\cust$ to $\infty$, such that 
(1) for any $\z\in \Lat'\cap\soell_{2,2}$, we have $(\z+[-\ssca, \ssca] + \i [-\ssca, \ssca])\cap \oell_{2,2}\neq \emptyset$, and (2) $\soell_{2,2}$ is disjoint from the $\lsca^{-3}/2$-neighborhood of $\R$.
\end{lem}
\begin{proof}
Let $\z_*$ be a vertex in $\Lat'\cap (\zzc+e^{\i \theta_2} \lsca^{-1/4+\ere}\cust+[-\ssca, \ssca] + \i [-\ssca, \ssca])$.
Denote
\[
\Lat'_{2,2}:=\{\z\in U\cap\Lat': (\z+[-\ssca, \ssca] + \i [-\ssca, \ssca])\cap \oell_{2,2}\neq \emptyset \}.
\]
If we cannot find a $\Lat'$-curve satisfying (1), then the set of points that are connected to $\z_*$ through a $\Lat'$-curve contained in $\Lat'_{2,2}$ must be finite. Then, we can find a sequence of numbers $\z_1, \ldots, \z_k \in \Lat'\setminus \Lat'_{2,2}$, such that $\z_k=\z_1$, $|\z_{i+1}-\z_i|\le 2\sqrt{2}\ssca$ for each $i$, and $\zzc+e^{\i \theta_2} \lsca^{-1/4+\ere}\cust$ is enclosed by $[\z_1\to\cdots\to\z_k]$. In this case, $\oell_{2,2}$ must intersect $[\z_1\to\cdots\to\z_k]$, contradicting the fact that $\z_1, \ldots, \z_k \not\in \Lat'_{2,2}$.
The statement (2) follows immediately from (1) and Lemma \ref{lem:oodis}.
\end{proof}
For $\oell_d$ from Lemma \ref{lem:difcur}, we have a similar discretization. Its proof is similar to that of Lemma \ref{lem:disa}, so we omit the details.
\begin{lem}   \label{lem:disb}
For any two points $\z_1, \z_2\in\oell_d$, there exists a $\Lat'$-curve connecting two $\Lat'$ points lying in the $2\ssca$-neighborhoods of $\z_1$ and $\z_2$, respectively, such that 
this $\Lat'$-curve is in the $2\ssca$-neighborhood of the part of $\oell_d$ between $\z_1$ and $\z_2$,
and is disjoint from the $\lsca^{-3}/2$-neighborhood of $E(\Dfin_1-\Dfin_2)$.
\end{lem}

We next state the discretizations for $\oell_{1,1}$ and $\oell_{1,3}$, which use the lattice $\Lat$ instead of $\Lat'$, so that their possible intersections with $\soell_{2,2}$ are points in the lattice $\Lat+\{\ssca, \i\ssca\}$.
We also require these points to be close to $\oell_d$, in order to handle the residues resulting from these intersections. Therefore, the discretizations for $\oell_{1,1}$ and $\oell_{1,3}$ are slightly more delicate than that for $\oell_{2,2}$.
\begin{lem}   \label{lem:disc}
There exists a $\Lat$-curve $\soell_{1,1}$ (resp.~$\soell_{1,3}$) from a point in the intersection of $\Lat$ and the $5\ssca$-neighborhood of $\zzc+e^{\i \vartheta_1} \lsca^{-1/4+\ere}\cust$ (resp.~$\zzc+e^{\i \vartheta_3} \lsca^{-1/4+\ere}\cust$) to $P_{1,-}$ (resp.~$P_{1,+}$), such that (1) it is contained in the $5\ssca$ neighborhood of $\oell_{1,1}$ (resp.~$\oell_{1,3}$), (2) it is disjoint from the $\lsca^{-3}/2$-neighborhood of $\R\setminus P_{1,-}$ (resp.~$\R\setminus P_{1,+}$), and (3) the set $\soell_{1,1}\cap\soell_{2,2}$ (resp. $\soell_{1,3}\cap\soell_{2,2}$) is contained in the $3\ssca$-neighborhood of $\oell_d$.
\end{lem}
\begin{proof}
We note that the set $(\HH\cup\R)\setminus(\oell_{1,1}\cup[\zzc\to\zzc+e^{\i \vartheta_1} \lsca^{-1/4+\ere}\cust])$ contains two connected components, and we denote the bounded one by $U_{1,-}$, and the unbounded one plus the boundary $\oell_{1,1}\cup[\zzc\to\zzc+e^{\i \vartheta_1} \lsca^{-1/4+\ere}\cust]$ by  $U_{1,+}$.
Let $\Lat'_-$ denote the set of $\z\in \HH\cap\Lat'$, such that either
$
(\z+[-\ssca, \ssca] + \i [-\ssca, \ssca]) \subset U_{1,-},
$ or
\begin{equation}  \label{eq:zlttemp}
(\z+[-\ssca, \ssca] + \i [-\ssca, \ssca])\cap (U_{1,-} \cap \oell_{2,2}) \neq \emptyset, \quad
(\z+[-\ssca, \ssca] + \i [-\ssca, \ssca]) \cap (U_{1,+} \cap \oell_{2,2}) =\emptyset.    
\end{equation}
Then, $\Lat'_-$ is a finite set.
In particular, we have
\begin{equation}  \label{eq:zlatmrel}
(\z+[-\ssca, \ssca] + \i [-\ssca, \ssca])\cap U_{1,-} \neq \emptyset, \  \forall \z \in \Lat'_-; \quad (\z+[-\ssca, \ssca] + \i [-\ssca, \ssca])\cap U_{1,+} \neq \emptyset, \  \forall \z \in (\HH\cap\Lat')\setminus \Lat'_-.
\end{equation}

Denote by $W_-$ the union of the set $\{\z\in \HH\cup\R: |\z-\zzc|\le \lsca^{-1/4+\ere}\cust\}\cap U_{1,-}$ and the $\lsca^{-3}$ neighborhood of $[\w_-, \zzc-\lsca^{-1/4+\ere}\cust]\setminus P_{1,-}$ in $U$. 
Denote by $W_+$ the $\lsca^{-3}$ neighborhood of $(-\infty, \w_-]\setminus P_{1,-}$ in $U$. See Figure \ref{fig:discr} for an illustration of these sets.
By Lemma \ref{lem:oodis}, we have $W_-\subset U_{1,-}$ and $U_{1,-}\cap W_+=\emptyset$.
Thus, we have
\[
\{\z\in \Lat':\z+[-\ssca, \ssca] + \i [-\ssca, \ssca] \subset W_-\} \subset \Lat'_-, \quad 
\{\z\in \Lat': \z+[-\ssca, \ssca] + \i [-\ssca, \ssca] \subset W_+\} \cap \Lat'_- = \emptyset.
\]
Now, consider the set $\cup_{\z\in \Lat'_-}(\z+[-\ssca, \ssca] + \i [-\ssca, \ssca])$, and denote its connected component that intersects $W_-$ as $W_*$.
The boundary of $W_*$ is a $\Lat$-curve, which contains a vertex in $\Lat$ (denoted by $\z_{*,c}$) within distance $5\ssca$ to $\zzc+e^{\i \vartheta_1} \lsca^{-1/4+\ere}\cust$.
It also contains $[\w_-, \zzc-5\ssca]\setminus P_{1,-}$ and is disjoint from $(-\infty, \w_-] \setminus P_{1,-}$. 
We denote its left-most intersection with $P_{1,-}$ as $\z_{*,l}$.
We let $\soell_{1,1}$ be the part of the boundary of $W_*$, going from $\z_{*,c}$ to $\z_{*,l}$ counter-clockwisely. See Figure \ref{fig:discr} for an illustration of $\soell_{1,1}$ and these related objects.
\begin{figure}[hbt!]
    \centering
\begin{tikzpicture}[line cap=round,line join=round,>=triangle 45,x=0.2cm,y=0.2cm]
\clip(-11,-1.7) rectangle (55,17);
\fill[green!30] (40,0) --  (35,5) arc(135:174.9:7.071068) -- (9.6,0.7) arc(90:180:0.7) --  cycle;
\fill[green!30] (-10,0) --  (-10,0.7) -- (6.4,0.7) arc(90:0:0.7) --  cycle;
\draw [thick] [->] (-10,0) -- (54,0);

\draw [cyan] [thick] plot [smooth] coordinates {(36,25) (23,14) (25.6,10) (30.3,10) (29.2,14) (36,12) (40.1,9) (40,7.071068)};

\draw [blue] [thick] plot coordinates {(40,0) (35,5)};

\draw [blue] [thick] plot [smooth] coordinates {(8,0) (7.7,1.5) (5,5) (6,8) (14,7) (19,12) (29,12) (32,7.5) (35,5)};

\draw [red] [thick] plot coordinates {(6.4,0) (9.6,0)};

\begin{scope}
    \clip (30,0) rectangle (50,10);
    \draw [dotted] (40,0) circle (7.071068);
\end{scope}

\begin{scriptsize}
\draw ($(39.5-0.3*104-0.5,0-0.4)$) node[anchor=south west]{$\z_{*,l}$};
\draw ($(39.5-0.3*15+1.2,0.3*15+0.4)$) node[anchor=north east]{$\z_{*,c}$};

\draw (10,10) node[anchor=west]{$U_{1,+}$};
\draw (10,5) node[anchor=west]{$U_{1,-}$};

\draw (8,0) node[anchor=north]{$\w_-$};

\draw (40,0) node[anchor=north]{$\zzc$};
\draw (34.7,5.5) node[anchor=west]{$\zzc+e^{\i \vartheta_1} \lsca^{-1/4+\ere}\cust$};
\draw (39.7,7.571068) node[anchor=west]{$\zzc+e^{\i \theta_2} \lsca^{-1/4+\ere}\cust$};

\draw [cyan] (36,12) node[anchor=south]{$\oell_{2,2}$};
\draw [blue] (20,12) node[anchor=south]{$\oell_{1,1}$};
\draw [red] (8,0) node[anchor=south east]{$P_{1,-}$};

\draw [orange] (20,3) node[anchor=south]{$W_*$};
\draw [orange] (17,6) node[anchor=south]{$\soell_{1,1}$};

\draw [darkgreen] (35,0) node[anchor=south]{$W_-$};
\draw [darkgreen] (-3,-0.7) node[anchor=south]{$W_+$};
\end{scriptsize}

\draw [fill=uuuuuu] (40,0) circle (.6pt);
\draw [fill=uuuuuu] (35,5) circle (.6pt);
\draw [fill=uuuuuu] (8,0) circle (.6pt);

\draw [fill=uuuuuu] (40,7.071068) circle (.6pt);

  \foreach \x in {0,...,15}
    \draw [orange] [very thin] ($(39.5-0.3*\x,0.3*\x)$) -- ($(39.5-0.3*\x,0.3*\x+0.3)$) -- ($(39.5-0.3*\x-0.3,0.3*\x+0.3)$);
    
\draw [orange] [very thin] ($(39.5-0.3*15,0.3*15+0.3)$) -- ($(39.5-0.3*16-0.3,0.3*15+0.3)$);    
  \foreach \x in {16,...,19}
    \draw [orange] [very thin] ($(39.2-0.3*\x,0.3*\x)$) -- ($(39.2-0.3*\x,0.3*\x+0.3)$) -- ($(39.2-0.3*\x-0.3,0.3*\x+0.3)$);
    
\draw [orange] [very thin] ($(39.5-0.3*22,0.3*20)$) -- ($(39.5-0.3*21,0.3*20)$); 

  \foreach \x in {20,...,23}
    \draw [orange] [very thin] ($(38.9-0.3*\x,0.3*\x)$) -- ($(38.9-0.3*\x,0.3*\x+0.3)$) -- ($(38.9-0.3*\x-0.3,0.3*\x+0.3)$);
    
\draw [orange] [very thin] ($(39.5-0.3*26,0.3*24)$) -- ($(39.5-0.3*26,0.3*25)$); 
    
  \foreach \x in {25,26}
    \draw [orange] [very thin] ($(39.2-0.3*\x,0.3*\x)$) -- ($(39.2-0.3*\x,0.3*\x+0.3)$) -- ($(39.2-0.3*\x-0.3,0.3*\x+0.3)$);
    
\draw [orange] [ultra thin] ($(39.5-0.3*28,0.3*27)$) -- ($(39.5-0.3*28,0.3*30)$) -- ($(39.5-0.3*29,0.3*30)$) -- ($(39.5-0.3*29,0.3*32)$) -- ($(39.5-0.3*30,0.3*32)$) -- ($(39.5-0.3*30,0.3*36)$) -- ($(39.5-0.3*32,0.3*36)$) -- ($(39.5-0.3*32,0.3*37)$) -- ($(39.5-0.3*33,0.3*37)$) -- ($(39.5-0.3*33,0.3*38)$) -- ($(39.5-0.3*34,0.3*38)$) -- ($(39.5-0.3*34,0.3*39)$) -- ($(39.5-0.3*36,0.3*39)$) -- ($(39.5-0.3*36,0.3*40)$) -- ($(39.5-0.3*43,0.3*40)$) -- ($(39.5-0.3*43,0.3*41)$) -- ($(39.5-0.3*59,0.3*41)$) -- ($(39.5-0.3*59,0.3*40)$) -- ($(39.5-0.3*67,0.3*40)$) -- ($(39.5-0.3*67,0.3*39)$) -- ($(39.5-0.3*69,0.3*39)$) -- ($(39.5-0.3*69,0.3*38)$) -- ($(39.5-0.3*71,0.3*38)$) -- ($(39.5-0.3*71,0.3*37)$) -- ($(39.5-0.3*72,0.3*37)$) -- ($(39.5-0.3*72,0.3*36)$) -- ($(39.5-0.3*73,0.3*36)$) -- ($(39.5-0.3*73,0.3*35)$) -- ($(39.5-0.3*74,0.3*35)$) -- ($(39.5-0.3*74,0.3*33)$) -- ($(39.5-0.3*75,0.3*33)$) -- ($(39.5-0.3*75,0.3*32)$) -- ($(39.5-0.3*76,0.3*32)$) -- ($(39.5-0.3*76,0.3*31)$) -- ($(39.5-0.3*77,0.3*31)$) -- ($(39.5-0.3*77,0.3*29)$) -- ($(39.5-0.3*78,0.3*29)$) -- ($(39.5-0.3*78,0.3*28)$) -- ($(39.5-0.3*79,0.3*28)$) -- ($(39.5-0.3*79,0.3*26)$) -- ($(39.5-0.3*80,0.3*26)$) -- ($(39.5-0.3*80,0.3*25)$) -- ($(39.5-0.3*81,0.3*25)$) -- ($(39.5-0.3*81,0.3*24)$) -- ($(39.5-0.3*82,0.3*24)$) -- ($(39.5-0.3*82,0.3*23)$) -- ($(39.5-0.3*84,0.3*23)$) -- ($(39.5-0.3*84,0.3*22)$) -- ($(39.5-0.3*93,0.3*22)$) -- ($(39.5-0.3*93,0.3*23)$) -- ($(39.5-0.3*98,0.3*23)$) -- ($(39.5-0.3*98,0.3*24)$) -- ($(39.5-0.3*103,0.3*24)$) -- ($(39.5-0.3*103,0.3*25)$) -- ($(39.5-0.3*108,0.3*25)$) -- ($(39.5-0.3*108,0.3*26)$) -- ($(39.5-0.3*112,0.3*26)$) -- ($(39.5-0.3*112,0.3*25)$) -- ($(39.5-0.3*113,0.3*25)$) -- ($(39.5-0.3*113,0.3*24)$) -- ($(39.5-0.3*114,0.3*24)$) -- ($(39.5-0.3*114,0.3*16)$) -- ($(39.5-0.3*113,0.3*16)$) -- ($(39.5-0.3*113,0.3*14)$) -- ($(39.5-0.3*112,0.3*14)$) -- ($(39.5-0.3*112,0.3*13)$) -- ($(39.5-0.3*111,0.3*13)$) -- ($(39.5-0.3*111,0.3*12)$) -- ($(39.5-0.3*110,0.3*12)$) -- ($(39.5-0.3*110,0.3*11)$) -- ($(39.5-0.3*109,0.3*11)$) -- ($(39.5-0.3*109,0.3*10)$)  -- ($(39.5-0.3*108,0.3*10)$) -- ($(39.5-0.3*108,0.3*9)$) -- ($(39.5-0.3*107,0.3*9)$) -- ($(39.5-0.3*107,0.3*7)$) -- ($(39.5-0.3*106,0.3*7)$) -- ($(39.5-0.3*106,0.3*6)$) -- ($(39.5-0.3*105,0.3*6)$) -- ($(39.5-0.3*105,0.3*4)$) -- ($(39.5-0.3*104,0.3*4)$) -- ($(39.5-0.3*104,0)$) -- ($(39.5,0)$);     

\draw [fill=uuuuuu] ($(39.5-0.3*104,0)$) circle (.6pt);
\draw [fill=uuuuuu] ($(39.5-0.3*15,0.3*15)$) circle (.6pt);

\end{tikzpicture}
\caption{An illustration of discretizing $\oell_{1,1}$ into $\soell_{1,1}$.}
\label{fig:discr}
\end{figure}
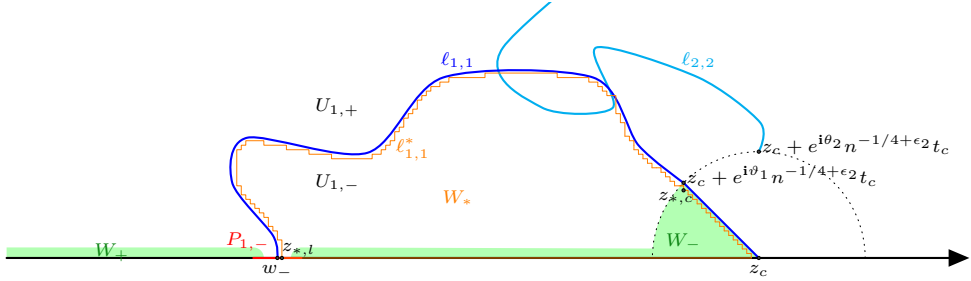

Next, we check that $\soell_{1,1}$ constructed this way satisfies the requirements (1), (2), and (3). 
\begin{enumerate}[leftmargin=*]
    \item[(1)] By \eqref{eq:zlatmrel}, $\Lat'_-$ is contained in the $\sqrt{2}\ssca$ neighborhood of $U_{1,-}$, and $(\HH\cap\Lat')\setminus\Lat'_-$ is contained in the $\sqrt{2}\ssca$ neighborhood of $U_{1,+}$. Hence, any point in $\soell_{1,1}$ is within distance $2\sqrt{2}\ssca$ to both $U_{1,-}$ and $U_{1,+}$, thereby within distance $2\sqrt{2}\ssca$ to $\oell_{1,1}\cup[\zzc\to\zzc+e^{\i \vartheta_1} \lsca^{-1/4+\ere}\cust]$. As $\soell_{1,1}$ is from $\z_{*,c}$ to $\z_{*,l}$, we have that it must be contained in the $5\ssca$ neighborhood of $\oell_{1,1}$.
    \item[(2)] This follows from (1) and Lemma \ref{lem:oodis}.
    \item[(3)] Take any $\z_0\in \soell_{1,1}\cap\soell_{2,2}$, we can find $\z_{0,-}, \z_{0,+}\in \Lat'\cap\soell_{2,2}$ such that $|\z_{0,+}-\z_{0,-}|=2\ssca$, $\z_0=(\z_{0,+}+\z_{0,-})/2$, and $\z_{0,-}$, $\z_{0,+}$ are in different sides of $\soell_{1,1}$. Then, only one of them is in $W_*$, and hence only one of them is in $\Lat'_-$.
    Without loss of generality, we assume that 
    $\z_{0,-}\in \Lat'_-$ and $\z_{0,+}\in (\HH\cap\Lat')\setminus\Lat'_-$.
    
    As $\z_{0,+}\in\soell_{2,2}$, we can find $\z_{b,+} \in \oell_{2,2}\cap (\z_{0,+}+[-\ssca, \ssca] + \i [-\ssca, \ssca])$. We can always choose $\z_{b,+}$ so that $\z_{b,+} \in U_{1,+}$, since otherwise $\z_{0,+} \in \Lat'_-$ by \eqref{eq:zlttemp}.
    On the other hand, we can also find $\z_{b,-} \in \oell_{2,2}\cap (\z_{0,-}+[-\ssca, \ssca] + \i [-\ssca, \ssca])$.
    Since $\z_{0,-}\in \Lat'_-$, if $(\z_{0,-}+[-\ssca, \ssca] + \i [-\ssca, \ssca])\subset U_{1,-}$, we can always choose $\z_{b,-}$ such that $\z_{b,-} \in U_{1,-}$.
Now, we have found $\z_{b,-} \in U_{1,-} \cap \oell_{2,2}$ and $\z_{b,+} \in U_{1,+} \cap \oell_{2,2}$, with $|\z_{b,-}-\z_0|, |\z_{b,+}-\z_0| < 3\ssca$.
Then, we have $\z_{b,-}, \z_{b,+}\in U$ (since $\oell_{2,2}\subset U$), and 
\begin{equation}  \label{eq:dfore1}
\Im(\Dfin_2(\z_{b,-}))=\Im(\Dfin_2(\z_{b,+}))=0.
\end{equation}
Also, notice that 
\begin{equation}\label{eq:dfore2}
\Im(\Dfin_1(\z_{b,-}))>0,    
\end{equation}
since $\Im \Dfin_1 >0$ on $U_{1,-}\setminus \R$ and $\oell_{2,2}\cap\R=\emptyset$.

We next claim that
\begin{equation}  \label{eq:dfore3}
\Im(\Dfin_1(\z_{b,+}))\le 0.
\end{equation}
Otherwise, $\z_{b,-}$ and $\z_{b,+}$ must be in different components of $U\setminus \oell_{1,2}$, so $[\z_{b,-}\to\z_{b,+}]$ (which is in $U$ and has length $\lesssim \ssca$) intersects both $\oell_{1,1}$ and $\oell_{1,2}$.
Take any $\z_{b,1}\in [\z_{b,-}\to\z_{b,+}]\cap\oell_{1,1}$ and $\z_{b,2}\in [\z_{b,-}\to\z_{b,+}]\cap\oell_{1,2}$.
By property (2) above, we have that $\z_{b,1}$ and $\z_{b,2}$ are not in the $\lsca^{-3}/3$ neighborhood of $\R\setminus P_{1,-}$, so Lemma \ref{lem:dD} implies that $\Dfin_1(\z_{b,1})-\Dfin_1(\z_{b,2})\lesssim \lsca^2\ssca$.
On the other hand, since $\z_{b,1}\in \oell_{1,1}$ and $\z_{b,2}\in\oell_{1,2}$, we have $\Dfin_1(\z_{b,1})-\Dfin_1(\zzc), \Dfin_1(\z_{b,2})-\Dfin_1(\zzc)\in \R$, and $\Dfin_1(\z_{b,1})-\Dfin_1(\zzc)\ge \Dfin_1(e^{\i \vartheta_1} \lsca^{-1/4+\ere}\cust)-\Dfin_1(\zzc) \asymp \lsca^{-1+4\ere}$, $-\Dfin_1(\z_{b,2})+\Dfin_1(\zzc)\ge -\Dfin_1(e^{\i \vartheta_2} \lsca^{-1/4+\ere}\cust)+\Dfin_1(\zzc) \asymp \lsca^{-1+4\ere}$ by (the $\Dfin_1$ version of) \eqref{eq:inicur}.
Hence, we have $\Dfin_1(\z_{b,1})-\Dfin_1(\z_{b,2})\gtrsim \lsca^{-1+4\ere}$, which leads to a contradiction. Thus, we conclude \eqref{eq:dfore3}.

Now, combining \eqref{eq:dfore1}, \eqref{eq:dfore2}, and \eqref{eq:dfore3}, we obtain that $\Im(\Dfin_1(\z_{b,-})-\Dfin_2(\z_{b,-}))>0\ge \Im(\Dfin_1(\z_{b,+})-\Dfin_2(\z_{b,+}))$.
Thus, there exists some $\z_{b,*}\in [\z_{b,-}\to\z_{b,+}]$ such that $\Im(\Dfin_1(\z_{b,*})-\Dfin_2(\z_{b,*}))=0$, i.e., $\z_{b,*}\in \oell_d$. Note that we must have $|\z_{b,*}-\z_0|<3\ssca$ since  $|\z_{b,-}-\z_0|, |\z_{b,+}-\z_0| < 3\ssca$, so $\z_0$ is contained in the $3\ssca$-neighborhood of $\oell_d$.
\end{enumerate}

Finally, the construction of $\soell_{1,3}$ and corresponding properties follow similar arguments.
\end{proof}

We are now ready to finish the proof of Lemma \ref{lem:Kconvconj} by deforming the contours to $\soell_{2,2}$ and $\soell_{1,1}$, $\soell_{1,3}$.

\begin{proof}[Proof of Lemma \ref{lem:Kconvconj}: Outer part]
As stated above, we just need to control \eqref{eq:contwz} over the contours stated in Lemma \ref{lem:Kconvconj} (Outer part).
For the convenience of notation, we introduce the following definitions.
Let $L_1$ be the contour of $[\zzc+e^{3\pi\i/4} \lsca^{-1/4+\ere}\cust \to \iota(\soell_{1,1})]$ followed by $\soell_{1,1}$, where $\iota(\soell_{1,1})$ is the starting point of $\soell_{1,1}$, with $|\iota(\soell_{1,1})-(\zzc+e^{\i \vartheta_1} \lsca^{-1/4+\ere}\cust)|<5\ssca$.
Let $L_2$ be the contour of $[\zzc+e^{\pi\i/2} \lsca^{-1/4+\ere}\cust \to \iota(\soell_{2,2})]$ followed by $\soell_{2,2}$, where $\iota(\soell_{2,2})$ is the starting point of $\soell_{2,2}$, with $|\iota(\soell_{2,2})-(\zzc+e^{\i \theta_2} \lsca^{-1/4+\ere}\cust)|<2\ssca$.

We first claim that there exists a constant $c>0$ such that
\begin{equation}  \label{eq:wzzcdb}
 \Re(\Dfin_1(\w)-\Dfin_1(\zzc))>c\lsca^{-1+4\ere}, \quad \Re(\Dfin_2(\z)-\Dfin_2(\zzc))<-c\lsca^{-1+4\ere},
\end{equation}
for any $\w\in L_1$ and $\z\in L_2$.

We now prove this claim. For any $\w\in\soell_{1,1}$, by Lemma \ref{lem:disc} we can find $\w'\in\oell_{1,1}$ such that $|\w-\w'|<5\ssca$; and $\w$ is disjoint from the $\lsca^{-3}/2$-neighborhood of $\R\setminus P_{1,-}$. By Lemma \ref{lem:curveDo}, we have $\Re(\Dfin_1(\w')-\Dfin_1(\zzc))>c\lsca^{-1+4\ere}$. Then, by Lemma \ref{lem:dD}, the first inequality in \eqref{eq:wzzcdb} holds for any $\w\in\soell_{1,1}$.
    For any $\w\in[\zzc+e^{ \i \theta_1} \lsca^{-1/4+\ere}\cust \to \iota(\soell_{1,1})]$, we must have that $|\w-(\zzc+e^{\i \vartheta_1} \lsca^{-1/4+\ere}\cust)|\lesssim \lsca^{-1/4-\ere}\cust$ by Lemma \ref{lem:curveDo} and Lemma \ref{lem:disc}. Then, using Lemma \ref{lem:dDzce}, we obtain the first inequality in \eqref{eq:wzzcdb}.
The second inequality in \eqref{eq:wzzcdb} can be proved in a similar way by using Lemmas \ref{lem:dD}, \ref{lem:dDzce}, \ref{lem:curveDo}, and \ref{lem:disa}.

We next analyze the contours in the outer part of Lemma \ref{lem:Kconvconj}, which can be further deformed and decomposed into several parts. By symmetry, it suffices to consider the following few cases.
\begin{itemize}[leftmargin=*]
    \item[(a)] The contour of $\z$ is $L_2':=[\zzc+e^{3\pi\i/2} \lsca^{-1/4+\ere}\cust\to \zzc+e^{\pi\i/2} \lsca^{-1/4+\ere}\cust]$, and the contour of $\w$ is $L_1$.
    We note that $L_1$ is deformed from $[\zzc+e^{3\pi\i/4} \lsca^{-1/4+\ere}\cust \to -\cust-1]$, and this deformation is allowed because the integrand in \eqref{eq:contwz} has no $\w$ pole in $\HH$ or $[-\cust-1,\w_-]\cup P_{1,-}$.
    Note that the contours $L_2'$ and $L_1$ do not intersect, and the distance between them is of order $\gtrsim \lsca^{-1/4+\ere}\cust$ by Lemma \ref{lem:disc} and the fact that the distance between $\oell_{1,1}$ and $L_2'$ is of order $\gtrsim \lsca^{-1/4+\ere}\cust$. 
    
    For any $\z\in L_2'$, we have  $ \Re(\Dfin_2(\z)-\Dfin_2(\zzc))\le C\lsca^{-1+2\ere}$ for a constant $C>0$ by Lemma \ref{lem:dDzce} (note it can be negative).    
Then, by \eqref{eq:wzzcdb}, for any $\z\in L_2'$ and $ \w\in L_1$, we have $\Re(\Dfin_2(\z)-\Dfin_1(\w) -\Dfin_2(\zzc) + \Dfin_1(\zzc))<-c\lsca^{-1+4\ere}$.
    By Lemma \ref{lem:Dlz}, $\oell_{1,1}$ is contained in a ball of radius $\lesssim \lsca$. Thus, $L_1$ is also contained in a ball of radius $\lesssim \lsca$, from which we see that its length satisfies $\lesssim \lsca^2\ssca^{-1}$.
    Now, the integral \eqref{eq:contwz} along these contours is at most of order 
    \[\lsca^2\ssca^{-1} \cdot \lsca^{1/4-\ere}\cust^{-1} \cdot \exp(-c\lsca^{4\ere}) \lesssim \exp(-c\lsca^{4\ere}/2).\]
    \item[(b)] The contour of $\w$ is $L_1':=[\zzc+e^{5\pi\i/4} \lsca^{-1/4+\ere}\cust\to \zzc\to \zzc+e^{3\pi\i/4} \lsca^{-1/4+\ere}\cust]$, and the contour of $\z$ is $L_2$.
    We note that $L_2$ is deformed from $[\zzc+e^{\pi\i/2} \lsca^{-1/4+\ere}\cust \to \infty e^{\pi\i/2}]$. This deformation is allowed because the integrand of \eqref{eq:contwz} has no $\z$ pole in $\HH$, and $|\exp(\lsca\Dfin_2(\z))|$
    decays exponentially along any direction between $\i$ and the asymptotic direction of  $\soell_{2,2}$, due to Lemma \ref{lem:Dlz}. Not that the contours $L_1'$ and $L_2$ do not intersect, and the distance between them is of order $\gtrsim \lsca^{-1/4+\ere}\cust$ by Lemma \ref{lem:disa} and the fact that the distance between $\oell_{2,2}$ and $L_1'$ is of order $\gtrsim\lsca^{-1/4+\ere}\cust$. 
    
    For any $\w\in L_1'$, we have $ \Re(-\Dfin_1(\w)+\Dfin_1(\zzc))\le C\lsca^{-1+2\ere}$ by Lemma \ref{lem:dDzce}.
Then, by \eqref{eq:wzzcdb}, for any $\z\in L_2$ and $\w\in L_1'$, we have $\Re(\Dfin_2(\z)-\Dfin_1(\w) -\Dfin_2(\zzc) + \Dfin_1(\zzc))<-c\lsca^{-1+4\ere}$.
    In addition, if $|\z|>\lsca$, by Lemma \ref{lem:Dlz} there is 
    \begin{equation}  \label{eq:bdredlaz}
        \Re(\Dfin_2(\z)-\Dfin_2(\zzc))<-c|\z|.
    \end{equation}
Therefore, the integral \eqref{eq:contwz} along these contours is at most of order
\[
 \lsca^2\ssca^{-1} \cdot \lsca^{1/4-\ere}\cust^{-1} \cdot \exp(-c\lsca^{4\ere}) + \sum_{W\in\Z, W>\lsca} W\ssca^{-1} \cdot \lsca^{1/4-\ere}\cust^{-1} \cdot \exp(-c\lsca W)  \lesssim \exp(-c\lsca^{4\ere}/2)  ,
\]
where the first term accounts for the part of $L_2$ where $|\z|\le \lsca$ (and hence has length $\lesssim \lsca^2\ssca^{-1}$), and the summand for each $W$ accounts for the part of $L_2$ where $||\z|-W|\le 1$ (and hence has length $\lesssim W\ssca^{-1}$).
\item[(c)] The contour of $\w$ is $L_1$, and the contour of $\z$ is the complex conjugate of $L_2$. The distance between them is $\gtrsim\lsca^{-3}$ by Lemmas \ref{lem:disa} and \ref{lem:disc}. 
By bounding the lengths of the contours as in (a) and (b), and using \eqref{eq:wzzcdb} and \eqref{eq:bdredlaz}, we conclude that  the integral \eqref{eq:contwz} along these contours is also $\lesssim \exp(-c\lsca^{4\ere})$.
\item[(d)] The contour of $\w$ is $L_1$, and the contour of $\z$ is $L_2$.
These contours $L_1$ and $L_2$ may intersect, since $\soell_{1,1}$ and $\soell_{2,2}$ may intersect. (As we have seen in (a) and (b), the distance between $[\zzc+e^{3\pi\i/4} \lsca^{-1/4+\ere}\cust \to \iota(\soell_{1,1})]$ and $\soell_{2,2}$ is of order $\gtrsim\lsca^{-1/4+\ere}$, and the distance between $[\zzc+e^{3\pi\i/4} \lsca^{-1/4+\ere}\cust \to \iota(\soell_{2,2})]$ and $\soell_{1,1}$ is also of order $\gtrsim\lsca^{-1/4+\ere}$.)
Denote $S=\soell_{1,1}\cap \soell_{2,2}=L_1\cap L_2$.
Since $\soell_{1,1}$ is a $\Lat$-curve and $\soell_{2,2}$ is a $\Lat'$-curve, $S$ is contained in the lattice $\Lat+\{\ssca, \i\ssca\}=\Lat'+\{\ssca, \i\ssca\}$. Using the fact that $L_1$ is contained in a ball of radius $\lesssim \lsca$ as shown above, we get the trivial bound $|S|\lesssim (\lsca\ssca^{-1})^2$.
For any $\w\in L_1$ and $\z\in L_2$, unless $\w, \z\in \z_*+(-\ssca,\ssca)+\i(-\ssca,\ssca)$ for some $\z_*\in S$, we must have $|\w-\z|\ge \ssca$.
For each $\z_*\in S$, the parts of $L_1$ and $L_2$ inside $\z_*+(-\ssca,\ssca)+\i(-\ssca,\ssca)$ are two orthogonal line segments, each having length $2\ssca$.
Therefore, by \eqref{eq:wzzcdb}, the integral \eqref{eq:contwz} over these segments is $\lesssim \ssca^{-1} \exp(-c\lsca^{4\ere})$.
Summing over all $\z_*\in S$ and bounding the integral \eqref{eq:contwz} over the rest parts of the $\w$ and $\z$ contours as in (a) and (b), we again get the bound $\lesssim \exp(-c\lsca^{4\ere})$.
\end{itemize}
In case (d), the $\w$ and $\z$ contours are deformed from $[\zzc+e^{3\pi\i/4} \lsca^{-1/4+\ere}\cust \to -\cust-1]$ and $[\zzc+e^{\pi\i/2} \lsca^{-1/4+\ere}\cust \to \infty e^{\pi\i/2}]$, respectively.
When $S\neq\emptyset$, this procedure potentially leads to residues at $\w=\z$, which are given by
\begin{equation}  \label{eq:residue}
\int d\z \exp(\lsca \Dfin_2(\z)-\lsca \Dfin_1(\z)-\lsca\Dfin_2(\zzc) + \lsca\Dfin_1(\zzc)),
\end{equation}
where the integral is along some curves $[\z_{*,1}\to \z_{*,2}]$, $\ldots$, $[\z_{*,2k-1}\to \z_{*,2k}]$, such that $\{\z_{*,i}\}_{i=1}^{2k}=S$.

We assume that $S\neq\emptyset$ and bound the integral \eqref{eq:residue} along $[\z_{*,2i-1}\to \z_{*,2i}]$ for each $i=1, \ldots, k$.
By Lemmas \ref{lem:disb} and \ref{lem:disc}, the curve $\oell_d\neq\emptyset$ and we can find $\z_{*,2i-1}'', \z_{*,2i}''\in \oell_d$ and a $\Lat'$-curve $L_{*,i}$ with endpoints $\z_{*,2i-1}'$, $\z_{*,2i}'$, such that $|\z_{*,2i-1}-\z_{*,2i-1}''|, |\z_{*,2i}-\z_{*,2i}''|\le 5\ssca$, and $|\z_{*,2i-1}'-\z_{*,2i-1}'', |\z_{*,2i}'-\z_{*,2i}''|\le 2\ssca$.
In addition, the $\Lat'$-curve $L_{*,i}$ is contained in the $2\ssca$ neighborhood of $\oell_d$ between $\z_{*,2i-1}''$ and $\z_{*,2i}''$, and is disjoint from the $\lsca^{-3}/2$-neighborhood of $E(\Dfin_1-\Dfin_2)$.

We deform the contour $[\z_{*,2i-1}\to\z_{*,2i}]$ into the contour consisting of $[\z_{*,2i-1}\to\z_{*,2i-1}']$, $L_{*,i}$, and $[\z_{*,2i}'\to\z_{*,2i}]$.
For each $\z$ on this contour,
by \eqref{eq:ddD} and the fact that $\Re(\Dfin_1-\Dfin_2)$ is monotone along $\oell_d$, we have that
\begin{equation}  \label{eq:difredot}
\begin{split}
&\Re(\Dfin_2(\z)-\Dfin_1(\z)) < \Re(\Dfin_2(\z'')-\Dfin_1(\z''))+ \OO(\ssca \lsca^2) \\ < & \Re(\Dfin_2(\z_{*,2i-1}'')-\Dfin_1(\z_{*,2i-1}'')) \vee \Re(\Dfin_2(\z_{*,2i}'')-\Dfin_1(\z_{*,2i}'')) + \OO(\ssca \lsca^2) \\ < &\Re(\Dfin_2(\z_{*,2i-1})-\Dfin_1(\z_{*,2i-1})) \vee \Re(\Dfin_2(\z_{*,2i})-\Dfin_1(\z_{*,2i})) + \OO(\ssca \lsca^2),
\end{split}
\end{equation}
where $\z''$ is a point in $\oell_d$ between $\z_{*,2i-1}''$ and $\z_{*,2i}''$, with $|\z''-z|\le 5\ssca$.
Using Lemma \ref{lem:dD} and the facts that $\z_{*,2i-1}\in S\subset \soell_{1,1}$ and $\oell_{1,1}$ is disjoint from the $\lsca^{-3}$ neighborhood of $E(\Dfin_1)$ by Lemma \ref{lem:oodis}, we obtain that
\begin{equation}   \label{eq:dfinoti}
\Re(\Dfin_1(\z_{*,2i-1})) > \Re(\Dfin_1(\z_{*,2i-1}'''))  +\OO(\ssca \lsca^2),    
\end{equation}
for some $\z_{*,2i-1}'''\in \oell_{1,1}$ with $|\z_{*,2i-1}'''-\z_{*,2i-1}|\lesssim \ssca$. By Lemma \ref{lem:curveDo}, we have $\Re(\Dfin_1(\z_{*,2i-1}''')-\Dfin_1(\zzc))>c\lsca^{-1+4\ere}$, which, together with \eqref{eq:dfinoti}, implies that $\Re(\Dfin_1(\z_{*,2i-1})-\Dfin_1(\zzc))>c\lsca^{-1+4\ere}$.
Similarly, we have 
\[
 \Re(\Dfin_1(\z_{*,2i})-\Dfin_1(\zzc)) > c\lsca^{-1+4\ere} , \quad
\Re(\Dfin_2(\z_{*,2i-1})-\Dfin_2(\zzc)), \Re(\Dfin_2(\z_{*,2i})-\Dfin_2(\zzc)) < -c\lsca^{-1+4\ere} .\]
Combining the above estimates with \eqref{eq:difredot}, we conclude that 
\begin{equation}   \label{eq:dfinoti2}
\Re(\Dfin_2(\z)-\Dfin_1(\z)- \Dfin_2(\zzc) + \Dfin_1(\zzc))<-c\lsca^{-1+4\ere}.
\end{equation}

We next bound the length of $L_{*,i}$.
Since $\z_{*,2i-1}, \z_{*,2i}\in S\subset \soell_{1,1}$, we have that $\z_{*,2i-1}, \z_{*,2i}\lesssim \lsca$ by Lemma \ref{lem:disc}, because $\oell_{1,1}$ is contained in a ball of radius $\lesssim \lsca$ as discussed above.
We claim that $L_{*,i}$ is also contained in a ball of radius $\lesssim \lsca$.
First, we have $\z_{*,2i-1}'', \z_{*,2i}''\lesssim \lsca$.
Next, consider the semi-circle $\{\z\in \HH: |\z|=r\}$ for $r=|\z_{*,2i-1}''|+|\z_{*,2i}''|+\lsca$.
Since $\oell_d\neq\emptyset$, by Lemma \ref{lem:difcur} we have either (i) $\oox_1-\oot_1<\oox_2-\oot_2+1$ and $\oox_1>\oox_2-1$, or (ii) $\oox_1-\oot_1>\oox_2-\oot_2+1$ and $\oox_1<\oox_2-1$.
In either case, the function $\Im(\Dfin_1-\Dfin_2)$ is strictly monotone along the semi-circle.
Therefore, the semi-circle $\{\z\in \HH: |\z|=r\}$ intersects $\oell_d$ (which is contained in $\HH$ and goes to $\infty$) exactly once, so the intersection cannot lie on the part of $\oell_d$ between $\z_{*,2i-1}''$ and $\z_{*,2i}''$.
As a result, the part of $\oell_d$ between $\z_{*,2i-1}''$ and $\z_{*,2i}''$ is contained in a ball of radius $\lesssim \lsca$, and the claim holds.
Thus, the length of $L_{*,i}$ is at most $\lesssim \lsca^2\ssca^{-1}$.

Using the estimate \eqref{eq:dfinoti2}, the bound on the length of $L_{*,i}$, and $|\z_{*,2i-1}-\z'_{*,2i-1}|, |\z_{*,2i}-\z_{*,2i}'|\lesssim \ssca$, we conclude that the integral \eqref{eq:residue} along $[\z_{*,2i-1}\to\z_{*,2i-1}']$, $L_{*,i}$, and $[\z_{*,2i}'\to\z_{*,2i}]$ is at most $\lesssim \exp(-c\lsca^{4\ere})$.
Since $|S|\lesssim (\lsca\ssca^{-1})^2$ (by Lemma \ref{lem:Dlz}), summing over $i$ yields the desired bound $\lesssim \exp(-c\lsca^{4\ere})$.
\end{proof}

\section{Optimal rigidity around cusps}\label{s:optimalR}

In this section, we prove the optimal height function concentration estimate for random lozenge tilings around cusps, which will imply the first part of \Cref{lem:sfqes}.

 \subsection{Concentration of height function}

Fix a rational polygonal set $\fP$ satisfying \Cref{pa}, and denote its liquid region and arctic curve by $\mathfrak{L} = \mathfrak{L} (\mathfrak{P})$ and $\mathfrak{A} = \mathfrak{A} (\mathfrak{P})$, respectively. 
Let $H^*$ denote the limiting height function of $\fP$.
Let $n$ be a large integer such that $\mathsf{P}=n\fP$ is a tileable domain, and let $\mathsf{H}$ denote the height function associated with the uniformly random tiling of $\mathsf{P}$, that has boundary value $nH^*$ on $\partial \sfP$.
As in \Cref{sec:univproof}, all the constants in this section (including those hidden in $\lesssim,\gtrsim,\asymp,\OO$) can depend on $\fP$.
 
 We recall the following height function concentration statement from \cite{huang2021edge}. 
	\begin{thm}
	\label{t:rigidity}
  Take any constant $\delta > 0$.
  Let $\mathfrak{L}_+ (\mathfrak{P})=\{u\in \overline{\fP}: \dist(u, \mathfrak L)\le n^{\delta-2/3}\}$ be the augmented liquid region.
  Then, the following two statements hold with overwhelming probability. 
		\begin{enumerate}
			\item $\big| \mathsf{H} (nv) - n H^* (v) \big| < n^{\delta}$ for any $v \in \overline{\mathfrak{P}}$.
			\item For any $v \in \overline{\mathfrak{P}} \setminus \mathfrak{L}_+ (\mathfrak{P})$, we have $\mathsf{H} (nv) = n H^* (v)$.
		\end{enumerate}
	\end{thm}

Theorem \ref{t:rigidity} does not give optimal rigidity estimates close to cusp locations. In this section, we prove an optimal version, as stated in the following theorem.

\begin{thm}\label{t:cusprigidity}
  Fix a cusp point $(x_c, t_c) \in \mathfrak{A}$.  By possibly rotating $\fP$ by $180^\circ$, the arctic curve $\fA$ in a neighborhood of $(x_c,t_c)$ consists of two analytic pieces $\{(E_{-}(t),t): t_c-\fs\le t\le t_c\}$ and $\{(E_{+}(t),t): t_c-\fs\le t\le t_c\}$, for some small constant $\fs>0$.
Then, for any constant $\delta>0$, the following holds with overwhelming probability: for any $v=(x,t)$ such that $t_c-\fs\le t\le t_c$ and $x\in [E_{-}(t)+ (t_c-t)^{1/6}n^{\delta-2/3}, E_{+}(t)- (t_c-t)^{1/6}n^{\delta-2/3}]$, there is
\begin{align}\label{e:Optrigidity2}
\sfH(nv)= nH^*(v),
\end{align}
\end{thm}

 \begin{rem}
 When $t_c-t\asymp 1$, the statement \eqref{e:Optrigidity2} reduces to item (2) in \Cref{t:rigidity}. However, when $t_c-t\ll 1$, \eqref{e:Optrigidity2} is stronger than \Cref{t:rigidity}. Such optimal height function concentration is crucial for the tiling path estimate  (\Cref{lem:sfqes}), where we consider a mesoscopic box around the cusp location $(x_c, t_c)$.
 \end{rem}

This optimal rigidity estimate will be proved based on an estimate from \cite{huang2021edge} on random lozenge tilings in a trapezoid with random boundary, stated as \Cref{p:rigidityBB} below, and a comparison argument invoking \Cref{lem:hfmon}. For this, we start by carving out a trapezoid domain from $\fP$ around a cusp.

\subsection{Trapezoid domain}
\label{s:stripd} 

As $\fP$ satisfies \Cref{pa},
for definiteness, for the rest of this section, we assume (without loss of generality) that the axis $\ell$ in \Cref{pa} is the horizontal axis $\{ \sft = 0 \}$.
We take a cusp point $(x_c, t_c) \in \mathfrak{A}$. Then, $(x_c, t_c)$ is not a tangency location. 
By possibly rotating $\fP$ by $180^\circ$, we can assume that the cusp `points upwards', i.e., in a small neighborhood of $(x_c, t_c)$ the arctic curve $\fA$ is contained on or below the line $t=t_c$.
Note this notion is weaker than `upward oriented' from \Cref{sr}, because the axis $\ell$ is now fixed as the horizontal axis.
Next, we carve out a \emph{trapezoid} around $(x_c, t_c)$.

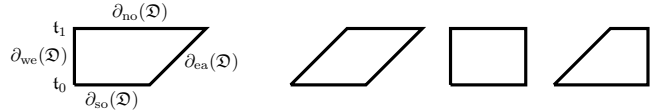
\begin{figure}[!ht]
		
		\begin{center}		
			
			\begin{tikzpicture}[
				>=stealth,
				auto,
				style={
					scale = .25
				}
				]
				
				\draw[black, very thick] (0, 0) node[left, scale = .7]{$\ft_0$} -- (4, 0) -- (7, 3) -- (0, 3) node[left, scale = .7]{$\ft_1$} -- (0, 0);
				\draw[black, very thick] (6.5+5, 0) -- (10.5+5, 0) -- (13.5+5, 3) -- (9.5+5, 3) -- (6.5+5, 0);
				\draw[black, very thick] (15+5, 0) -- (19+5, 0) -- (19+5, 3) -- (15+5, 3) -- (15+5, 0);
				\draw[black, very thick] (20.5+5, 0) -- (25.5+5, 0) -- (25.5+5, 3) -- (23.5+5, 3) -- (20.5+5, 0);
				
				\draw[] (3.5, 3) node[above, scale = .7]{$\partial_{\north} (\mathfrak{D})$};
				\draw[] (5.5, 1.5) node[below = 2, right, scale = .7]{$\partial_{\ea} (\mathfrak{D})$};
				\draw[] (0, 1.5) node[left, scale = .7]{$\partial_{\we} (\mathfrak{D})$};
				\draw[] (2, 0) node[below, scale = .7]{$\partial_{\so} (\mathfrak{D})$};	
				
			\end{tikzpicture}
			
		\end{center}
		
		\caption{\label{ddomain} Shown above are the four possibilities for $\mathfrak{D}$.}
		
	\end{figure} 
 
A \emph{trapezoid} is a subset of $\R^2$ of the following form
	\begin{flalign}
		\label{d}
		\mathfrak{D} = \big\{ (x, t) \in \mathbb{R} \times [\ft_0,\ft_1] : \mathfrak{a} (t) \le x \le \mathfrak{b} (t) \big\},
	\end{flalign}
where $\ft_0<\ft_1$, and $\mathfrak{a}, \mathfrak{b} $ are linear functions on $ [\ft_0, \ft_1]$ with $\mathfrak{a}' (t), \mathfrak{b}' (t) \in \{ 0, 1 \}$ and $\mathfrak{a} (t) \le \mathfrak{b} (t)$ for each $t \in [\ft_0, \ft_1]$. We denote its four boundaries by 
	\begin{flalign}
		\label{dboundary}
		\begin{aligned} 
		&\partial_{\so} (\mathfrak{D}) = \big\{ (x, t) \in {\mathfrak{D}}: t = \ft_0 \big\}; \qquad \quad
		\partial_{\north} (\mathfrak{D}) = \big\{ (x, t) \in {\mathfrak{D}}: t = \ft_1 \big\}; \\ 
		& \partial_{\we} (\mathfrak{D}) = \big\{ (x, t) \in {\mathfrak{D}}: x = \mathfrak{a} (t) \big\}; \qquad 
		\partial_{\ea} (\mathfrak{D}) = \big\{ (x, t) \in {\mathfrak{D}}: x = \mathfrak{b} (t) \big\}.
		\end{aligned} 
	\end{flalign}
	\noindent We refer to \Cref{ddomain} for a depiction.

   We now construct the trapezoid $\mathfrak{D}$ associated with $(x_c, t_c)$.
			Let $x_1 \in \mathbb{R}$ and $x_2 \in \mathbb{R}$ be the maximal and minimal numbers such that $x_1 < x_0 < x_2$, $u_1 = (x_1, t_c) \in \mathfrak{A}$, and $u_2 = (x_2, t_c) \in \mathfrak{A}$. By \Cref{pa}, neither $u_1$ nor $u_2$ is a cusp of $\mathfrak{A}$. If $u_1 \in \partial \mathfrak{P}$, then it is a (non-horizontal) tangency location of $\mathfrak{A}$, so it lies along a side of $\partial \mathfrak{P}$ with slope $1$ or $\infty$. We then let this side contain the west boundary of $\mathfrak{D}$. If instead $u_1 \notin \partial \mathfrak{P}$, then there exist $\varepsilon = \varepsilon (\mathfrak{P}, u) > 0$ and $r = r(\mathfrak{P}, u) \in (0,\varepsilon)$ such that the radius $r$ disk $\bB_{r} (x_1 - \varepsilon, t_c)$ does not intersect $\overline{\mathfrak{L}}$. Then, depending on whether $\nabla H^* (x_1, t_c) = (1, -1)$ or $\nabla H^* (x_1, t_c) \in \big\{ (0, 0), (1, 0) \big\}$ (one of them must hold by the first statement of \Cref{pla}), the west boundary of $\mathfrak{D}$ is contained in the segment obtained as the intersection between $\bB_{r} (x_1 - \varepsilon, t_c)$ and the line passing through $(x_1 - \varepsilon, t_c)$ with slope $1$ or $\infty$, respectively.

   So far, we've specified a segment containing the west boundary of $\mathfrak{D}$, and one containing its east boundary can be specified similarly. Then, we choose the interval $[\ft_0, \ft_1]$ as $\ft_0 = t_c - \fs$ and $\ft_1 = t_c + \fs$, where $\fs$ is chosen sufficiently small so that the east and west boundary of $\mathfrak{D}$ are contained in the segments specified above. These information determine the trapezoid $\fD$ associated with $(x_c, t_c)$.

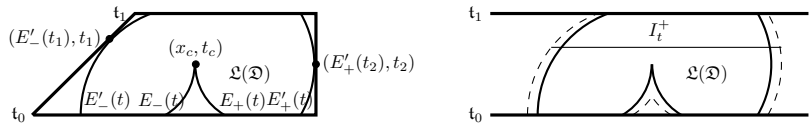
\begin{figure}[!hbt]
		
		\begin{center}		
			
			\begin{tikzpicture}[
				>=stealth,
				auto,
				style={
					scale = .45
				}
				]
				
				\draw[black, very thick] (-4.9, 0) node[left, scale = .7]{$\ft_0$}-- (3.45, 0)--(3.45,3)--(-1.9,3)node[left, scale = .7]{$\ft_1$}-- (-4.934877870428601, 0);
				
				\draw[black, very thick] (8.6, 0) node[left, scale = .7]{$\ft_0$}-- (18, 0);
				\draw[black, very thick] (8.6, 3) node[left, scale = .7]{$\ft_1$}-- (18, 3);
				
				\draw[black] (10.4, 2)-- (17.2, 2);
				\draw[] (13.6, 2) node[above, scale = .7]{$I_t^+$};
				
				\draw[black,fill=black] (-2.65,2.25) circle (.1);
				\draw[](-2.7,2.25) node[left, scale=.7]{$(E'_{-}(t_1),t_1)$};
				
				\draw[black,fill=black] (3.45,1.5) circle (.1);
				\draw[](3.45,1.5) node[right, scale=.7]{$(E'_{+}(t_2),t_2)$};
				
				\draw[black,fill=black] (-0.1,1.5) circle (.1);
				\draw[](-0.1,1.5) node[above, scale=.7]{$(x_c,t_c)$};
				
				\draw[](-3.6,0.4) node[right, scale=.7]{$E'_{-}(t)$};
				\draw[](-2,0.4) node[right, scale=.7]{$E_{-}(t)$};
				\draw[](0.4,0.4) node[right, scale=.7]{$E_{+}(t)$};
				\draw[](1.8,0.4) node[right, scale=.7]{$E'_{+}(t)$};

				\draw[black, thick] (-3.5, 0) arc (180:110:3.2);

				\draw[black, thick] (3, 0) arc (-30:30:3);

				\draw[black, thick] (-1, 0) arc (-60:0:1.732);
				\draw[black, thick] (.75, 0) arc (240:180:1.732);
				
				\draw[black, thick] (10, 0) arc (180:110:3.2);
				\draw[black, dashed] (9.7, 0) arc (180:110:3.2);
				
				\draw[black, thick] (16.5, 0) arc (-30:30:3);
				\draw[black, dashed] (16.8, 0) arc (-30:30:3);

				\draw[black, thick] (12.5, 0) arc (-60:0:1.732);
				\draw[black, thick] (14.25, 0) arc (240:180:1.732);
				\draw[black, dashed] (12.8, 0) arc (-60:-33:1.732);
				\draw[black, dashed] (13.95, 0) arc (240:212:1.732);	
		
				\draw[] (1.5, 1.2)  node[scale=0.7]{$\mathfrak{L}(\fD)$};
				\draw[] (15, 1.2)  node[scale=0.7]{$\mathfrak{L}(\fD)$};
				
			\end{tikzpicture}
			
		\end{center}
		
		\caption{\label{f:It} Left: the tangency locations $(E'_{-}(t_1), t_1), (E'_{+}(t_2), t_2)$, and the cusp $(x_c,t_c)$. Right: the enlarged liquid region, with its time slices $I_t^+$ as in \eqref{e:defI_t1} and \eqref{e:defI_t2}.}
		
	\end{figure}

In summary, for a polygonal set $\fP$ satisfying \Cref{pa}, and a cusp point $(x_c, t_c) \in \mathfrak{A}$, (by possibly rotating $\fP$ by $180^\circ$) we can carve out a trapezoid $\fD$ contained in $\fP$ and the time strip $[\ft_0=t_c-\fs, \ft_1=t_c+\fs]$ for a small enough $\fs>0$, such that the followings hold.
		\begin{enumerate} 
			\item The limiting height function $H^*$ is constant along both $\partial_{\ea} (\mathfrak{D})$ and $\partial_{\we} (\mathfrak{D})$.
			\item 
   Denote by $\fA(\fD)=\fD\cap \fA$ the arctic curve in $\fD$, $\fL(\fD)=\fD\cap\fL$ the liquid region restricted to $\fD$, and by $I^*_t$ the closure of $\{x: (x,t)\in \fL(\fD)\}$ for each $\ft_0\le t\le \ft_1$.
   Then, $\fA(\fD)=\{(E_{-}(t), t), (E_{+}(t),t): \ft_0\le t\le t_c\}\cup \{(E'_{-}(t), t), (E'_{+}(t),t): \ft_0\le t\le \ft_1\}$ (See Figure \ref{f:It}). For $t_c\le t\le \ft_1$, the slice $I^*_t$ is a single interval
$
I^*_t= \big[E_{-}'(t), E_{+}'(t) \big];
$
 for $\ft_0\le t\le t_c$, the slice $I^*_t$ consists of two intervals
$
I^*_t= \big[E'_{-}(t), E_{-}(t) \big] \cup \big[E_{+}(t), E'_{+}(t) \big]$.
The complement $\{x:(x,t)\in\fD\}\setminus I^*_t$ consists of several intervals. On each interval, we  have either $\del_xH^*(x,t)\equiv 0$ 
or $\del_xH^*(x,t)\equiv 1$.
\item 
Any tangency location along $ \fA(\fD)$ is of the form $\min I_t^*$ or $\max I_t^*$ for some $t \in (\ft_0, \ft_1)$. At most one tangency location is of the form $\min I_t^*$, and at most one is of the form $\max I_t^*$. Moreover, these tangent locations are contained in either $\del_{\we}(\fD)$ or $\del_{\ea}(\fD)$.
\end{enumerate}

\subsection{Lozenge tiling in a trapezoid with random boundary}
For the trapezoid $\fD$ given above, we next state an optimal rigidity estimate for  uniformly random lozenge tilings on it, with a \emph{random} north boundary height function.
To state it, we need an enlarged version of the time slice
$I^*_t$. 
Fix an arbitrarily small constant $\fd>0$.
For any $\big( E'_\pm(t),t \big)$ on the arctic curve $\fA(\fD)$, we define the distance function 
\begin{align}\label{e:disf}
\tau \big( E'_\pm(t),t \big)=
|t-t_c|^{2/3}n^{6\fd-2/3}\vee n^{ -1+10\fd}.
\end{align}

\noindent Moreover, for any $\big( E_\pm(t), t \big)$ on the arctic curve $\fA(\fD)$ with  $t\le t_c$,  we define the distance function $\tau \big( E_\pm(t),t \big):=
(t_c-t)^{1/6}n^{6\fd-2/3}$.
We then define the enlarged intervals: for $t_c \le t \le \ft_1$,
\begin{align}\label{e:defI_t1}
I_t^+ = \big[ E'_{-}  (t) - \tau (E'_{-} (t), t), E'_{+} (t) + \tau (E'_{+} (t), t) \big], 
\end{align}
and for $\ft_0 \le t < t_c$,
\begin{align}\label{e:defI_t2}
I_t^+= \big[ E'_{-}(t)-\tau(E'_{-}(t), t), E_{-}(t)+\tau(E_{-}(t), t) \big] \cup \big[ E_{+}(t)-\tau(E_{+}(t), t), E'_{+}(t)+\tau(E'_{+}(t), t) \big].
\end{align}
By \Cref{p:rhot} below, $E_{+}(t)-E_{-}(t)\asymp (t_c-t)^{3/2}$. Hence, for $t\ge t_c-n^{-1/2+4\fd}$, \eqref{e:defI_t2} reduces to a single interval 
\begin{align}\label{e:defI_t3}
&I_t^+= \left[E'_{-}(t)-\tau \big( E'_{-}(t), t \big), E'_{+}(t)+\tau \big(E'_{+}(t), t \big) \right].
\end{align}
See Figure \ref{f:It} for an illustration.

Due to some rounding issues, the set $n\fD$ may not be a tileable domain for lozenge tilings. Therefore, we use the notions of plausible boundary height functions and uniformly random height functions from \Cref{defn:uhf}.
\begin{prop}\label{p:rigidityBB}
Denote $\mathsf{D}=n \mathfrak{D}$. 
Let $\sfh:\partial\sfD\to \R$ be a plausible boundary height functions of $\mathsf{D}$, such that (1) $\sfh$ is constant on $\del_{\rm we}\sfD$ and on $\del_{\rm ea}\sfD$, respectively, and (2) on $\del_{\rm so}\sfD$, 
\begin{equation}\label{e:boundaryc0}
\begin{split}
 \big|\sfh(nv)- nH^*(v) \big| \le n^{\mathfrak{d}/3 },\quad & \text{for}\quad  v=(x,\ft_0),\quad \dist(x, I_{\ft_0}^*)\le n^{\fd/3-2/3}, \\
\sfh(nv)= nH^*(v),\quad & \text{for}\quad v=(x,\ft_0), \quad  \dist(x, I_{\ft_0}^*)> n^{\fd/3-2/3}.
\end{split}    
\end{equation}
Then, there exists a \emph{random plausible boundary height function} $\widetilde \sfh$ of $\sfD$, which is equal to $\sfh$ on $\del_{\rm we}\sfD\cup \del_{\rm ea}\sfD\cup \del_{\rm so}\sfD$, such that the following holds with overwhelming probability. Denote by $\widetilde\sfH$ the uniformly random height function of $\mathsf{D}$ with boundary $\widetilde\sfh$. For any $\ft_0\le t\le\ft_1$, we have
\begin{align}\label{e:Optrigidity}
	\begin{aligned}
 \big| \widetilde\sfH(nv)- nH^*(v) \big| \le n^{3\fd },\quad &\text{for}\quad v=(x,t),\quad x\in I_t^+, \\
 \widetilde\sfH(nv)= nH^*(v),\quad &\text{for}\quad v=(x,t),\quad x\not\in I_t^+.
\end{aligned} 
\end{align}
\end{prop}
This proposition is the same as \cite[Proposition 4.4]{huang2021edge}, via the equivalence between lozenge tilings and non-intersecting Bernoulli paths stated in Section \ref{ssec:nbp}.
 
With the above preparations, we complete the proof of \Cref{t:cusprigidity} using the height function comparison (i.e., \Cref{lem:hfmon}) between the uniformly random tiling of $n\fP$ and the uniformly random tiling of $n\fD$ (carved out around a cusp point) with random boundary.

\begin{proof}[Proof of \Cref{t:cusprigidity}]
We carve out a trapezoid $\fD$ around the cusp point $(x_c,t_c)$, as given above, and denote $\sfD=n\fD$.
By \Cref{pla}, $\nabla H^*\equiv (0,0)$, $(1,0)$, or $(1,-1)$ in $\{(x,t): \ft_0\le t\le t_c, x\in [E_{-}(t), E_{+}(t)]\}$.
For the rest of the proof, we assume the first case, while the proofs in the other two cases are very similar and thus omitted.
This assumption implies that $H^*$ is constant in this region. We assume that $H^*\equiv 0$ without loss of generality, as we can always add a global constant to $H^*$.

Take an arbitrarily small constant $\fd>0$.
We denote by $\Omega$ the set of plausible boundary height functions $\sfh:\partial\sfD\to\R$, such that (1) $\sfh$ is constant on $\del_{\rm we}\sfD$ and on $\del_{\rm ea}\sfD$, respectively, and (2) \eqref{e:boundaryc0} holds on $\del_{\rm so}\sfD$, and (3) with overwhelming probability,
the uniformly random height function $\widehat\sfH$ of $\sfD$ with boundary $\sfh$ satisfies
\begin{align}\label{e:rigidity}
    |\widehat \sfH(nv)-nH^*(v)|\le n^{\fd}, \quad v\in \overline\fP\cap \fD.
\end{align}
Then, using \Cref{t:rigidity} and the fact that the west and east boundaries of $\fD$ either coincide with the boundary of $\fP$ or are in the frozen region and bounded away from the liquid region, we see that: with overwhelming probability, the restriction of $\sfH$ on $\del \sfD$ is in $\Omega$.

In the rest of the proof, we fix a $\sfh\in\Omega$, and denote by $\widehat\sfH$ the uniformly random height function $\sfD$ with boundary $\sfh$.
By \Cref{p:rigidityBB}, there is a random plausible boundary height function $\widetilde \sfh$ of $\sfD$, such that (1) $\widetilde\sfh=\sfh$ on $\del_{\we}\sfD\cup \del_{\ea}\sfD\cup \del_{\so}\sfD$, 
and (2) if $\widetilde\sfH$ is a uniformly random height function of $\mathsf{D}$ with boundary $\widetilde\sfh$, then with overwhelming probability, \eqref{e:Optrigidity} holds for any $\ft_0\le t\le \ft_1$.

We consider a small box around the cusp location: $\fB:=[x_c-\fc, x_c+\fc]\times [\ft_0=t_c-\fs, \ft_1=t_c+\fs]$ with $\fc>0$ being a small constant.
By taking $\fc$ small and then $\fs$ small, we can ensure that $\fB\subset \fD$ and the west, north, and east boundaries of $\fB$ are all in the liquid region.
In particular, we have $x_c-\fc\le E_-(t)$ and $x_c+\fc\ge E_+(t)$ for each $\ft_0\le t\le t_c$. 
We next show that with overwhelming
probability, for any $(x,t)\in\partial\fB$, 
\begin{align}\begin{split}\label{e:heightshift}
&\widetilde\sfH(nx-n^{4\fd}, nt)\le  \widehat\sfH(nx,nt),
\quad (t,x)\in \partial \fB.
\end{split}\end{align}
For the south boundary, since $\widetilde \sfH(nx,n\ft_0)=\widehat \sfH(nx,n\ft_0)=\sfh(nx,\ft_0)$ for $x\in [x_c-\fc, x_c+\fc]$, 
\eqref{e:heightshift} follows trivially from the monotonicity of $\widetilde \sfH$. For $(x,t)\in \partial\fB\setminus \del_{\mathrm so}\fD$, it is in the liquid region, so $\partial_x H^*$ is bounded away from $0$ and $1$ in a neighborhood of $(x,t)$.
Therefore,  
\begin{align}\label{e:heightlipschitz}
    H^*(x,t)-H^*(x-n^{4\fd-1},t)\ge cn^{4\fd-1},
    \quad
    H^*(x+n^{4\fd-1},t)-H^*(x,t)\ge cn^{4\fd-1},
\end{align}
for some constant $c>0$. Combining \eqref{e:Optrigidity} with \eqref{e:heightlipschitz}, we then obtain that
\[
    \widetilde\sfH(nx,nt)\ge nH^*(x,t)-n^{\fd}\ge nH^*(x-n^{4\fd-1},t)+cn^{4\fd}-n^{\fd}\ge \widetilde\sfH(nx-n^{4\fd},nt).
\]

Now, given \eqref{e:heightshift}, using \Cref{lem:hfmon}, we can couple $\widetilde \sfH$ on $n\fB-(n^{4\fd},0)$ with $\widehat \sfH$ on $n\fB$ such that
\begin{equation}\label{e:couplingfrozen}
\widetilde\sfH(nx-n^{4\fd}, nt)\le  \widehat\sfH(nx,nt)
\end{equation}
for any $(x,t)\in \fB$. Thus, under the above assumption that $H^*\equiv 0$ in $\{(x,t): \ft_0\le t\le t_c, x\in [E_{-}(t), E_{+}(t)]\}$,
\eqref{e:Optrigidity} and \eqref{e:couplingfrozen} imply that with overwhelming probability, $\widehat\sfH(nx, nt)\ge 0$ for any $\ft_0\le t\le t_c$ and 
\begin{equation}\label{eq:intervalx}
    x\in [E_{-}(t)+ (t_c-t)^{1/6}n^{6\fd-2/3}+n^{-1+4\fd}, E_{+}(t)- (t_c-t)^{1/6}n^{6\fd-2/3}-n^{-1+4\fd}].
\end{equation}
With a similar inequality $\widetilde\sfH(nx+n^{4\fd}, nt)\ge  \widehat\sfH(nx,nt)$, we can show that with overwhelming probability, $\widehat\sfH(nx, nt)\le 0$ for all such $t$ and $x$, thereby concluding $\widehat\sfH(nx, nt)=0$. 
Recall that the interval in \eqref{eq:intervalx} is non-empty only when $t\le t_c-n^{-1/2+4\fd}$, in which case the term $n^{-1+4\fd}$ is negligible. Hence, we have obtained that $\widehat\sfH(nx, nt)=0$ for any $\ft_0\le t\le t_c$ and $x\in [E_{-}(t)+ (t_c-t)^{1/6}n^{\delta-2/3}, E_{+}(t)- (t_c-t)^{1/6}n^{\delta-2/3}] $ as long as we take $\delta>6\fd$. Since $\sfh$ is arbitrarily taken from $\Omega$, by Lemma \ref{lem:Gibbstiling} and the above fact that the restriction of $\sfH$ on $\del \sfD$ is in $\Omega$ with overwhelming probability, the conclusion follows.
\end{proof}

\section{Complex slope and proofs of some deterministic estimates}
\label{s:det}

In this section, we analyze the limiting height function using the complex Burgers equation (\Cref{fequation}). Combining the obtained estimates with the optimal height function concentration results from \Cref{s:optimalR}, we will finish the proofs of the remaining statements in \Cref{sec:univproof}.

We work under the same setup as in \Cref{sec:univproof}.
More precisely, we fix a rational polygonal set $\fP$ satisfying \Cref{pa} and a cusp point $(x_c, t_c) \in \fA$, which is upward oriented as in \Cref{sr} with curvature parameters $\fr,\fq$.
Let $n$ be a large integer such that $n\fP$ is a tileable domain.
All the constants in this section may depend on $\fP$.

We denote $\Delta t=n^{-\omega}$ for some constant $\omega\in(0,1/2)$, and take $t_0<t_c<t_1$, such that $t_0, t_1\in n^{-1}\Z$ and $t_c-t_0, t_1-t_c\asymp \Delta t$. 
Around $(x_c, t_c)$ and between time $t_0$ and $t_1$, the arctic curve $\fA$ contains two analytic pieces $\{(E_{-}(t),t): t_0\le t\le t_c\}$ and $\{(E_{+}(t),t): t_0\le t\le t_c\}$.
Let $\fc>0$ be a small enough constant depending on $\fP$ and $(x_c, t_c)$. Then, $M, N\in\N$ are  defined such that
\[\qq{-M, N}=\{i\in \Z: H^*(x_c-\fc,t_0) \le H^*(x_c,t_c)+i/n < H^*(x_c+\fc,t_0)\},
\]
where $H^*$ is the limiting height function. We denote the density $\rho^*_t(x)=\del_x H^*(x,t)$, which is defined almost everywhere and takes values in $[0,1]$ since $H^*$ is admissible.

Besides the setup in \Cref{sec:univproof}, 
we further assume that $H^*(x,t)=0$ for $t_0 \le t \le t_c$ and $E_-(t)\le x \le E_+(t)$.  Then, the $\rho^*_t$ quantiles $\gamma_i(t)$ are defined through the relation $H^*(\gamma_i(t),t)=i/n$ (and  $\gamma_0(t)$ is chosen to equal $E_+(t)$). 
We also denote $c(t)=x_c+(t-t_c) /\fr$.

\subsection{Density estimate: Proofs of Lemmas \ref{lem:gammaies} and \ref{lem:sfqes}}

We start with the following estimate of the density $\rho_t^*$ in a neighborhood of the cusp location $(x_c,t_c)$. 
\begin{prop}\label{p:rhot}
The followings hold for a sufficiently small constant $c_0>0$ and arbitrarily large $C>0$:
\begin{enumerate}
\item For $t_0\le t\le t_c$, we have $E_{+}(t)-E_{-}(t)\asymp (t_c-t)^{3/2}$. For $0\le x-E_{+}(t) \le C(t_c-t)^{3/2}$, we have 
\begin{align}\label{e:cubiceq}
\rho_t^*(x)=\frac{\sqrt{C_*(x-E_{+}(t))}}{(t_c-t)^{1/4}} +\OO\left((t_c-t)^{1/4}|x-E_{+}(t)|^{1/2}+\frac{|x-E_{+}(t)|}{t_c-t}\right),
\end{align}
where $C_*>0$ is a constant. 
For $C(t_c-t)^{3/2}\le x-E_+(t)\le c_0$, we have $\rho_t^*(x) \asymp  | x-c(t) |^{1/3}$.
Analogous statements hold for $0\le E_{-}(t)-x\le c_0$. 
\item For $t_c< t\le t_1$, we have $\rho_t^*(x) \asymp (t-t_c)^{1/2}\vee | x-c(t) |^{1/3}$ when $|x-c(t)|\le c_0$.
\end{enumerate}
\end{prop}
Part (1) of this lemma has been proved in \cite[Proposition 3.3]{huang2021edge}. 
The proof of part (2) will be given in \Cref{s:proofd}.

The following lemma gives an estimate on $\gamma_0(t)$, which will be proved in \Cref{sec:compare}. 
\begin{lem}\label{lem:gamma0}
    There exists a constant $\fC>0$ such that for any $t_c< t\le t_1$, $|\gamma_0(t)-c(t)|\le \fC(t-t_c)^{3/2}$.
\end{lem}

Combining \Cref{p:rhot} and \Cref{lem:gamma0}, we readily conclude the proof of \Cref{lem:gammaies}.
\begin{proof}[Proof of Lemma \ref{lem:gammaies}]
Recall that $t_c-t_0\asymp\Delta t=n^{-\omega}$ with $\omega\in (0,1/2)$. 
For $1\le i\lesssim \Delta t^2 n$, we can integrate \eqref{e:cubiceq} to get
\begin{align}\label{e:gamma_irelation}
    \frac{i}{n}=\int^{\gamma_i(t_0)}_{E_+(t_0)}\rho_{t_0}^*(x)\rd x\asymp \frac{(\gamma_i(t_0)-E_+(t_0))^{3/2}}{\Delta t^{1/4}},
\end{align}
which yields $\qut_i(t_0) - \Ep(t_0) \asymp \Delta t^{1/6} (i/\lsca)^{2/3}$. This gives the first relation in \eqref{e:gammailoc1}.  The second one follows similarly.

Next, we prove the first relation in \eqref{e:gammailoc2}, and the second relation can be proven in the same way. We recall that  $c(t)=x_c+(t-t_c)/\fr\in [E_-(t), E_+(t)]$. For $t\le t_c$, similar to \eqref{e:gamma_irelation}, using Item 1 of \Cref{p:rhot}, we have that for $i\gtrsim \Delta t^2 n$, 
\begin{align}
    \frac{i}{n}=\int^{\gamma_i(t)}_{E_+(t)}\rho_{t}^*(x)\rd x\asymp (\gamma_i(t)-c(t))^{4/3},
\end{align}
which concludes that $\gamma_i(t)-c(t)\asymp (i/n)^{3/4}$.

For $t\ge t_c$, by \Cref{lem:gamma0} and Item 2 of \Cref{p:rhot}, for any $x\ge c(t)+ 2\fC(t-t_c)^{3/2}$ we have
\begin{align*}
    &\int^{x}_{\gamma_0(t)}\rho_{t}^*(y) \rd y
    \le \int^{x}_{c(t)-\fC\Delta t^{3/2}}\rho_{t}^*(y)\rd y\lesssim \int^{x}_{c(t)-\fC\Delta t^{3/2}}(t-t_c)^{1/2}\vee |y-c(t)|^{1/3}
    \rd y\lesssim (x-c(t))^{4/3},\\
    &\int^{x}_{\gamma_0(t)}\rho_{t}^*(y)\rd y
    \ge \int^{x}_{c(t)+\fC\Delta t^{3/2}}\rho_{t}^*(y)\rd y\gtrsim \int^{x}_{c(t)+\fC\Delta t^{3/2}}(t-t_c)^{1/2}\vee |y-c(t)|^{1/3}
    \rd y\gtrsim (x-c(t))^{4/3}.
\end{align*}
Namely, we have $\int^{x}_{\gamma_0(t)}\rho_{t}^*(y) \rd y \asymp (x-c(t))^{4/3}$ for $x\ge c(t)+ 2\fC(t-t_c)^{3/2}$.
This implies that $\gamma_i(t)-c(t)\asymp (i/n)^{3/4}$ when $i\ge C\Delta t^2 n$ for some sufficiently large constant $C$.
This finishes the proof of \eqref{e:gammailoc2}.
\end{proof}

For \Cref{lem:sfqes}, we also need to use the optimal rigidity proved in \Cref{s:optimalR}.
\begin{proof}[Proof of \Cref{lem:sfqes}]
By \Cref{t:rigidity}, with overwhelming probability, for any $x$ and $t\in [t_0, t_1]\cap n^{-1}\Z$,
\begin{align}  \label{eq:61app}
|\{i\ge 0: \sfq_i(nt)\le nx\}|
=n H^*(x,t)+\OO(n^{\fd}).
\end{align}
It then follows that there is a sufficiently large constant $S>0$, such that for any $j\ge 0$, 
\begin{align}\label{e:gammai_0}
\sfq_0(nt)/n \vee \gamma_{j-\lfloor Sn^{\fd}\rfloor}(t)\le \sfq_j(nt)/n\le \gamma_{j+ \lfloor Sn^{\fd}\rfloor}(t).
\end{align}
By \Cref{t:cusprigidity}, with overwhelming probability, we have $\sfH(nv)= nH^*(v)$ for $v=(x,t)$ with $t\in [t_0, t_c]$ and $x\in [E_{-}(t)+ (t_c-t)^{1/6}n^{\fd-2/3}, E_{+}(t)- (t_c-t)^{1/6}n^{\fd-2/3}]$. We then have $E_{+}(t)- (t_c-t)^{1/6}n^{\fd-2/3} \le \sfq_0(nt)$, which, together with \eqref{e:gammai_0} and \Cref{lem:gammaies}, gives
\begin{align*}
E_+(t_0)-n^{-2/3+\fd}(t_c-t_0)^{1/6}\le \sfq_0(nt_0)/n\le \gamma_{\lfloor S n^{\fd}\rfloor}(t_0)
\le E_+(t_0)+n^{-2/3+\fd}\Delta t^{1/6}.
\end{align*}
Thus, we conclude that $\sfq_0(nt_0)/n-E_+(t_0)\lesssim n^{-2/3+\fd}\Delta t^{1/6}$.
A similar argument leads to the bound for $\sfq_{-1}(nt_0)$ and concludes \eqref{e:concentrate1}.
The statement \eqref{e:concentrate2} is a consequence of  \eqref{eq:61app}, by noticing that
\begin{align*}
    &|\{i \in \llbracket -M, N\rrbracket: \sfq_i(nt_0)<x\lsca\}|=\sfH(x,t_0)+M + \OO(1),\\
    &|\{i \in \llbracket -M, N\rrbracket: \qut_i(t_0)<x\}|=nH(x,t_0)+M + \OO(1).
\end{align*} 

We next prove the first estimate of \eqref{e:concentrate3}.
Recall $L=\lceil n^{1+\delta}\Delta t^2\rceil\gg \Delta t^2n$. Using  \eqref{e:gammai_0}, \eqref{e:gammailoc2}, and Item 1 of \Cref{p:rhot}, we get
\begin{align*}
\sfq_L(nt)/n -\gamma_L(t)\le \gamma_{L + \lfloor S n^{\fd}\rfloor}(t)-\gamma_L(t) \lesssim \frac{n^{\fd}}{n\rho^*_t(\gamma_{L}(t))}
&\lesssim 
\frac{n^{\fd}}{n^{3/4}L^{1/4}},\\
\gamma_L(t)-\sfq_L(nt)/n \le \gamma_L(t) - \gamma_{L-\lfloor Sn^{\fd}\rfloor}(t)\lesssim \frac{n^{\fd}}{n\rho^*_t(\gamma_{L-\lfloor Sn^{\fd}\rfloor}(t))}
&\lesssim 
 \frac{n^{\fd}}{n^{3/4}|L-\lfloor Sn^{\fd}\rfloor|^{1/4}}.
\end{align*}
Then, using that $L\gg n^\delta$ and choosing $\fd$ sufficiently small depending on $\delta$, we conclude that $\sfq_{L}(nt)/n-\gamma_L(t)\lesssim n^{-3/4-\fd}$ for any $t \in [t_0, t_1]\cap n^{-1}\Z$.
The proofs of the second estimate of \eqref{e:concentrate3} and \eqref{e:concentrate4} are similar.
\end{proof}

\subsection{Complex slope revisit}
Recall (from \Cref{s:cslope}) the complex slope $f^*_t(x)$ for $(x,t)\in \fL$.
Consider a box around the cusp location $(x_c,t_c)$ as $\fB=[x_c-\fc', x_c+\fc']\times [t_c-\fc', t_c+\fc']$ for a small enough constant $\fc'>0$.
We denote $\fit=t_c-\fc'$ and $\fft=t_c+\fc'$. 
We can take $\fc'$ small enough such that the complex slope on the liquid region $\fL(\fB)=\fL\cap \fB$ can be reparametrized as an analytic function. 
Such a reparametrization has been done in \cite{huang2021edge}, as summarized in the following proposition.
\begin{prop}
		\label{xhh2} 
There exists a small enough constant $\fc'>0$ such that the followings hold:
\begin{enumerate}
    \item For any $t\in [\fit,\fft]$, let $\fB^t=\{(x,s)\in \fL(\fB): t\le s\le \fft\}$. The following map 
			\begin{align}\label{e:xrmap0}
\varphi_t: (x,s)\in \fB^t\mapsto  x + (t-s)\frac{   f^*_s( x)}{ f^*_s(x)+1}\in \HH\cup \R
\end{align}
is a bijection to its image.
In addition, $(x,s)\mapsto f^*_s(x)$ can be continuously extended to the boundary of $\fB^t$, therefore \eqref{e:xrmap0} can also be continuously extended to the boundary of $\fB^t$. It maps the north, west, and east boundaries of $\fB^t$, $\partial_{\north} \fB^t\cup \partial_{\we}\fB^t\cup \partial_{\ea}\fB^t$, to a curve in the upper half-plane, and the remaining boundary of $\fB^t$ to an interval in $\R$.
Therefore, \eqref{e:xrmap0} and its complex conjugate together give a bijection from two copies of  $\fB^t$, glued along the arctic curve, to a symmetric domain $\mathscr U_t\subseteq \bC$. 

\item 
The complex slope induces a family of analytic functions $f_t: \mathscr U_t\cap \bH \to \bH^-$ for $t\in [\fit,\fft]$, satisfying the following relation:
\begin{align}\label{e:xrmap}
f_t\left( \varphi_t(x,s)\right)= f_s^*(x),\quad (x,s)\in \fB^t.
\end{align}
In particular, for $s=t$ we have $f_t(x)=f_t^*(x)$.
On $\mathscr U_t\cap \HH$, $f_t$ satisfies the complex Burgers equation
\begin{align}\label{e:Burgeq}
 \del_{t}  f_t( z)+\del_{z}  f_t( z ) \frac{f_t( z )}{ f_t( z)+1}=0.
\end{align}
\item \label{i:decompf} 
Recall the density $\rho^*_t(x)=\del_x H^*(x,t)$ defined in \Cref{sec:univproof}, and denote its Stieltjes transform as
\begin{align}\label{e:stieltjes}
 m_t^*(z)=\int \frac{\rho_t^*(x)\rd x}{z-x}.
\end{align}
Then $f_t$ can be extended to the whole domain $\mathscr U_t$, and we have the decomposition
\begin{align}\label{e:gszmut}
 f_t(z)=e^{m_t^*(z)+g_t(z)},
\end{align}
where $g_t$ is a real analytic function on $ {\mathscr U}_t$.\footnote{A function $g$ defined on a subset of $\bC$ is called real analytic if it is analytic and satisfies $\overline{g(z)}=g(\overline z)$.}
\item \label{i:defQ}  
When $\fc'$ sufficiently small, there is a one-variable real analytic function $Q$, such that for any $(x,t)$ in the closure of $\fL(\fB)$, 
\begin{equation}\label{q0f}
    Q(f^*_t(x))=x \big(  f^*_t(x) + 1 \big) - t  f^*_t(x).
\end{equation}
Also, for any $(x,t)\in \fB$, $(x,t)\in \fA(\fB)=\fA\cap\fB$ if and only if $f_{t}(x)$ is a double root of $f\mapsto Q(f)-x(f+1)+tf$, except for that $f_{t_c}(x_c)=f^*_{t_c}(x_c)$ is a triple root (but not a quadruple root) of $f\mapsto Q(f)-x_c(f+1)+t_cf$.
\end{enumerate}
\end{prop}
The first and third items follow from \cite[Proposition 3.4]{huang2021edge}.
The second item follows from \cite[Proposition 3.1]{huang2021edge}. The last item follows from \cite[Proposition A.2]{huang2021edge}, and the classification of singular points (see the discussion
at the end of \cite[Section 1.6]{LSCE}).

The complex Burgers equation \eqref{e:Burgeq} can be solved readily using the characteristic flow. Fix any time $  t\in [\fit, \fft]$ and $u\in \mathscr U_t$, we have that for $  s \in [t, \fft]$,
\begin{align}\label{e:ccff}
\del_s f_s(z_s(u))=0, \quad \del_s z_s(u)= \displaystyle\frac{f_s(z_s(u))}{f_s(z_s(u))+1}=\frac{f_t(u)}{f_t(u)+1}, \quad z_t(u)=u.
\end{align}
The characteristic flow maps the sub-region 
$\{u\in \mathscr U_t\cap \HH: \Im u\ge -(s-t) \Im[f_t(u)/(f_t(u)+1)]\}$ bijectively to $\mathscr U_s\cap \HH$. 
It then follows that $f_s$ satisfies 
\begin{align}\label{e:Burger:sol0}
    f_s\left(u+(s-t)\frac{f_t(u)}{f_t(u)+1}\right)=f_t(u),\quad \fit\le t \le s\le \fft,\quad u\in \mathscr U_t.
\end{align}
For simplicity of notations, we introduce 
$w_t(z):= f_t(z)/(f_t(z)+1)$.
Then \eqref{e:Burger:sol0} can be rewritten as 
\begin{equation}\label{e:Burger:sol1}
    w_s\left(z+(s-t)w_t(z)\right)=w_t(z). 
\end{equation}
Performing Taylor expansion of $Q$ around $f_{t_c}^*(x_c)$ and using \eqref{e:Burger:sol1}, we can show that $w_t(z)$ satisfies the following equation \eqref{zf_tc2}.

\begin{lem}\label{lem:eq_around_cusp2}
        For any $t\in [\fit, \fft]$ and $z\in \mathscr U_t$, we have  
\begin{equation}\label{zf_tc2}
    z - x_c + (t_c-t)w_t(z)  = \frac{\coffa}{3} (w_{t}(z)-w_{t_c}(x_c))^3 +  \g(w_{t}(z)-w_{t_c}(x_c)),
\end{equation}
where $w_{t_c}(x_c)=\mathfrak r^{-1}\in (0,1)$, $\coffa=\frac{\fr^5}{2(\fr-1)^5} Q'''(f^*_{t_c}(x_c)) $ is a positive constant, and $\g$ is an analytic function in a neighborhood around $0$ satisfying $\g(w)=\OO(|w|^4)$.
\end{lem}
\begin{proof}
Recall that the slope $\mathfrak r$ of the tangent line through $(x_c,t_c)$ is in $(1,\infty)$. Together with \eqref{e:slope}, it implies that $w_{t_c}(x_c)=\mathfrak r^{-1} \in (0,1)$ and $f_{t_c}^*(x_c)=(\fr - 1)^{-1}\in (0,\infty)$. Hence, as long as $\fc'$ is chosen sufficiently small, we have $f_{t_c}(z)$ (resp. $w_t(z)$) for $z\in \mathscr U_{t_c}$ is away from $\{-1,0,\infty\}$ (resp. $\{0,1,\infty\}$) by a distance of order 1. 

By \Cref{i:defQ} of \Cref{xhh2}, \eqref{q0f} holds for any $(x,t)\in  \overline {\fL(\fB)}$, and 
\begin{align}\begin{split}\label{e:cuspcharacterization}
    &Q(f_{t_c}^*(x_c))=x_c \big(  f^*_{t_c} (x_c) + 1 \big) - t_c  f^*_{t_c} (x_c),\\
    &Q'(f_{t_c}^*(x_c))=x_c -t_c,\quad
    Q''(f_{t_c}^*(x_c))=0,\quad Q'''(f_{t_c}^*(x_c))\neq 0.
\end{split}\end{align}
Next, with \eqref{e:xrmap} and \eqref{q0f}, we can derive that for $z \in \mathscr U_{t_c}$, the following equation holds:
$$ Q(f_{t_c}(z))=z(f_{t_c}(z)+1)-t_c f_{t_c}(z).$$
Then, with \eqref{e:cuspcharacterization}, performing the Taylor expansion of $Q$ at $f^*_{t_c}(x_c)$ gives that
\begin{equation}\label{zf_tc0}
    \left(z - x_c\right)\left[ 1+f_{t_c}(z)\right]= \frac{1}{6}{Q'''(f^*_{t_c}(x_c))}(f_{t_c}(z)-f^*_{t_c}(x_c))^3 +  {\g_2(f_{t_c}(z)-f^*_{t_c}(x_c))},
\end{equation}
where $\g_2(w)=\OO(|w|^4)$ is an analytic function in a neighborhood around $0$. 
We further write \eqref{zf_tc0} as
\begin{equation}\label{zf_tc}
    z - x_c = \frac{\cofa}{3} (f_{t_c}(z)-f^*_{t_c}(x_c))^3 +  {\g_1(f_{t_c}(z)-f^*_{t_c}(x_c))}, 
\end{equation}
where $\g_1(w)=\OO(|w|^4)$ is the analytic function obtained in this expansion and 
$$\cofa:=\frac{1}{2} \frac{Q'''(f^*_{t_c}(x_c))}{1+f_{t_c}^*(x_c)}=\frac{\fr-1}{2\fr} Q'''(f^*_{t_c}(x_c)) .$$ 
Recall that $f_t(z)$ satisfies \eqref{e:Burger:sol0}, which implies that 
\begin{equation}\label{e:Burger:sol0w}
    f_{t}\left(z\right)=f_{t_c}\left(z+(t_c - t)\frac{f_{t}(z)}{f_{t}(z)+1}\right).
\end{equation}  
By plugging $z$ in \eqref{zf_tc} as $z+(t_c-t)f_t(z)/(f_t(z)+1)$, we get that
\begin{equation}\label{zf_tc1}
    z - x_c +(t_c - t)\frac{f_{t}(z)}{f_{t}(z)+1}= \frac{\cofa}{3} (f_{t}(z)-f^*_{t_c}(x_c))^3 +  {\g_1(f_{t}(z)-f^*_{t_c}(x_c))}.
\end{equation}
Then, plugging $f_t=w_t/(1-w_t)$ into \eqref{zf_tc1}, we can deduce  \eqref{zf_tc2}.

It remains to show that $a$ is positive. In fact, as the cusp is upward oriented, for  $t\in [t_0,t_c)$, $w_t(x)$ is real for $x \in [E_-(t), E_+(t)]$. 
By \Cref{i:defQ} of \Cref{xhh2}, $f_t(E_{\pm}(t))$ are double roots of $f\mapsto Q(f)-E_\pm(t)(f+1)+tf$, so
$$ Q'\left(f_t(E_{\pm}(t))\right)=E_\pm(t) - t. $$
Then, performing the Taylor expansion of $Q'$ around $f_{t_c}^*(x_c)$ and using \eqref{e:cuspcharacterization}, we obtain that 
$$ x_c-t_c + \frac{1}{2} {Q'''(f^*_{t_c}(x_c))}(f_{t}(E_\pm(t))-f^*_{t_c}(x_c))^2 +  \OO(|f_{t}(E_\pm(t))-f^*_{t_c}(x_c)|^3)=E_\pm(t) - t.$$
Plugging into $f_t=w_t/(1-w_t)$ and using \eqref{zf_tc2} with $z=E_\pm(t)$, we can rewrite this equation as 
\begin{align}\label{eq:www0}
     \frac{\fr^4}{2(\fr-1)^4} {Q'''(f^*_{t_c}(x_c))}(w -w_{t_c}(x_c))^2 +  \OO(|w -w_{t_c}(x_c)|^3)=(t_c - t) \left(1 - w \right)
\end{align} 
for $w=w_t(E_\pm(t))$.
Writing the right-hand side as $1 - w= 1-\fr^{-1} - (w-w_{t_c}(x_c))$, we can reduce \eqref{eq:www0} to 
\begin{equation}\label{eq:www}
   a(w - w_{t_c}(x_c))^2 + \frac{\fr}{\fr-1} (t_c-t)(w - w_{t_c}(x_c)) -( t_c - t)= \OO(|w -w_{t_c}(x_c)|^3).
\end{equation} 
Note that as an equation for $w$, \eqref{eq:www} has two real roots around $w_{t_c}(x_c)$, i.e., $w_t(E_\pm(t))$; but that only happens when $a>0$. This concludes the proof.
\end{proof}

\subsection{Matching the curvature parameters: Proof of \Cref{lem:relapears}}

First, we notice that the two analytic pieces $E_{\pm}(t)$, $t\in [t_0,t_c]$, near the cusp are determined by $a$ as follows. 
\begin{lem}
For any $t\in [t_0,t_c]$, we have
    \begin{align}\label{eq:Epm0}
    \begin{aligned} 
E_-(t)=x_c- \frac{t_c-t}{\mathfrak r}-\frac{2 (t_c-t)^{3/2}}{3\sqrt{a}}+\OO(|t_c-t|^2),\quad & w_t(E_-(t))=w_{t_c}(x_c) + \sqrt{\frac{t_c-t}{a}}+\OO(|t_c-t|),\\ 
E_+(t)=x_c- \frac{t_c-t}{\mathfrak r} +\frac{2(t_c-t)^{3/2}}{3\sqrt{a}} +\OO(|t_c-t|^2),\quad & w_t(E_+(t))=w_{t_c}(x_c)- \sqrt{\frac{t_c-t}{a}}+\OO(|t_c-t|).
\end{aligned} 
\end{align}
\end{lem}
\begin{proof}
In the proof of \Cref{lem:eq_around_cusp2} above, we have seen that $w=w_t(E_{\pm}(t))$ satisfies \eqref{eq:www}. 
Solving it, we get the estimates on $w_t(E_\pm(t))$. Plugging them further into \eqref{zf_tc2}, we obtain the estimates on $E_\pm(t)$.
\end{proof}

Comparing \eqref{eq:Epm0} with \eqref{xyql}, we observe that 
\begin{equation}\label{eq:aQ3}
3\fq^{-2}=a=  \frac{1}{2} \frac{\fr^5}{(\fr-1)^5}  Q'''(f^*_{t_c}(x_c)).
\end{equation}

The following lemma computes the derivatives of the complex slope $f_{t_0}$ at $z_c=x_c-(t_c-t_0)/\fr$, as defined in \eqref{e:defzc}.
For the convenience of notations and easier comparison with \Cref{lem:xtzdef}, for the rest of this section, we shift the domain $\fP$ (by an amount depending on $n$) to assume that $t_0=0$. Note that $t_c$ would then be $n$ dependent with $t_c \asymp \Delta t$.

\begin{lem}  \label{lem:xtzdeft}
We have
\begin{align}\label{eq:fderis}
\begin{aligned}
  &  f_0(z_c)=(\fr-1)^{-1},\quad f_0'(z_c)=-t_c^{-1}\frac{\fr^2}{(\fr-1)^{2}},\quad f_0''(z_c)=2t_c^{-2}\frac{\fr^3}{(\fr-1)^{3}},\\
&f_0'''(z_c)=-t_c^{-4}\frac{\fr^7}{(\fr-1)^7} Q'''(f_{t_c}(x_c))   - 6t_c^{-3}\frac{\fr^4}{(\fr-1)^{4}} .
\end{aligned} 
\end{align}
\end{lem}

\begin{proof}
First, by \eqref{e:Burger:sol0}, we have $f_0(z_c)=f_{t_c}(x_c)=(\fr-1)^{-1}$. Then,  \eqref{e:cuspcharacterization} gives that
\begin{equation}\label{e:cuspcharacterization2}
    Q'(f_0(z_c))=x_c -t_c,\quad
    Q''(f_0(z_c))=0,\quad Q'''(f_0(z_c))\neq 0
\end{equation}
Next, by the relation \eqref{e:xrmap},
\begin{align*}
    f_0(\varphi_0(x,t))=f_t^*(x),\quad \varphi_0(x,t)=x - t\frac{f_t^*(x)}{f_t^*(x)+1}.
\end{align*}
Denoting $z=\varphi_0(x,t)$, and plugging the above line into \eqref{q0f}, we get that
\begin{align}\label{e:Q0f0}
    \frac{Q(f_0(z))}{f_0(z)+1}=x-t\frac{f_t^*(x)}{f_t^*(x)+1}=z \ \ \Rightarrow \ \ Q(f_0(z))=z(f_0(z)+1).
\end{align}
Taking the derivative of \eqref{e:Q0f0} with respect to $z$ gives
\begin{align}\label{e:der1}
    Q'(f_0(z))f_0'(z)=f_0(z)+1+z f_0'(z).
\end{align}
Plugging $z=z_c$ into \eqref{e:der1}, using \eqref{e:cuspcharacterization2} and $z_c=x_c-t_c w_0(z_c)$, we get
\begin{equation}  \label{eq:defcust2}
 1+\cust \frac{\fff'(\zzc)}{(\fff(\zzc)+1)^2}=0,
\end{equation}
from which we can solve $\fff'(\zzc)$. 
Taking one more derivative of \eqref{e:der1} with respect to $z$, we get
\begin{align}\label{e:der2}
    Q''(f_0(z))(f_0'(z))^2+Q'(f_0(z))f_0''(z)=2f_0'(z)+z f_0''(z).
\end{align}
Plugging $z=z_c$ into \eqref{e:der2} and using \eqref{e:cuspcharacterization2} and \eqref{eq:defcust2}, we get
\begin{equation}  \label{eq:defzzzc2}
\fff''(\zzc)-\frac{2\fff'(\zzc)^2}{\fff(\zzc)+1}=0,
\end{equation}
from which we can solve $\fff''(\zzc)$. 
Finally, taking another derivative of \eqref{e:der2} with respect to $z$, we get
\begin{align} \label{e:der3}
    Q'''(f_0(z))(f_0'(z))^3 + 3 Q''(f_0(z)) f_0'(z)  f_0''(z)  +Q'(f_0(z))f_0'''(z)=3f_0''(z)+z f_0'''(z).
\end{align}
Plugging $z=z_c$ into \eqref{e:der3} and using \eqref{e:cuspcharacterization2}, we get
\begin{align*} 
   (z_c-x_c+t_c)f_0'''(z_c)=  Q'''(f_0(z_c))(f_0'(z_c))^3 -3f_0''(z_c) ,
\end{align*}
from which we can solve $f_0'''(z_c)$. 
\end{proof}

Now, we are ready to complete the proof of \Cref{lem:relapears} with the above lemma and \Cref{lem:cuspclos}.
\begin{proof}[Proof of \Cref{lem:relapears}]
By \Cref{lem:cuspclos}, we have $\widetilde B=(\wt x_c-\wt z_c)/\wt t_c=\wt w_{0}(\wt z_c)\to \fr^{-1}$, and that  
$$\widetilde t_c+\widetilde z_c-\widetilde x_c=t_c+z_c-x_c +\OO(\Delta t^2)= (1-\fr^{-1})t_c+\OO(\Delta t^2)\asymp \Delta t,\quad \widetilde x_c-\widetilde z_c = \fr^{-1}t_c+\OO(\Delta t^{2})\asymp \Delta t.$$
These two estimates show that the second and third terms in the definition of $\widetilde A$ are of order $\OO(\Delta t)$. 
Next, by \eqref{eq:Epm0}, we have that 
\begin{equation}\label{eq:zcingap}
     z_c\in (E_-(0), E_+(0)),\quad \text{with}\quad  z_c-E_-(0), E_+(0)-z_c\asymp \Delta t^{3/2},
\end{equation}
and a similar estimate holds for $\wt z_c=z_c+\OO(\Delta t^2)$.
Then, combining \eqref{eq:zcingap} with \eqref{eq:mmmfo}, we get that  
$$ \int \frac{ \widetilde\rho_{0}(x) }{(\widetilde z_c-x)^4}\rd x -\int \frac{ \widetilde\rho_{0}(x) }{(z_c-x)^4} \rd x\lesssim  \Delta t^{-11/2}\Delta t^2 =\Delta t^{-7/2},$$
which is negligible under the scaling $ \Delta t ^4$.  
Furthermore, using the fact that $ z_c$ is away from the support of $ \rho_{0}^*-\widetilde\rho_{0}$, which is contained in $\R\setminus [\gamma_{-M}(0), \gamma_N(0)]$, by a distance of order 1, we easily get that 
$$\widetilde t_c^4 \int \frac{|\rho_{0}^*(x)-\widetilde\rho_{0}(x)|}{4(z_c-x)^4} \rd x =\OO(\Delta t^4). $$
Combining the above facts and using that ${\wt t_c}/{t_c}=(t_c +\OO(\Delta t^2))/{t_c}\to 1$, we observe that to show the limit of $\wt A$, it suffices to prove 
\begin{equation}\label{eq:match_para}
    t_c^4 \int \frac{\rho_{0}^*(x)}{4(z_c-x)^4} \rd x = \frac{-t_c^4}{24}m'''_{0}(z_c)\to \fr^2(\fr-1)^{-1}\fq^{-2}/4.
\end{equation}  
Using the decomposition \eqref{e:gszmut} and that $g_0(z)$ is real analytic, we can calculate that 
\begin{align}\label{eq:m'''}
    m'''_{0}(z_c) = [\log f_0(z_c)]'''+\OO(1)= \frac{f_0'''(z_c)}{f_0(z_c)} - \frac{3f_0''(z_c)f_0'(z_c)}{f_0(z_c)^2}+\frac{2 f_0'(z_c)^3}{f_0(z_c)^3}+\OO(1).
\end{align}  
Plugging \eqref{eq:fderis} into \eqref{eq:m'''}, we obtain that 
\begin{align}\label{eq:m'''2}
-\frac{t_c^4}{24}  m'''_{0}(z_c) \to \frac{\fr^7}{24(\fr-1)^6} Q'''(f_{t_c}(x_c)) .
\end{align}
Finally, plugging \eqref{eq:aQ3} into \eqref{eq:m'''2} concludes \eqref{eq:match_para}.    
\end{proof}

\subsection{Density estimate: Proof of Item 2 of \Cref{p:rhot}}\label{s:proofd}
By \Cref{i:decompf} of \Cref{xhh2}, we can recover the density $\rho^*_t$ as
$$\rho^*_t(x)=-\frac{1}{\pi} \arg^* f_t(x)=-\frac{1}{\pi}\arg^* \frac{w_t(x)}{1-w_t(x)}.$$
For $w_t(x)$ in a sufficiently small neighborhood of $w_{t_c}(x_c)\in (0,1)$, we have 
\begin{equation}\label{eq:wrho} \left|\frac{w_t(x)}{1-w_t(x)}\right| \gtrsim 1 \quad \text{and}\quad -\im \frac{w_t(x)}{1-w_t(x)} = \frac{-\im w_t(x)}{|1-w_t(x)|^2} ,
\end{equation}
which gives that $\rho_t(x)\sim -\im w_t(x)$. Hence, to prove Item 2 in \Cref{p:rhot}, we only need to estimate the order of $\im w_t(x)$. For simplicity of notations, given $x, t$ with $t_c< t \le t_1$ and $|x-c(t)|\le c_0$, we denote
$$ \varpi:=w_{t}(x)-w_{t_c}(x_c),\quad \tau := \frac{t-t_c}{a},\quad y:=\frac{3(x-c(t))}{2a},\quad \e(\varpi):=\frac{3\g(\varpi)}{2a}, $$
where $\g$ is from \Cref{lem:eq_around_cusp2} and $\g(\varpi)=\OO(|\varpi|^4)$.
Then we can rewrite the equation \eqref{zf_tc2} as
\begin{equation}\label{zf_tc3}
\varpi^3 + 3\tau \varpi + 2 \e(\varpi)=2y.
\end{equation}
Using the general cubic formula, we obtain that 
\begin{equation}\label{eq:w_cubic}
 \varpi= \al \left[(y-\e(\varpi))+\sqrt{(y-\e(\varpi))^2+\tau^3}\right]^{1/3} + \al^{-1}\left[(y-\e(\varpi))-\sqrt{(y-\e(\varpi))^2+\tau^3}\right]^{1/3}, 
\end{equation}
where $\al$ is a primitive cube root of unity chosen such that $\im \varpi<0$. Here (and also for the rest of this paper) we use the convention that $z^{1/3}\in\R$ for $z\in\R$.

We first consider the case $|y|\le C_0\tau^{3/2}$ for a large enough constant $C_0>0$. From equation \eqref{zf_tc3}, we obtain that $\varpi=\OO(\tau^{1/2})$. Then, we can expand \eqref{eq:w_cubic} as 
$$  \varpi= \al \left[y+\sqrt{y^2 +\tau^3}+\OO(\tau^2)\right]^{1/3} + \al^{-1}\left[y-\sqrt{y^2 +\tau^3}+\OO(\tau^2)\right]^{1/3}.  $$
Then, we have
\begin{align*}
    -\im \varpi &= \frac{\sqrt{3}}{2}\left[y+\sqrt{y^2 +\tau^3}+\OO(\tau^2)\right]^{1/3}- \frac{\sqrt{3}}{2}\left[y-\sqrt{y^2 +\tau^3}+\OO(\tau^2)\right]^{1/3}\sim \tau^{1/2}.
\end{align*} 
Therefore, we conclude that when $3|x-c(t)|/2a=|y|\le C_0\tau^{3/2}\asymp (t-t_c)^{3/2}$, we have \begin{align}\label{e:brhot1}
    \rho_t(x)\asymp -\Im w_t(x)=-\Im \varpi\asymp \tau^{1/2}\asymp (t-t_c)^{1/2}.
\end{align}

Next, consider the case $C_0\tau^{3/2}\le |y| \le c_0$. From equation \eqref{zf_tc3}, we obtain that $\varpi=\OO( |y|^{1/3})$, with which we can expand \eqref{eq:w_cubic} as 
$$  \varpi= \al \left[y+|y|+\OO\left(|y|^{4/3}+\frac{\tau^3}{|y|}\right)\right]^{1/3} + \al^{-1}\left[y-|y|+\OO\left(|y|^{4/3}+\frac{\tau^3}{|y|}\right)\right]^{1/3}. $$
Then, we have
\begin{align*}
    -\im \varpi &= \frac{\sqrt{3}}{2}\left[2|y|+\OO\left(|y|^{4/3}+\frac{\tau^3}{|y|}\right)\right]^{1/3}- \frac{\sqrt{3}}{2}\left[\OO\left(|y|^{4/3}+\frac{\tau^3}{|y|}\right)\right]^{1/3}\sim |y|^{1/3}.
\end{align*} 
Therefore, we conclude that when $c_0\ge 3|x-c(t)|/2a=|y|\ge C_0\tau^{3/2}\asymp (t-t_c)^{3/2}$, we have 
\begin{align}\label{e:brhot2}
\rho_t(x)\asymp -\Im w_t(x)=-\Im \varpi\asymp |y|^{1/3}\asymp |x-x_c|^{1/3}
\end{align}

Item 2 of \Cref{p:rhot} follows from combining \eqref{e:brhot1} and \eqref{e:brhot2}. 

\subsection{NBRW estimates: Proofs of \Cref{lem:cuspclos} and \Cref{lem:tsfQes}}
We first show that the complex slope corresponding to the limit shape of an NBRW also solves a complex Burgers equation.

\begin{prop}\label{prop:slopeNBRW}
    Take any $\beta\in (0,1)$ and a density $ \widetilde\rho_0:\R\to [0,1]$. There exists a process $\{\widetilde \rho_t\}_{t\ge 0}$ with Stieltjes transform 
    \begin{align}\label{eq:def_wtft}
        \widetilde m_t(z):=\int \frac{\widetilde\rho_t(x)\rd x}{z-x},\quad \widetilde f_t(z):=\frac{\beta}{1-\beta}e^{\widetilde m_t(z)},
    \end{align}
    which solves the complex Burgers equation
    \begin{align}\label{e:cBurger}
        \del_t \widetilde f_t(z)+\del_z \widetilde f_t(z)
        \frac{\widetilde f_t(z)}{\widetilde f_t(z)+1}=0, \quad z\in \bH.
    \end{align}
    
\end{prop}
\begin{proof}
We recall the free convolution with the semicircle law from random matrix theory. The \emph{semicircle distribution} is described by the density
		$\varrho_{\semci} (x) = \sqrt{4-x^2}/(2\pi) \cdot \don_{x \in [-2, 2]}$.
For any $t > 0$, we denote the rescaled semicircle density as $
		\varrho_{\semci}^{(t)} (x) := t^{-1/2} \varrho_{\semci} (t^{-1/2} x)
$. 	Given a positive measure $\nu$, the free convolution $\nu_t:= \nu\boxplus\varrho_{\semci}^{(t)}$ of $\nu$ with $\varrho_{\semci}^{(t)}$ is characterized by its Stieltjes transform 
$\chi_t(z)=\int \frac{\rd \nu_t(x)}{z-x}, $
which satisfies the equation
\begin{align}\label{e:chi_t}
    \chi_t(z+t\chi_0(z))=\chi_0(z).
\end{align}

The complex Burgers equation \eqref{e:cBurger} can be solved using characteristic flow as
\begin{align}\label{e:tft}
    \widetilde f_t\left(z+t \frac{\widetilde f_0(z)}{\widetilde f_0(z)+1}\right)=\widetilde f_0(z).
\end{align}
Now, we define
\begin{align}\label{e:chitrelation}
    \chi_0(z):=\frac{\widetilde f_0(z)}{\widetilde f_0(z)+1}-\beta.
\end{align}
Then, for $z\in \bH$, we have $\Im[ \widetilde m_0(z)]\in (-\pi,0)$, and 
\begin{align}
    \Im[\chi_0(z)]= \frac{\Im \widetilde f_0(z)}{|\widetilde f_0(z)+1|^2}
    =\frac{\beta}{1-\beta}\frac{|e^{\widetilde m_0(z)}| \cdot \im e^{\ri \Im[\widetilde m_0(z)]}}{|\widetilde f_0(z)+1|^2}<0.
\end{align}
Moreover, by our construction, $\lim_{z\rightarrow \infty}\chi_0(z)=0$.
Hence, by the Nevanlinna representation, there exists a positive measure $\nu$ such that $\chi_0(z)$ is the Stieltjes transform of $\nu$. 
Then, we can construct $\chi_t(z)$ as in \eqref{e:chi_t}, which is the Stieltjes transform of $\nu_t= \nu\boxplus\varrho_{\semci}^{(t)} $. Once we have constructed $\chi_t(z)$, we let 
\begin{align}
    \chi_t(z)=\frac{\widetilde f_t(z+\beta t)}{\widetilde f_t(z+\beta t)+1}-\beta, \quad  \wt f_t(z)=\frac{\beta}{1-\beta} e^{\widetilde m_t(z)}=\frac{\chi_t(z-\beta t)+\beta}{1-\beta-\chi_t(z-\beta t)}.
\end{align}
With \eqref{e:chi_t}, we can readily check that $\widetilde f_t$ satisfies the complex Burgers equation \eqref{e:tft}.
For $z\in \bH$, we have $\chi_t(z)\in \bH^-$, thus the above construction gives $e^{\widetilde m_t(z)}\in \bH^-$ and $\Im[\widetilde m_t(z)]\in (-\pi,0)$. Moreover, we have $\lim_{z\rightarrow \infty}\widetilde m_t(z)=0$.
Using the Nevanlinna representation again, there exists a density $\wt \rho_t:\R\to [0,1]$ such that 
$
    \widetilde m_t(z)=\int\frac{\widetilde \rho_t(x)\rd x}{z-x}.
$
This gives the construction of the process $\{\widetilde \rho_t(x)\}_{t\ge 0}$.
\end{proof}

For the convenience of notations, in the rest of this section, we also shift the domain $\fP$ by an amount depending on $n$ such that $x_c=(t_c-t_0)/\fr=t_c/\fr$. 
Then, $x_c$ would be $n$ dependent and $z_c=0$ from \eqref{e:defzc}.

Below we take $\widetilde\rho_0=\widetilde\rho_{t_0}$ from \eqref{eq:deftildem}, and let $\widetilde\rho_t$, $\widetilde m_t$, $\widetilde f_t$ for $0\le t\le t_1$ be given by \Cref{prop:slopeNBRW}. Denote $\wt w_t(z):= {\wt f_t(z)}/({\wt f_t(z)+1})$. 
Take $\beta$ from \eqref{e:chosebeta}. Then, we have
\begin{equation}\label{eq:choosebeta2}
    \wt f_{0}(0) = f_{t_c}(x_c) = f_0\left(x_c- t_cw_{t_c}(x_c) \right) = f_0(0),
\end{equation}
where we used \eqref{e:Burger:sol0w} in the second equality and $w_{t_c}(x_c)=\mathfrak r^{-1}$ in the third equality. 

We denote  $\Delta w_0(z)= \wt w_0(z) - w_0(z)$ for $z\in \mathscr U_0$. 
We claim that  
\begin{equation}  \label{eq:g-wtg}
\Delta w_0(z) = \OO( \left|z\right|). 
\end{equation}
For the proof of this claim, notice that $g_0(z)$ in the decomposition \eqref{e:gszmut} is a real analytic function, so  $g_0(z)-g_0(0)\lesssim |z|$. Furthermore,
$$ \left|\left[\wt m_{0}(z) - m_{0}^*(z)\right]-\left[\wt m_{0}(0) - m_{0}^*(0)\right]\right| \le  \int_{x\notin [\gamma_{-M}(0), \gamma_N(0)]}\left|\frac{1}{x-z}-\frac{1}{x}\right|\rho_{0}^*(x)\rd x \lesssim |z|,$$
where we used that $x, x-z \gtrsim 1$ for $x\notin [\gamma_{-M}(0), \gamma_N(0)]$ and $z\in \mathscr U_0$ as long as $\fc'$ is chosen sufficiently small depending on $\fc$. Thus, from \eqref{eq:choosebeta2} and \eqref{e:gszmut}, we derive that 
\begin{align}\label{eq:wtf_f}
    \wt f_{0}(z)= f_{0}(0) e^{\wt m_0(z) - \wt m_0(0)}= f_{0}(z) e^{\wt m_0(z) - \wt m_0(0) - ( m_0^*(z) - m_0^*(0))-(g_0(z)-g_0(0))} = f_0(z)\left( 1+\deltaw(z)\right),
\end{align}
where $\deltaw(z)$ is an analytic function around 0 defined as 
$$ \deltaw(z)=e^{\wt m_0(z) - \wt m_0(0) - ( m_0^*(z) - m_0^*(0))-(g_0(z)-g_0(0))}-1=\OO(|z|).$$
With the above two equations, we conclude that
\begin{align}\label{eq:Deltaw0}
    \Delta w_0(z)=\frac{\wt f_0(z)}{\wt f_0(z)+1}-\frac{f_0(z)}{f_0(z)+1}
    =\frac{\wt f_0(z)-f_0(z)}{(f_0(z)+1)(\wt f_0(z)+1)} = \frac{w_0(z)(1-w_0(z))\deltaw(z)}{w_0(z)\deltaw(z)+1} =\OO(|z|),
\end{align}
where for the last step we used that $w_0(z)$ is bounded.

\subsubsection{Proof of \Cref{lem:cuspclos}}
We first prove the estimate for $\wt t_c-t_c$. Recall from \Cref{lem:xtzdef} that there is an a priori estimate $\wt t_c=\wt t_c-t_0\sim \Delta t$. 
We denote the two edges of $\wt \rho_t$ for $0\le t \le \wt t_c$ as $\wt E_{\pm}(t)$, so that $\wt \rho_t(x)=0$ for $x\in [\wt E_{-}(t), \wt E_{+}(t)]$.
It is known from classical Stieltjes transform theory that $\wt E_{\pm}(t)$ are characterized as the points $x\in \R$ where $\wt m_{t}'(x)$ diverges, which, by the definitions of $\wt f_t$ and $\wt w_t$ in \eqref{eq:def_wtft}, implies that 
$1/ \wt w_t'(\wt E_{\pm}(t))=0.$
Similar to \eqref{e:Burger:sol1}, from equation \eqref{e:cBurger}, we obtain that $ \wt w_t(z)= \wt w_0 (z - t\wt w_t(z))$ for $t\in [0,t_1]$ and $z \in \mathscr U_t$. Then, the implicit differentiation with respect to $z$ yields that 
$ [{\wt w_t'(z)}]^{-1}= [\wt w_0' (z - t\wt w_t(z))]^{-1}+t,$ 
so $\wt E_{\pm}(t)$ satisfy the equation 
\begin{equation}\label{eq:Epmt}
    \frac{1}{\wt w_0'  ( \wt E_{\pm}(t) - t\wt w_t(\wt E_{\pm}(t)) )}+t=0.
\end{equation}
Furthermore, with the definition of $\Delta w_0$ in \eqref{eq:Deltaw0}, we can calculate that  
\begin{equation*} 
\wt w_0'(z)=w_0'(z)+\Delta w_0'(z)=w_0'(z) + \OO(|w_0'(z)\deltaw(z)| + |\deltaw'(z)|)=w_0'(z)(1+\OO( |z|)) + \OO(1).
\end{equation*}
Plugging it into \eqref{eq:Epmt}, we get that
\begin{equation}\label{eq:Epmt2}
\frac{1}{w_0'  ( \wt E_{\pm}(t) - t\wt w_t(\wt E_{\pm}(t)) )} + t +\OO \left(t^2+t |\wt E_{\pm}(t) - t\wt w_t(\wt E_{\pm}(t)) | \right)=0. 
\end{equation}

For simplicity of notations, we rewrite equation \eqref{zf_tc2} with $t=0$ and $x_c=t_c/\fr$ as
\begin{equation}\label{eq_zw}
    z = F_0(w_0(z)-w_{t_c}(x_c)) , \quad \text{for}\quad F_0: w\mapsto  -  t_c w+ \frac{a}{3} w^3 + \cal E(w ).
\end{equation}
Then, the implicit differentiation of \eqref{eq_zw} with respect to $w_0$ and taking $z= \wt E_{\pm}(t) - t\wt w_t(\wt E_{\pm}(t))$ gives 
$$ \frac{1}{w_0'  ( \wt E_{\pm}(t) - t\wt w_t(\wt E_{\pm}(t)) )} = F'_0\left(w_0( \wt E_{\pm}(t) - t\wt w_t(\wt E_{\pm}(t)) )- w_{t_c}(x_c)\right).$$
In addition, taking $z= \wt E_{\pm}(t) - t\wt w_t(\wt E_{\pm}(t))$ in \eqref{eq_zw}, we get that 
\begin{equation}\label{eq:Epmt3}
\wt E_{\pm}(t) - t\wt w_t(\wt E_{\pm}(t)) = F_0(w_0 ( \wt E_{\pm}(t) - t\wt w_t(\wt E_{\pm}(t)) ) - w_{t_c}(x_c) ).
\end{equation}
Hence, \eqref{eq:Epmt2} can be rewritten as the following equation of $w=w_0 ( \wt E_{\pm}(t) - t\wt w_t(\wt E_{\pm}(t)) ) - w_{t_c}(x_c) $:
\begin{equation}\label{eqF2}
a w^2 =t_c - t +\OO \left(t^2+t |F_0(w)| +|w|^3\right)=t_c - t +\OO \left(t^2 +|w|^3\right) .
\end{equation}
At the cusp $(\wt x_c,\wt t_c)$, we have $\wt E_{+}(\wt t_c)=\wt E_{-}(\wt t_c)$. Hence, the above equation \eqref{eqF2} of $w$ has a double root around 0 when $t=\wt t_c$, from which we readily get that $\wt t_c-t_c=\OO(\Delta t^2)$.  

For the estimate on $\wt x_c-x_c$, from \eqref{eqF2}, we can solve that $w_0 ( \wt x_c - \wt t_c \wt w_{\wt t_c }(\wt x_c) ) - w_{t_c}(x_c)=\OO(\Delta t)$. Applying it to \eqref{eq:Epmt3} and \eqref{eq:Deltaw0}, we get that $\wt z_c - z_c=\wt x_c - \wt t_c \wt w_{\wt t_c }(\wt x_c) =\OO(\Delta t^2)$ (recall that $z_c=0$) and 
$$
\wt w_{0}(\wt z_c) - \fr^{-1}= \wt w_0 ( \wt x_c - \wt t_c \wt w_{\wt t_c}(x_c) ) - w_{t_c}(x_c)=\OO(\Delta t + |\wt x_c - \wt t_c \wt w_{\wt t_c }(\wt x_c)|)=\OO(\Delta t).$$
With these two estimates and the estimate on $\wt t_c$, we finally get that 
$$ \wt x_c = \wt t_c \wt w_{\wt t_c }(\wt x_c) + \OO(\Delta t^2) = ( t_c+ \OO(\Delta t^2))(\fr^{-1} + \OO(\Delta t))+ \OO(\Delta t^2) = x_c + \OO(\Delta t^2).$$

\subsubsection{Proof of \Cref{lem:tsfQes}}
We can define the NBRW height function as 
\begin{equation}\label{eq:deftildeH}
    \wt H(x,t)=-M/n + \int_{-\infty}^x \wt \rho_t(y)\rd y.
\end{equation}
Then, for $t \in [0, t_1]$ and $i\in\qq{-M, N}$, we define $\wt\gamma_i(t)$ as in \eqref{eq:defgammai} with $H^*$ replaced by $\wt H$. Similar to \eqref{fh}, the complex slope $\wt f_t$ is related to the height function $\wt H(x,t)$ through
\begin{equation}\label{tildefh}
    \arg^* \wt f_t(x) = - \pi \partial_x \wt H (x,t), \qquad \arg^* \left( \wt f_t(x) + 1 \right) = \pi \partial_t \wt H (x,t).
\end{equation}
\begin{proof} [Proof of \eqref{tildefh}]
The first equation follows directly from the definition of $\wt H$ and \eqref{eq:def_wtft}.
The second equation can be derived by
\begin{align*}
    \pi \partial_t \wt H (x,t) &= -\im \int_{-\infty}^x \partial_t  \wt m_t(y)\rd y=- \im \int_{-\infty}^x \partial_t  \log \wt f_t(y)\rd y =\im \int_{-\infty}^x \partial_y \log (\wt f_t(y)+1)\rd y \\
    &=\im \log (\wt f_t(x)+1) =\pi \arg^* (\wt f_t(x)+1),
\end{align*} 
where we used the complex Burgers equation \eqref{e:cBurger}, rewritten as 
$\partial_t \log \wt f_t(z)+\partial_z \log (\wt f_t(z)+1)=0$.
\end{proof}
We need the following optimal rigidity estimate, for the NBRW  $\widetilde\sfQ=\{\widetilde\sfq_i\}_{i=-M}^N:\qq{0, \infty}\to \Z^{\qq{-M,N}}$ constructed in \Cref{ssec:cnbrw}.
It follows from \Cref{lem:sfqes} and \cite[Proposition 4.4]{huang2021edge} (which has been stated as \Cref{p:rigidityBB} in the tiling setting).
\begin{lem}\label{p:localwalk0}
Under the setting of \Cref{lem:tsfQes}, with overwhelming probability:
\begin{align*}
 \widetilde\sfq_{L}(nt)/n-\wt\gamma_L(t), \ \widetilde\sfq_{-L}(nt)/n-\wt\gamma_{-L}(t)  \lesssim n^{-3/4-\fd},&\quad \forall t \in [0, t_1]\cap n^{-1}\Z,\\
 \widetilde\sfq_{i}(nt_1)/n-\wt \gamma_i(t_1) \lesssim n^{-3/4-\fd},&\quad \forall i\in\qq{-L, L}.\end{align*}
\end{lem}

Now, to conclude \Cref{lem:tsfQes}, it remains to show that the quantiles $\wt \gamma_i$ are sufficiently close to $\gamma_i$, which is the content of the following lemma. It will be proved in the next subsection.
\begin{lem}\label{l:locwalk_comp}
Under the setting of \Cref{lem:tsfQes}, we have  
\begin{align}\label{eq:bdd_stab}
 |\wt \gamma_{L}(t)-\gamma_L(t)| + |\wt \gamma_{-L}(t)-\gamma_{-L}(t)|\le n^{\delta}\Delta t^{5/2},&\quad \forall t\in [0,t_1],\\\label{eq:bdd_stab2}
 |\wt\gamma_{i}(t_1)-\gamma_i(t_1)|\le n^{\delta}\Delta t^{2} ,&\quad \forall i\in \qq{-L,L}.
\end{align}
\end{lem}

\subsection{Evolution of quantiles: Proofs of \Cref{lem:gamma0} and \Cref{l:locwalk_comp}}\label{sec:compare}
We first define functions  
\[
    h_t(z):=z + t w_0(z),\quad \wt h_t(z):=z+t\wt w_0(z)=z + t(w_0(z)+\Delta w_0(z)).
\]
Then, \eqref{e:Burger:sol1} and a similar equation for $\wt w_t$ from \eqref{e:cBurger} give that, for $t\in [0,t_1]$ and $\xi \in \mathscr U_t$,
\begin{equation}\label{eq:ht_ht2}
w_t(\xi)=w_0(h_t^{-1}(\xi)),\quad \wt w_t(\xi)=\wt w_0(\wt h_t^{-1}(\xi)).
\end{equation}
Using \eqref{eq:g-wtg}, we get that 
\begin{align*} 
w_t(\xi)-\wt w_t(\xi) & = w_0(h_t^{-1}(\xi)) - w_0(\wt h_t^{-1}(\xi)) - \Delta w_0(\wt h_t^{-1}(\xi))\\
&= w_0(u_t) - w_0({\wt u_t})  - \Delta w_0({\wt u_t}) = \wtt(\xi) - \wwtt(\xi) + \OO(|{\wt u_t}(\xi)|), 
\end{align*}
where we abbreviated that
\[
u_t = u_t(\xi)=h_t^{-1}(\xi),\;  {\wt u_t}= {\wt u_t}(\xi)=\wt h_t^{-1}(\xi),\;
\wtt=\wtt(\xi) = w_0(u_t(\xi))- \fr^{-1} ,\;  \wwtt=\wwtt(\xi) = w_0({\wt u_t}(\xi)) - \fr^{-1} . 
\]
These variables satisfy the following equations:
\begin{align} 
& u_t + t \wtt = \xi - \fr^{-1} t  = {\wt u_t} + t(\wwtt+\Delta w_0({\wt u_t})), \label{eq_zw0}\\
& u_t = F_0(\wtt) =-t_c \wtt + \frac{a}{3} \wtt^3 + \cal E(\wtt), \label{eq_zw1} \\
& {\wt u_t} = F_0(\wwtt)  = -t_c \wwtt + \frac{a}{3} \wwtt ^3 + \cal E(\wwtt),\label{eq_zw2} 
\end{align}
where \eqref{eq_zw1}  and \eqref{eq_zw2} are by \eqref{eq_zw}.

From $\rd H^*(\gamma_i(t),t)/\rd t=0$, using \eqref{fh}, we can derive that
\begin{align}\label{evol_quant}
    \gamma_i'(t)= \frac{\arg^*[f_t(\gamma_i(t))+1]}{\arg^*[f_t(\gamma_i(t))]}.
\end{align}
A similar differential equation for $\wt\gamma_i(t)$ with $f_t$ replaced by $\wt f_t$ can also be derived using \eqref{tildefh}.
Since $f_{t_c}(x_c)=(\fr-1)^{-1}$ is a positive constant, for  $z=f_{t_c}(x_c)+\oo(1)$, we have that
$$\frac{\arg^*(z +1)}{\arg^*(z)}=\frac{\arctan \frac{\im z}{f_{t_c}(x_c)+1+\Re (z-f_{t_c}(x_c))}}{\arctan \frac{\im z}{f_{t_c}(x_c)+\Re (z-f_{t_c}(x_c))}}.$$
Hence, with the Taylor expansion of $\arctan$, we deduce that 
\begin{align}\label{eq:vf-vf0} &\gamma_i'(t)
=\fr^{-1}+\OO(|w_t(\gamma_i(t))-w_{t_c}(x_c)|),\quad \wt\gamma_i'(t)=\fr^{-1}+\OO(|\wt w_t(\wt\gamma_i(t))-w_{t_c}(x_c)|),\\
\label{eq:vf-vf}
   & \gamma_i'(t)-\wt\gamma_i'(t)\lesssim  f_t(\gamma_i(t))-\wt f_t(\wt \gamma_i(t)) \lesssim w_t(\gamma_i(t))-\wt w_t(\wt \gamma_i(t)).&
\end{align}

We next complete the proofs of \Cref{lem:gamma0} and \Cref{l:locwalk_comp} using \eqref{eq_zw0}--\eqref{eq:vf-vf}.

\subsubsection{Proof of \Cref{lem:gamma0}}

At $t=t_c$, we have $\gamma_0(t_c)=x_c$. Then, from \eqref{eq_zw0} and \eqref{eq_zw1}, we obtain that 
$$ (t-t_c)\wtt(\gamma_0(t)) + \frac{a}{3} \wtt(\gamma_0(t))^3 + \cal E(\wtt(\gamma_0(t)))=\gamma_0(t) - c(t) = \int_{t_c}^{t}  (\gamma_0'(t')-\fr^{-1})\rd t' \lesssim \int_{t_c}^{t} |\wtt(\gamma_0(t'))|\rd t',$$
where we used the fact that $t/\fr=c(t)$ (since $x_c=t_c/\fr$) and applied \eqref{eq:vf-vf0} for the last inequality.
We can rewrite the above equation as \eqref{zf_tc3}, with $\varpi=\varpi(t)= \wtt(\gamma_0(t))$, $|\e(\varpi)|\le C_1|\varpi|^4$, $\tau=(t-t_c)/a$, and 
$$ y(\varpi):=\frac{3}{2a}\int_{t_c}^t \gamma_0'(t')-\fr^{-1}\rd t', \quad |y(\varpi)|\le C_2\int_{t_c}^{t} |\varpi(t')|\rd t',$$
for some constants $C_1,C_2>0$. Then, we have $\varpi=q_t(\varpi)$, where the function $q_t(\varpi)$ is defined in terms of $\tau$, $\e(\varpi)$, and $y(\varpi)$ as the right-hand side of \eqref{eq:w_cubic}. Thus,
$$ |q_t(\varpi)|\le 4|y(\varpi)|^{1/3} + 4|\e(\varpi)|^{1/3} + 2\tau^{1/2}\le 4C_2^{1/3}|t-t_c|^{1/3} \sup_{t'\in [t_c,t]}|\varpi(t')|^{1/3} + 4C_1^{1/3}|\varpi(t)|^{4/3}+2\sqrt{\frac{t-t_c}{a}}.$$
With this bound, we can check that there exists a constant $\cal A>0$ depending on $a, C_1, C_2$ such that if $|\varpi(t')|\le \cal A|t-t_c|^{1/2}$ for all $t'\in [t_c,t] $, then we have $|\varpi(t)|\le  \cal A|t-t_c|^{1/2}/2$.  

Combining the above fact with the continuity of $\varpi(t)$, we can conclude that $|\varpi(t)|\le  \cal A|t-t_c|^{1/2}/2$ for all $t\in [t_c,t_1]$. More precisely, we first notice that $\varpi(t)$ is H{\"o}lder-$1/3$ continuous in $t$ (since the right-hand side of \eqref{eq:w_cubic} is H{\"o}lder-$1/3$ continuous in both $t$ and $x$ and $\gamma_0(t)$ is Lipschitz continuous in $t$). Suppose $|\varpi(t')|\le \cal A|t'-t_c|^{1/2}/2\le \cal A|t-t_c|^{1/2}/2$ for all $t'\in [t_c,t]$. Then, for a sufficiently small $\varepsilon$,  $|\varpi(t')|\le \cal A|t'-t_c|^{1/2}$ for all $t'\in [t_c,t+\varepsilon]$, from which we get that $|\varpi(t')|=|q_{t'}(\varpi)|\le  \cal A|t'-t_c|^{1/2}/2$ for all $t'\in [t_c,t+\varepsilon]$. In this way, we can extend the estimate $|\varpi(t)|\le  \cal A|t-t_c|^{1/2}/2$ at $t=t_c$ all the way to $t_1$.

Finally, plugging the estimate $| \varpi(t)|\le  \cal A|t-t_c|^{1/2}/2$ into \eqref{eq_zw1}, we conclude the proof.

\subsubsection{Proof of \eqref{eq:bdd_stab}}
To bound $|\gamma_{L}(t)- \wt\gamma_{L}(t)|$, we now bound the right-hand side of \eqref{eq:vf-vf} for $i=L$.

By \Cref{lem:gammaies} and \eqref{eq:Epm0} (recall that $x_c=t_c/\fr$ and $\wt \gamma_i(0)=\gamma_i(0)$ for $i\in \llbracket -M, N\rrbracket$), we have  
\begin{equation}\label{gammai1}
    \gamma_{L}(0)= \wt \gamma_{L}(0) \sim n^{3\delta/4} \Delta t^{3/2}.
\end{equation}
For $t\in [0,t_1]$, we get from \eqref{eq_zw0}--\eqref{eq_zw2} (and using \eqref{eq:g-wtg}) that
\begin{align} 
 \gamma_{L}(t) -  t/\fr &=  F_0(\wtt(\gamma_{L}(t))) + t \wtt(\gamma_{L}(t))= (t-t_c)\wtt(\gamma_{L}(t))+ \frac{a}{3}\wtt(\gamma_{L}(t))^3+\cal E(\wtt(\gamma_{L}(t))) ,  \label{eq:gamma1}\\
 \wt \gamma_{L}(t) -  t/\fr &=F_0(\wwtt(\wt \gamma_{L}(t))) + t (\wwtt(\wt \gamma_{L}(t))+\Delta w_0({\wt u_t}(\wt \gamma_{L}(t))) ) \label{eq:gamma2}\\
&= (t-t_c)\wwtt(\wt \gamma_{L}(t))+ \frac{a}{3}\wwtt(\wt\gamma_{L}(t))^3+\cal E (\wwtt(\wt \gamma_{L}(t)))+\OO( |{\wt u_t}(\wt \gamma_{L}(t))|\Delta t) . \nonumber  
\end{align}

We first assume the following a priori bound: 
\begin{equation}\label{priori-quant2}
    |\gamma_{L}(t)- t/\fr| \le n^{\delta} \Delta t^{3/2},\quad |\wt \gamma_{L}(t)- t/\fr| \le n^{\delta} \Delta t^{3/2},\quad  \forall t\in [0, t_1].
\end{equation}
Under \eqref{priori-quant2}, using \eqref{eq:gamma1} and \eqref{eq:gamma2}, we can check that for any $t\in [0,t_1]$,
\begin{equation}\nonumber
\begin{split}
  \wtt(\gamma_{L}(t)) = \OO(n^{\delta/3}\Delta t^{1/2}),\quad  &u_t(\gamma_{L}(t)) =F_0( w(\gamma_{L}(t)) ) = \OO( n^{\delta}\Delta t^{3/2}),\\ 
\wwtt(\wt \gamma_{L}(t)) = \OO(n^{\delta/3}\Delta t^{1/2}),\quad &{\wt u_t}(\wt \gamma_{L}(t)) =F_0( \wwtt(\wt \gamma_{L}(t)) ) = \OO( n^{\delta}\Delta t^{3/2}),
\end{split}
\end{equation} 
which imply that 
\begin{equation}\label{eq:pri-quant}
\begin{split}
  &|w_t(\gamma_{L}(t))-w_{t_c}(x_c) | = |\wtt(\gamma_{L}(t))|= \OO(n^{\delta/3}\Delta t^{1/2}),\\
  &|\wt w_t(\wt \gamma_{L}(t))-w_{t_c}(x_c) | =|\wwtt(\wt \gamma_{L}(t))| +\OO(|{\wt u_t}(\wt \gamma_{L}(t))|)= \OO(n^{\delta/3}\Delta t^{1/2}).
\end{split}
\end{equation} 
Plugging these estimates into \eqref{eq:vf-vf0} for $i=L$ yields that for $t\in [0,t_1]$,
\begin{align}
& \gamma_{L}'(t)= \fr^{-1}  +  \OO(n^{\delta/3}\Delta t^{1/2}),\quad \wt \gamma_{L}'(t)= \fr^{-1}  +  \OO(n^{\delta/3}\Delta t^{1/2}). 
\end{align}
By integrating them, we obtain that
\begin{equation}\label{eq:pos-quant}
\gamma_{L}(t)-\gamma_{L}(0)- t/\fr = \OO(n^{\delta/3}\Delta t^{3/2}),\quad \wt \gamma_{L}(t)-\gamma_{L}(0)- t/\fr = \OO(n^{\delta/3}\Delta t^{3/2}),
\end{equation} 
under \eqref{priori-quant2}. Note that \eqref{gammai1} and \eqref{eq:pos-quant} together imply \eqref{priori-quant2}. Thus, to show \eqref{eq:pos-quant} without assuming \eqref{priori-quant2}, we only need to consider an $n^{-10}t_1$-net of $[0,t_1]$ and use a simple induction argument. More precisely, we define a sequence of times $\ft_k:=kn^{-10}t_1$, $k=0,1,\ldots, n^{10}$. First, the estimates \eqref{priori-quant2} and \eqref{eq:pos-quant} hold at $t=\ft_0$ by \eqref{gammai1}. Second, suppose \eqref{eq:pos-quant} holds at some $\ft_k$. With \eqref{eq_zw0}--\eqref{eq_zw2}, we can check that $\gamma_{L}(t)$ and $\wt \gamma_{L}(t)$ are H{\"o}lder-$1/3$ continuous in $t$. Thus, from \eqref{gammai1} and \eqref{eq:pos-quant} at $t=\ft_k$, we obtain that \eqref{priori-quant2} holds uniformly for all $t\in [\ft_k,\ft_{k+1}]$. The arguments above then imply that \eqref{eq:pos-quant} holds at $t=\ft_{k+1}$. With mathematical induction in $k$, we conclude \eqref{eq:pos-quant} for all $t\in [0,t_1]$.

Now, with \eqref{gammai1} and \eqref{eq:pos-quant}, we get that
\begin{equation}\label{posteri-quant3.0}
\gamma_{L}(t)- t/\fr=(1+\oo(1)) (\wt \gamma_{L}(t)- t/\fr)=(1+\oo(1))\gamma_{L}(0)\sim n^{3\delta/4}\Delta t^{3/2}.
\end{equation} 
Applying it to equations \eqref{eq:gamma1} and \eqref{eq:gamma2}, we obtain that 
\begin{equation}\label{posteri-quant3}
\begin{split}
& \wtt(\gamma_{L}(t)) = (1+\oo(1)) \wwtt(\wt \gamma_{L}(t))\sim n^{\delta/4}\Delta t^{1/2},\\
& u_t(\gamma_{L}(t)) = (1+\oo(1)) {\wt u_t}(\wt \gamma_{L}(t)) =(1+\oo(1)) \gamma_{L}(0) \sim n^{3\delta/4}\Delta t^{3/2}.
\end{split}
\end{equation} 
Subtracting the equation \eqref{eq:gamma1} from \eqref{eq:gamma2} and applying \eqref{posteri-quant3} yield that 
\begin{align}
    & | \wtt(\gamma_L(t)) - \wwtt(\wt \gamma_L(t))| \cdot|\wtt(\gamma_{L}(t))|^2  \lesssim \frac{a}{3}| \wtt(\gamma_L(t))^3 - \wwtt(\wt \gamma_L(t))^3| \le |\gamma_{L}(t)-\wt \gamma_{L}(t)| \nonumber\\
    &\quad + |t-t_c|| \wtt(\gamma_L(t)) - \wwtt(\wt \gamma_L(t))|+|\cal E (\wtt(\gamma_{L}(t)))-\cal E (\wwtt(\wt \gamma_{L}(t)))|+\OO( |{\wt u_t}(\wt \gamma_{L}(t))|\Delta t) \nonumber\\
    & \lesssim |\gamma_{L}(t)-\wt \gamma_{L}(t)| + (\Delta t+|\wtt(\gamma_{L}(t))|^3)| \wtt(\gamma_L(t)) - \wwtt(\wt \gamma_L(t))| + n^{3\delta/4}\Delta t^{5/2}.\label{eq_stability2.0}
\end{align}
Thus, we obtain that 
\begin{align}
     | \wtt(\gamma_L(t)) - \wwtt(\wt \gamma_L(t))| &\lesssim \frac{n^{3\delta/4}\Delta t^{5/2} + |\gamma_{L}(t)-\wt \gamma_{L}(t)|}{|\wtt(\gamma_{L}(t))|^2 - \OO(\Delta t+|\wtt(\gamma_{L}(t))|^3)}\nonumber\\
     &\lesssim n^{\delta/4}   \Delta t^{3/2} + n^{-\delta/2} \Delta t^{-1}|\gamma_{L}(t)-\wt \gamma_{L}(t)|,
    \label{eq_stability2}
\end{align}
which, together with \eqref{eq:vf-vf} for $i=L$, implies that
\begin{align*}
|\gamma_{L}'(t)- \wt\gamma_{L}'(t)| \lesssim  n^{\delta/4}   \Delta t^{3/2}+ n^{-\delta/2} \Delta t^{-1}|\gamma_{L}(t)-\wt \gamma_{L}(t)|.  
\end{align*}
Finally, an application of the Gr{\"o}nwall's inequality gives that
\begin{equation*} 
\max_{0\le t\le t_1}|\gamma_{L}(t)- \wt\gamma_{L}(t)| \le   n^{\delta/2} \Delta t^{5/2}.
\end{equation*}
The proof for the bound on $|\gamma_{-L}(t)- \wt\gamma_{-L}(t)| $ is similar.

\subsubsection{Proof of \eqref{eq:bdd_stab2}}
To be concise, we abuse the notations and abbreviate $u(\xi):=u_{t_1}(\xi)$, $\varpi(\xi):=\varpi_{t_1}(\xi)$ and $\wt u(\xi):=\wt u_{t_1}(\xi)$, $\wt\varpi(\xi):=\wt \varpi_{t_1}(\xi)$. 
By \eqref{eq_zw0}--\eqref{eq_zw2} and \eqref{eq:g-wtg}, for any fixed $\xi \in [ \gamma_{-L}(t_1)\wedge \wt\gamma_{-L}(t_1), \gamma_{L}(t_1)\vee \wt\gamma_{L}(t_1) ]$, $\varpi(\xi)$ and $\wt \varpi(\xi)$ satisfy the equations
\begin{align}
    &(t_1-t_c)\varpi +  \frac{a}{3} \varpi^3 + \cal E(\varpi)=\xi-t_1/\fr , \label{eq:wt1}\\
    &(t_1-t_c)\wt \varpi +  \frac{a}{3} \wt \varpi^3 + \cal E(\wt \varpi)=\xi-t_1/\fr+\OO(\Delta t |F_0(\wt \varpi)|).\label{eq:wt2}
\end{align}
By equation \eqref{posteri-quant3.0}, the chosen $\xi$ satisfies $|\xi-t_1/\fr| \lesssim n^{3\delta/4}\Delta t^{3/2}$. Then, from \eqref{eq:wt1}, \eqref{eq:wt2} and \eqref{eq_zw1}, \eqref{eq_zw2}, we obtain that 
\begin{equation}\label{eq:gamma_dist} |\varpi(\xi)|+|\wt \varpi(\xi)| \lesssim n^{\delta/4}\Delta t^{1/2},\quad |u(\xi)|+|\wt u(\xi)|\lesssim n^{3\delta/4}\Delta t^{3/2}.
\end{equation}

First, consider the case where $|\xi-t_1/\fr| > C\Delta t^{3/2}$ for a large enough constant $C>0$, so that the $\frac{a}{3} \varpi^3$ and $\frac{a}{3} \wt\varpi^3$ terms dominate in \eqref{eq:wt1} and \eqref{eq:wt2}. In particular, we can choose $C$ such that 
\begin{equation}\label{priori-stab2}
\varpi \sim \wt \varpi \sim |\xi-t_1/\fr|^{1/3},\quad \frac{a}{3}|\varpi^3-\wt \varpi^3| > 2 (t_1-t_c)|\varpi-\wt \varpi| .
\end{equation}
Subtracting the equations \eqref{eq:wt1} and \eqref{eq:wt2} and using a similar argument as in \eqref{eq_stability2.0} and \eqref{eq_stability2}, we get
\begin{equation}\label{stab2}
    |\varpi-\wt \varpi|\lesssim \frac{\Delta t |F_0(\wt \varpi)|}{|\varpi|^2-\OO(|\varpi|^3)}\lesssim \frac{ \Delta t|\wt \varpi|^3 +\Delta t^2 |\wt \varpi|}{|\varpi|^2} \lesssim n^{\delta/4}\Delta t^{3/2},
\end{equation}  
where we used \eqref{eq:gamma_dist} and \eqref{priori-stab2} in the last step. 

We next consider the case where $|\xi-t_1/\fr| \le C \Delta t^{3/2}$. From \eqref{eq:wt1} and \eqref{eq:wt2}, we get 
\begin{equation}\label{prior_w}
     |\varpi|\lesssim  \Delta t^{1/2},\quad |\wt \varpi|\lesssim  \Delta t^{1/2} .
\end{equation}  
Moreover, we can write \eqref{eq:wt1} and \eqref{eq:wt2} as
\begin{align*}
     \varpi^3 + 3\tau \varpi - 2(y+\e_1) =0,\quad \wt \varpi^3 + 3 \tau\wt \varpi - 2(y+\e_2) =0,
\end{align*}
where $\tau:=(t_1-t_c)/a$, $y:=3(\xi-t_1/\fr)/(2a)$, $\e_1=\OO(|\varpi|^4)$ and $\e_2=\OO(|\wt \varpi|^4+ \Delta t |F_0(\wt \varpi)|)$. Using the general cubic formula, we obtain that 
\begin{equation}\label{eq:cubic}
   \begin{split}
    & \varpi= \al \left[(y+\e_1)+\sqrt{(y+\e_1)^2+\tau^3}\right]^{1/3} + \al^{-1}\left[(y+\e_1)-\sqrt{(y+\e_1)^2+\tau^3}\right]^{1/3}, \\
   & \wt \varpi= \al \left[(y+\e_2)+\sqrt{(y+\e_2)^2+\tau^3}\right]^{1/3} + \al^{-1}\left[(y+\e_2)-\sqrt{(y+\e_2)^2+\tau^3}\right]^{1/3},
\end{split}  
\end{equation}
where $\al$ is a primitive cube root of unity chosen such that $\im \varpi<0$ and $\im \wt \varpi<0$. With the estimates $\tau\gtrsim \Delta t$, $|y| \lesssim  \Delta t^{3/2}$, and $\e_1, \e_2=\OO(\Delta t^2)$ by \eqref{prior_w}, it is easy to check that  
\begin{align*}
    \left|\left[(y+\e_1)\pm\sqrt{(y+\e_1)^2+\tau^3}\right]^{1/3}- \left[(y+\e_2)\pm \sqrt{(y+\e_2)^2+\tau^3}\right]^{1/3}\right|\lesssim  \frac{|\e_1|+|\e_2|}{\tau} \lesssim \Delta t,
\end{align*}
thereby giving that $|\varpi-\wt \varpi| \lesssim \Delta t$.
Combining this and \eqref{stab2}, we obtain that for $\xi \in [ \gamma_{-L}(t_1)\wedge \wt\gamma_{-L}(t_1), \gamma_{L}(t_1)\vee \wt\gamma_{L}(t_1) ]$,
\begin{equation}\label{fw_diff}
    |f_{t_1}(\xi)-\wt f_{t_1}(\xi)|  \lesssim |w_{t_1}(\xi)-\wt w_{t_1}(\xi)|   \lesssim  |\varpi-\wt\varpi| +\OO(|\wt z|)\lesssim \Delta t ,
\end{equation}
where we also used $n^{3\delta/4}\Delta t^{1/2}\ll 1$ since $0<\delta<\omega/2$. 
Since $f_{t_1}$ and $\wt f_{t_1}$ are bounded away from $0$, we have
\begin{equation}\label{density_diff}
    |\rho_{t_1}(\xi)-\wt \rho_{t_1}(\xi)|=\frac{1}{\pi} \left|\arg^* f_{t_1}(\xi)-\arg^* \wt f_{t_1}(\xi)\right| \lesssim |f_{t_1}(\xi)-\wt f_{t_1}(\xi)| \lesssim \Delta t.
\end{equation}
 On the other hand, by Item 2 of \Cref{p:rhot} (whose proof also applies to $\wt \rho_{t_1}$), we have 
\begin{equation}\label{density_t} 
\rho_{t_1}(\xi) \sim \wt\rho_{t_1}(\xi) \sim  \begin{cases}  \Delta t^{1/2}, &  | \xi-{t_1}/\fr|\le C\Delta t^{3/2} \\
|\xi-{t_1}/\fr|^{1/3}, &   C\Delta t^{3/2} \le | \xi-{t_1}/\fr| \le Cn^{3\delta/4}\Delta t^{3/2}
\end{cases}.
\end{equation}
Now, we compare the quantiles $\gamma_i({t_1})$ and $\wt\gamma_i({t_1})$ through the following equation:
\begin{equation}\label{eq_quantile}
  \int_{\gamma_{i}({t_1})}^{\gamma_{L}({t_1})} \rho_{t_1}(\xi)\rd \xi= \frac{L-i}{n}=\int_{\wt\gamma_{i}({t_1})}^{\wt\gamma_{L}({t_1})} \wt\rho_{t_1}(\xi)\rd \xi,\quad i \in \qq{-L, L}.  
\end{equation}
With \eqref{density_diff}, we can write
\begin{align*}
\int_{\wt\gamma_{i}({t_1})}^{\wt\gamma_{L}({t_1})} \wt\rho_{t_1}(\xi)\rd \xi&=\int_{\wt\gamma_{i}({t_1})}^{\wt\gamma_{L}({t_1})} \rho_{t_1}(\xi)\rd \xi+ \OO\left( \Delta t |\wt\gamma_{L}({t_1})-\wt\gamma_{i}({t_1})| \right)\\
&=\int_{\wt\gamma_{i}({t_1})}^{\gamma_{L}({t_1})} \rho_{t_1}(\xi)\rd \xi+ \OO\left(n^{3\delta/4}\Delta t^{5/2}+  n^{\delta/4}\Delta t^{1/2}  \left|\wt\gamma_{L}({t_1})- \gamma_{L}({t_1})\right| \right)\\
&=\int_{\wt \gamma_{i}({t_1})}^{\gamma_{L}({t_1})} \rho_{t_1}(\xi)\rd \xi+ \OO( n^{3\delta/4}\Delta t^{5/2}),
\end{align*}
where in the second step we used \eqref{density_t} and the estimates $|\xi-t_1/\fr| \lesssim n^{3\delta/4}\Delta t^{3/2}$, $|\wt\gamma_{L}(t_1)-\wt \gamma_i(t_1)|\lesssim n^{3\delta/4}\Delta t^{3/2}$ by \eqref{posteri-quant3.0}, and in the third step we used \eqref{eq:bdd_stab}. 
Plugging it into \eqref{eq_quantile} then gives
$$\int_{\wt \gamma_{i}({t_1}) \wedge \gamma_{i}({t_1})}^{\wt \gamma_{i}({t_1})\vee \gamma_{i}({t_1})}\rho_{t_1}(\xi)\rd \xi\lesssim n^{3\delta/4}\Delta t^{5/2}.$$
Combining this equation with \eqref{density_t}, we get that
$\Delta t^{1/2} |\wt\gamma_{i}({t_1})- \gamma_{i}({t_1})| \lesssim n^{3\delta/4}\Delta t^{5/2}$,
which concludes \eqref{eq:bdd_stab2}.

\bibliography{bibliography.bib}
\bibliographystyle{alpha}

\end{document}